\documentclass[reqno]{amsart}
\usepackage{amssymb,amsmath,hyperref}
\usepackage{amsrefs}
\usepackage[foot]{amsaddr}
\usepackage{bbold,stackrel}
 
\newtheorem{thm}{Theorem}

\newtheorem{cor}[thm]{Corollary}
\newtheorem{defi}[thm]{Definition}
\newtheorem{rem}[thm]{Remark}
\newtheorem{nota}[thm]{Notation}

\newtheorem{princ}[thm]{Principle}

\newtheorem{ack}[thm]{Acknowledgement}

\newtheorem{conj}[thm]{Conjecture}
\newtheorem*{tempo*}{Template}
\newtheorem{theorem}[thm]{Theorem}

\newtheorem{corollary}[thm]{Corollary}
\newtheorem{remark}[thm]{Remark}
\newtheorem{convention}[thm]{Convention}

\newtheorem{obs}[thm]{Observation}

\newcommand\be{\begin{equation}}
\newcommand\ee{\end{equation}} 

\usepackage{amsmath,amsfonts} 

\usepackage[applemac]{inputenc}

\def\bdefi{\begin{defi}\rm}
\def\edefi{\end{defi}}
\def\bnota{\begin{nota}\rm}
\def\enota{\end{nota}}
\def\FIVE{\Pi_{1}^{1}\text{-\textup{\textsf{CA}}}_{0}}
\def\FIVEK{\Pi_{k}^{1}\text{-\textup{\textsf{CA}}}_{0}}

\def\SIX{\Pi_{2}^{1}\text{-\textsf{\textup{CA}}}_{0}}

\def\SIXk{\Pi_{k}^{1}\text{-\textsf{\textup{CA}}}_{0}}
\def\SIXK{\Pi_{k}^{1}\text{-\textsf{\textup{CA}}}_{0}^{\omega}}

\def\ATR{\textup{\textsf{ATR}}}
\def\LOC{\textup{\textsf{LOC}}}

\def\PIT{\textup{\textsf{PIT}}}

\def\PR{\textup{\textsf{PR}}}
\def\WPR{\textup{\textsf{WPR}}}

\def\Z{\textup{\textsf{Z}}}

\def\NFP{\textup{\textsf{NFP}}}
\def\seq{\textup{\textsf{seq}}}
\def\ZFC{\textup{\textsf{ZFC}}}
\def\ZF{\textup{\textsf{ZF}}}

 \def\r{\mathbb{r}}

\def\LUB{\textup{\textsf{LUB}}}

\def\u{\textup{\textsf{u}}}
\def\c{\textup{\textsf{c}}}
\def\w{\textup{\textsf{w}}}
\def\bs{\textup{\textsf{bs}}}
\def\RCA{\textup{\textsf{RCA}}}
\def\({\textup{(}}
\def\){\textup{)}}
\def\WO{\textup{\textsf{WO}}}
\def\RCAo{\textup{\textsf{RCA}}_{0}^{\omega}}
\def\ACAo{\textup{\textsf{ACA}}_{0}^{\omega}}
\def\WKL{\textup{\textsf{WKL}}}

\def\bye{\end{document}}
\def\N{{\mathbb  N}}
\def\Q{{\mathbb  Q}}
\def\R{{\mathbb  R}}
\def\L{\textsf{\textup{L}}}

\def\MPC{\textup{\textsf{MPC}}}

\def\di{\rightarrow}
\def\asa{\leftrightarrow}
\def\ACA{\textup{\textsf{ACA}}}

\def\QFAC{\textup{\textsf{QF-AC}}}

\def\HBU{\textup{\textsf{HBU}}}
\def\LGP{\textup{\textsf{LGP}}}

\def\PRE{\textup{\textsf{PRE}}}

\def\FEJ{\textup{\textsf{FEJ}}}
\def\DIN{\textup{\textsf{DIN}}}
\def\LIN{\textup{\textsf{LIND}}}
\def\LIND{\textup{\textsf{LIND}}}

\def\UCT{\textup{\textsf{UCT}}}
\def\HBC{\textup{\textsf{HBC}}}

\def\MCT{\textup{\textsf{MCT}}}

\def\eps{\varepsilon}

\def\ECF{\textup{\textsf{ECF}}}

\def\SCF{\textup{\textsf{SCF}}}

\newcommand{\T}{\mathcal{T}}

\usepackage{graphicx}
\usepackage{tikz}
\usetikzlibrary{matrix, shapes.misc}

\numberwithin{equation}{section}
\numberwithin{thm}{section}

\usepackage{comment}

\begin{document}
\title[On Pincherle's theorem]{Pincherle's theorem in Reverse Mathematics and computability theory}
\author{Dag Normann}
\address{Department of Mathematics, The University 
of Oslo, Norway}
\email{dnormann@math.uio.no}
\author{Sam Sanders}
\address{Department of Mathematics, TU Darmstadt, Germany}
\email{sasander@me.com}

\begin{abstract}
We study the logical and computational properties of basic theorems of uncountable mathematics, in particular \emph{Pincherle's theorem}, published in 1882.  % in Reverse Mathematics and computability theory.  
This theorem states that a locally bounded function is bounded on certain domains, i.e.\ one of the first `local-to-global' principles.  
It is well-known that such principles in analysis are intimately connected to (open-cover) \emph{compactness}, but we nonetheless exhibit fundamental differences between compactness and Pincherle's theorem.    
For instance, the main question of Reverse Mathematics, namely which set existence axioms are necessary to prove Pincherle's theorem, does not have an unique or unambiguous answer, in contrast to compactness.
We establish similar differences for the \emph{computational} properties of compactness and Pincherle's theorem.  We establish the same differences for other local-to-global principles, even going back to Weierstrass. 
We also greatly sharpen the known computational power of compactness, for the most shared with Pincherle's theorem however.      
Finally, countable choice plays an important role in the previous, we therefore study this axiom together with the intimately related Lindel\"of lemma.  
\end{abstract}

%\setcounter{page}{0}
%\tableofcontents
%\thispagestyle{empty}
%\newpage
%%	
\maketitle
\thispagestyle{empty}
%-0) BABA is used to denote addition after submission to APAL.  Look for NEW as well, of course 
%0)Include Baire's normal convergence somewhere?
%1) One-sided limits
%\vspace{-0.5cm}
\section{Introduction}
\subsection{Compactness by any other name}\label{intro}
The importance of \emph{compactness} cannot be overstated, as it allows one to treat uncountable sets like the unit interval as `almost finite' while also connecting local properties to global ones. 
A famous example is \emph{Heine's theorem}, i.e.\ the \emph{local} property of continuity implies the \emph{global} property of \emph{uniform} continuity on the unit interval.  
In general, Tao writes:  
\begin{quote}
Compactness is a powerful property of spaces, and is used in many ways in many
different areas of mathematics. One is via appeal to local-to-global principles; one
establishes local control on some function or other quantity, and then uses compactness to boost the local control to global control. (\cite{taokejes}*{p.\ 168})
\end{quote}
In this light, compactness and local-to-global principles are intimately related.   
In this paper, we study the logical and computational properties of local-to-global principles with an emphasis on the \emph{fundamental differences} between the latter and compactness as studied in \cite{dagsamIII}.  % with far-reaching implications for the program \emph{Reverse Mathematics} (RM hereafter; see 
To this end, we study a typical, and historically one of the first, local-to-global principle known as \emph{Pincherle's theorem}, published around 1882 (see \cite{tepelpinch}) and formulated as follows.   
\begin{thm}[Pincherle, 1882]\label{gem}
Let $E$ be a closed and bounded subset of $\mathbb{R}^{n}$ and let $f : E \di \R$ be locally bounded. Then $f$ is bounded on $E$.
\end{thm}
At first glance, Pincherle's theorem and compactness seem \emph{intimately} related, if not the same thing.  Indeed, Pincherle himself states in \cite{werkskes}*{p.\ 341} that his Theorem~\ref{gem} `corresponds to the celebrated Heine-Borel theorem', where the latter deals with (open-cover) compactness.   
Nonethless, we shall exhibit fundamental (and surprising) differences between Pincherle's theorem and compactness.  
As discussed in Section \ref{dackground}, we study Pincherele's theorem for the most basic case, namely restricted to Cantor space, to avoid any coding (of e.g.\ real numbers), 
lest the critical reader hold the false belief coding is the cause of our results.   
In particular, simplicity is the reason we choose Pincherle's theorem, while local-to-global principles\footnote{The use of the expression `local-global' in print is analysed in \cite{chorda}, starting in 1898 with the work of Osgood.  We discuss an earlier `local-to-global' principle by Weierstrass in Remark \ref{kowlk}.} with similar properties, even going back to Weierstrass, are discussed in Remark~\ref{kowlk}.
%BABA
Moreover, Pincherle's theorem is a statement in the language third-order arithmetic and we shall always work in a framework encompassing this language. 
The associated systems $\Z_{2}^{\Omega}$ and $\SIXK$ are described in Footnote \ref{christustepaard} and introduced in Section \ref{HCT}.

\smallskip

To be absolutely clear, the aim of this paper is to study the local-to-global principle called Pincherle's theorem (for Cantor space) in Kohlenbach's higher-order \emph{Reverse Mathematics} (see Section \ref{RM}; we use `RM' for Reverse Mathematics hereafter) and computability theory, with a strong emphasis on the \emph{differences} with compactness.  
The most significant differences can be listed as follows:
\begin{enumerate}
\renewcommand{\theenumi}{\alph{enumi}}
\item In terms of the usual scale\footnote{The system $\Z_{2}^{\Omega}:=\RCAo+(\exists^{3})$ from Section \ref{HCT} proves the same second-order sentences as second-order arithmetic $\Z_{2}$. 
% The system $\Z_{2}^{\Omega}$ proves Pincherle's theorem by Theorem~\ref{mooi}, 
Similarly, $\SIXK:=\RCAo+(S_{k}^{2})$ is a higher-order version of $\SIXk$ with a third-order functional $S_{k}^{2}$ deciding $\Pi_{k}^{1}$-formulas (involving no higher-order parameters).\label{christustepaard}} of comprehension axioms,   
 Pincherle's theorem is provable from {weak K\"onig's lemma} \emph{assuming countable choice}, while \emph{without countable choice} (and for any $k$), $\SIXK$ cannot prove Pincherle's theorem; $\Z_{2}^{\Omega}$ does not include countable choice but can prove the latter. \label{ko111l}
\item Computationally speaking, it is essentially trivial to compute a finite sub-covering of a countable covering of Cantor space or the unit interval; by contrast, the upper bound in Pincherle's theorem cannot be computed by any type two functional, while the functional defined by $(\exists^{3})$ suffices. \label{ada}
\end{enumerate}
By item \eqref{ko111l}, the so-called Main Question of RM (see \cite{simpson2}*{p.\ 2}), namely which set existence axioms are necessary to prove Pincherle's theorem, does not have an unique or unambiguous answer, while the two possible answers diverge quite dramatically.  Note that in case of compactness, both for countable (\cite{simpson2}*{IV.I}) and uncountable (\cite{dagsamIII}*{\S3}) coverings, there \emph{is} a unique answer \emph{and} countable choice has no influence.    
Moreover, since compactness and local-to-global principles coincide in second-order arithmetic, the results in this paper highlight another major difference between second- and higher-order mathematics.   
Finally, items \eqref{ko111l} and \eqref{ada} establish the `schizophrenic' nature of Pincherle's theorem: the first item suggests a certain level of constructivity\footnote{By Corollary \ref{ofmoreinterest}, Pincherle's theorem (for Cantor space) is provable in
 %$\RCAo+\WKL+\QFAC^{0,1}$; the latter system is included in 
 the \emph{classical} system of proof mining from \cite{kohlenbach3}*{Theorem~10.47}, which enjoys rather general term extraction properties.  The use of countable choice \emph{generally} means that extracting algorithms from proofs is not possible.\label{PCbrigade}}, while the second item completely denies such nature.  

\smallskip

In light of the previous, it is clear that countable choice plays an important role in this paper.   
In Section \ref{not2bad} we therefore study countable choice and the intimately\footnote{By \cite{heerlijk}*{\S3.1}, countable choice and the Lindel\"of lemma for $\R$ are equivalent over $\ZF$.} related \emph{Lindel\"of lemma}.  
In particular, we show that the {Lindel\"of lemma} is highly dependent on its formulation, namely provable from $(\exists^{3})$ versus unprovable in $\ZF$ set theory (see Section \ref{finne}).    
We also show that the Lindel\"of lemma for Baire space yields $\FIVE$ when combined with $\ACAo$ (Section \ref{kurzweil}).

\smallskip

As to the rest of this section, we discuss some background on Pincherle's theorem in Section \ref{dackground} and formulate the particular questions we will answer in this paper.  
We provide the formal definition of Pincherle's theorem (on Cantor space to avoid coding) in Section~\ref{hung}, as well as the related Heine's theorem; based on these, we can formulate the exact results to be obtained in this paper. 

\smallskip

To provide some context, \emph{general non-monotone induction} provides a natural upper bound for most of the classical theorems under investigation in this paper, both from the point of view of computability theory and RM.
A natural lower bound can be found in \emph{arithmetical transfinite recursion}.  
This observation is based on the results in \cite{dagcie18} and discussed in more detail in Remark \ref{draak}.

\smallskip

Finally, the place of this paper in a broader context is discussed in~Remark~\ref{X}. 
In particular, our paper \cite{dagsamVII} is a `successor' to the paper at hand, where the aim is to identify (many) theorems pertaining to open sets that exhibit logical and computational behaviour similar to Pincherle's theorem
as in items (a)-(b) above.

\subsection{Questions concerning Pincherele's theorem}\label{dackground}
In this section, we sketch the history and background pertaining to Pincherle's theorem, as well as the kind of technical questions we intend to answer below.  

\smallskip

First of all, as to its history, Theorem~\ref{gem} was established by \emph{Salvatore Pincherle} in 1882 in \cite{tepelpinch}*{p.\ 67} in a more verbose formulation. 
Indeed, Pincherle did not use the notion of \emph{local boundedness}, and a function is nowadays called \emph{locally bounded on $E$} if every $x\in E$ has a neighbourhood $U\subset E$ on which the function is bounded.  
In this theorem, Pincherle assumed the existence of $L, r:E\di \R^{+}$ such that for any $x\in E$ the function is bounded by $L(x)$ on the ball $B(x, r(x))\subset E$ (\cite{tepelpinch}*{p.\ 66-67}).
We refer to these functions $L, r:E\di \R^{+}$ as \emph{realisers} for local boundedness.  We do not restrict the notion of realiser to any of its established technical definitions.
We note that Pincherle's theorem is a typical `local-global' principle.

\smallskip

Secondly, as to its conceptual nature, Pincherle's theorem may be found as \cite{gormon}*{Theorem~4} in a \emph{Monthly} paper aiming to provide conceptually easy proofs of well-known theorems.  
Furthermore, Pincherle's theorem is the \emph{sample theorem} in \cite{thom2}, a recent monograph dealing with \emph{elementary real analysis}.  
Thus, Pincherle's theorem qualifies as `basic' mathematics in any reasonable sense of the word, and is also definitely within the scope of RM as it essentially predates set theory (\cite{simpson2}*{I.1}).  % (see Section \ref{prelim} for RM).  

\smallskip

Thirdly, despite the aforementioned `basic nature' of Pincherle's theorem, its proofs in \cites{gormon, thom2, tepelpinch,bartle2} actually provide `highly uniform' information: as shown in Section~\ref{forgopppp}, these proofs establish Pincherle's theorem \emph{and that the bound in the consequent only depends on the realisers $r, L:E\di \R^{+}$ for local boundedness}; in the case of \cite{tepelpinch} we need a minor modification of the proof.
In general, we shall call a theorem \emph{uniform} if the objects claimed to exist depend only on few of the theorem's parameters.  
Historically, Dini, Pincherle, Bolzano, Young, Hardy, Riesz, and Lebesgue (the first three after minor modification) have proved uniform versions of e.g.\ Heine's theorem, as discussed in Section \ref{heikel}.
More recently, uniform theorems have been obtained as part of the development of analysis based on techniques from the \emph{gauge integral}, a generalisation of Lebesgue's integral.  
We have collected a number of such uniform results in Section \ref{pproof} as they are of independent interest.

\smallskip

Fourth, as discussed in detail in Sections \ref{sum}, one of our aims is the study of the uniform version of Pincherle's theorem in which the bound in the consequent only depends on the realisers $r, L:E\di \R^{+}$.  
As it turns out, both the original and uniform versions of Pincherle's have noteworthy properties from the point of view of RM and computability theory.  
In particular, we provide answers to \eqref{kata}-\eqref{turk}, where `computable' refers to Kleene's S1-S9, as discussed in Section \ref{HCT}.
\begin{enumerate}
\renewcommand{\theenumi}{Q\arabic{enumi}}
\item How hard is it to compute the upper bound in Pincherle's theorem in terms of (some of) the data?\label{kata}
\item What is the computational strength of the ability to obtain the upper bounds from Pincherle's theorem?\label{greek}
\item What (higher-order) comprehension axioms prove (uniform and original) Pincherele's theorem?\label{turk}
\end{enumerate}
The answers provided below include some surprises as sketched in items \eqref{ko111l}-\eqref{ada} in Section \ref{intro}: while Pincherle's theorem is provable \emph{without countable choice}, the latter axiom has a major impact on the answer to \eqref{turk}. 
Indeed, Pincherle's theorem is provable from weak K\"onig's lemma \emph{given countable choice}; in absence of the latter, Pincherle's theorem is not provable in $\SIXK$ (for any $k$) though provable in $\Z_{2}^{\Omega}$, where the latter does not involve countable choice.  
Moreover, Pincherle's theorem is actually \emph{equivalent} to the Heine-Borel theorem for countable coverings (given countable choice), but the finite sub-covering in the latter is (trivially) computable, while the upper bound in the former cannot be computed by any type two functional, while the functional from $(\exists^{3})$ suffices. 

\smallskip

Fifth, we wish to stress that (original and uniform) Pincherle's theorem is naturally a statement of \emph{third-order} arithmetic.  Similarly, $\SIXK$ from Section \ref{HCT} is the higher-order version 
of $\SIXk$ involving the \emph{third-order comprehension functional}  $S_{k}^{2}$ that can decide $\Pi_{k}^{1}$-formulas (only involving first- and second-order parameters).  In other words, 
all the aforementioned is naturally at home in third-order arithmetic.  In this light, it is a natural question whether $\SIXK$ can prove Pincherle's theorem or whether $S_{k}^{2}$ can compute the upper bound in this theorem from the other data.   
As noted above, the answer is negative while the \emph{fourth-order} comprehension axiom $(\exists^{3})$ does suffice (for both proof and computation).   
As discussed in Remark \ref{NFP}, there are alternative scales that are more fine-grained, compared to the standard one based on comprehension and related fragments of $\Z_{2}$, and capture e.g.\ compactness for uncountable coverings quite well. 

\smallskip

Sixth, while Pincherle's theorem constitutes an illustrative example, it is by no means unique: we analogously study \emph{Heine's theorem} (see Section \ref{heineke}) on uniform continuity and sketch the (highly similar) approach for \emph{Fej\'er's theorem}.  
A number of similar theorems will be studied in a follow-up paper (see Remark~\ref{flurki}).  We discuss variations and generalisations of Pincherle's theorem in Remark~\ref{kowlk} and Section \ref{pitche}, the latter based on \emph{subcontinuity}, a natural weakening of continuity from the literature equivalent to local boundedness. 
% in the same way.  

\smallskip

Finally, Section \ref{not2bad} is devoted to the detailed study of certain, in our opinion, subtle aspects of the results obtained above and in \cite{dagsamIII}.  
Firstly, in light of our use of the axiom of (countable) choice in our RM-results, Section~\ref{finne} is devoted to the study of quantifier-free countable choice, its tight connection to the  \emph{Lindel\"of lemma} in particular.
In turn, we show in Section \ref{kurzweil} that the Lindel\"of lemma \emph{for Baire space}, together with $(\exists^{2})$, proves $\FIVE$, improving the results in \cite{dagsamIII}*{\S4}.  
We also show that the status of the {Lindel\"of lemma} is highly dependent on its formulation, namely provable from $(\exists^{3})$ versus unprovable in $\ZF$.    

\subsection{Answers regarding Pincherle's and Heine's theorem}\label{hung}
We list the formal definition of the Pincherle and Heine theorems, as well as answers to (Q1)-(Q3). 
 
\subsubsection{Pincherle's theorem and uniformity}\label{sum}
We formally introduce Pincherle's theorem and the aforementioned `highly uniform' version, and discuss the associated results, to be established in Sections \ref{PRS} and \ref{prm}.
Remark \ref{kiekenkkk} at the end of this section provides some historical context, lest there be any confusion there.  

\smallskip

First of all, to reduce technical details to a minimum, we mostly work with Cantor space, denoted $2^{\N}$ or $C$, rather than the unit interval; the former is homeomorphic to a closed subset of the latter anyway. 
The advantage is that we do not need to deal with the coding of real numbers using Cauchy sequences, which can get messy.  

\smallskip

Secondly, in keeping with Pincherle's use of $L, r:\R\di \R^{+}$, we say that $G:C\di \N$ is a \emph{realiser} for the \emph{local boundedness} of the functional $F:C\di \N$ if 
\[
\LOC(F, G)\equiv (\forall f , g\in C)\big[ g\in [\overline{f}G(f)] \di F(g)\leq G(f)    \big].
\]
Note that $\overline{f}n= \langle f(0), f(1), \dots, f(n-1)\rangle $ for $n\in \N$, while $g\in [\overline{f}n]$ means that $g(m)=f(m)$ for $m< n$.  
Hence, $\LOC(F, G)$ expresses that $G$ provides, for every $f\in C$, a \emph{neighbourhood} $[\overline{f}G(f)]$ in $C$ in which $F$ is bounded by $G(f)$.  
We make use of \emph{one} functional $G$ for \emph{both} the neighbourhood and upper bound, while Pincherle uses \emph{two} separate 
functions $L$ (for the upper bound) and $r$ (for the neighbourhood); as discussed in Remark \ref{nodiff}, this makes no difference.  % for any of the below results.  

\smallskip

Thirdly, the following are the \emph{original} and \emph{uniform} versions of Pincherle's theorem for Cantor space, respectively $\PIT_{o}$ and $\PIT_{\u}$.
As discussed in Section \ref{forgopppp}, Pincherle's proof from \cite{tepelpinch} (with minor modification only) yields $\PIT_{\u}$; the same holds for \cite{thom2, gormon, bartle2} without any changes to the proofs.  
%\bdefi[$\PIT_{o}$]
\be\tag{$\PIT_{o}$}
(\forall F, G:C\di \N)(\exists N\in \N)\big[  \LOC(F, G)\di (\forall g \in C)(F(g)\leq N)\big]
\ee
\vspace{-0.5cm}
\be\tag{$\PIT_{\u}$}
(\forall G:C\di \N)(\exists N\in \N)(\forall F:C\di \N)\big[  \LOC(F, G)\di (\forall g \in C)(F(g)\leq N)\big]
\ee
%\edefi
The difference in quantifier position is quite important as will become clear from our  answers (A1)-(A3) below to the questions (Q1)-(Q3) from Section \ref{intro}.

\smallskip

Fourth, it is a natural question how hard it is to compute an upper bound as in Pincherle's theorem from (some of) the data.  To this end, we consider the specification for a (non-unique) functional $M:(C\di \N)\di \N$ as follows.  
\be\tag{$\PR(M)$}
(\forall F, G:C\di \N)\big[  \LOC(F, G)\di (\forall g \in C)(F(g)\leq M(G))\big].
\ee
Any $M$ satisfying $\PR(M)$ is called a \emph{realiser}\footnote{We use the term \emph {realiser} in a quite liberal way. In fact, Pincherle realisers are witnesses to the truth of \emph{uniform} Pincherle's theorem by selecting, to each $G$, an upper bound as in $\PIT_{\u}$.  However, the set of upper bounds, seen as a function of $G$, is highly complex: the PR that selects the \emph{least} bound is computationally equivalent to $\exists^3$ from Section \ref{HCT}, which is left as an exercise.} for Pincherle's theorem $\PIT_{\u}$, or a \emph{Pincherle realiser} (PR) for short.  
A \emph{weak} Pincherle realiser additionally has the function $F$ as input, as discussed in Section \ref{introwpr}.  
In conclusion, we shall provide the following answers to the questions \eqref{kata}-\eqref{turk} from Section \ref{intro}. 
\begin{enumerate}
\renewcommand{\theenumi}{\roman{enumi}}\renewcommand{\theenumi}{A\arabic{enumi}}
\item Pincherle realisers cannot be computed (Kleene's S1-S9) from any type two functional, but some may be computed from $\exists^3$ from Section~\ref{HCT}.
\item Pincherle realisers compute realisers of $\Pi^1_1$-separation for subsets of $\N^\N$ and natural generalisations to sets of objects of type two.
\item Compactness of $2^{\N}$ for uncountable coverings is equivalent to $\PIT_{\u}$ given countable choice; $\PIT_{o}$ is equivalent to {weak K\"onig's lemma} \emph{given countable choice}. Without the latter, $\SIXK$ cannot prove $\PIT_{o}$, while $\Z_{2}^{\Omega}$ can.  
\end{enumerate}
We note the huge difference in logical hardness between the uniform and original versions of Pincherle's theorem, and the important role of countable choice.  
Nonetheless, both $\PIT_{o}$ and $\PIT_{\u}$ are provable \emph{without} this axiom by Theorem \ref{mooi}.

\smallskip

Finally, we consider the following remark on the history of the function concept.    
\begin{rem}[A function by any other name]\label{kiekenkkk}\rm
We show that Pincherle intended to formulate his theorem for \emph{any} function, not just continuous ones.  
First of all, Pincherle includes the following expression in his theorem: 
\begin{quote}
\emph{Funzione di $x$ nel senso pi\`u generale della par\'ola} (\cite{tepelpinch}*{p.\ 67}),
\end{quote}
which translates to `function of $x$ in the most general sense'.  However, discontinuous functions had already enjoyed a long history by 1882: they were discussed by Dirichlet in 1829 (\cite{didi1}); Riemann studied such functions in his 1854 \emph{Habilitationsschrift} (\cite{kleine}*{p.\ 115}), 
and the 1870 dissertation of Hankel, a student of Riemann, has `discontinuous functions' in its title (\cite{hankelijkheid}).  
We also mention \emph{Thomae's function}, similar to Dirichlet's function and introduced in \cite{thomeke}*{p.\ 14} around 1875.     

\smallskip

Secondly, Pincherle refers to a number of theorems due to Dini and Weierstrass as \emph{special cases} of his theorem in \cite{tepelpinch}*{p.\ 66-68}.
He also mentions that Dini's theorem is about continuous functions, i.e.\ it seems unlikely he just implicitly assumed his theorem to be about continuous functions.  
Finally, the proof on \cite{tepelpinch}*{p.\ 67} does not require the function to be continuous (nor does it mention the latter word).  
Since Pincherle explicitly mentions establishing \emph{una proposizione generale}, it seems unlikely he overlooked the fact that his \emph{Teorema} was about \emph{arbitrary} functions.  

\smallskip

Finally, the (modern) concept of arbitrary/general function is generally credited to Lobachevsky (\cite{loba}) and Dirichlet (\cite{didi3}) in 1834-1837.  Fourier's earlier work (\cite{fourierlachaud}) was instrumental in that he (for the first time) made a clear distinction between a function on one hand and its analytic representation on the other hand.
\end{rem}

\subsubsection{Heine's theorem and uniformity}\label{heineke}
We formally introduce \emph{Heine's theorem} and the associated `uniform' version, and discuss the associated results, to be established in Sections \ref{PRS} and \ref{prm}. 
As in the previous section, we work over $2^{\N}$.  % to avoid coding real numbers.  

\smallskip

First of all, Heine's theorem is the statement that \emph{a continuous $f:X\di \R$ on a compact space $X$ is uniformly continuous}.   
Dini's proof (\cite{dinipi}*{\S41}) of Heine's theorem makes use of a \emph{modulus of continuity}, i.e.\ a functional computing $\delta$ from $\eps>0$ and $ x\in X$ in the usual $\eps$-$\delta$-definition of continuity.  
As discussed in \cite{ruskesnokken}, Bolzano's definition of continuity involves a modulus of continuity, while his (apparently faulty) proof of Heine's theorem may be found in \cite{nogrusser}*{p.\ 575}.    
The following formula expresses that $G$ is a modulus of (pointwise) continuity for $F$ on $C$:
\be\label{XYX}\tag{$\MPC(F, G)$}
(\forall f, g\in C)(\overline{f}G(f)=\overline{g}G(f)\di F(f)=F(g)).
\ee
Secondly, we introduce $\UCT_{\u}$, the \emph{uniform} Heine's theorem for $C$.  
By Section \ref{heikel}, the proofs by Dini, Bolzano, Young, Hardy, Riesz, Thomae, and Lebesgue (\cite{dinipi,lebes1,nogrusser,younger, manon,hardy, thomeke}) establish the uniform $\UCT_{\u}$ for $[0,1]$ (with minor modification for \cite{dinipi,nogrusser, thomeke}); the same for \cite{langebaard, thom2,knapgedaan ,gormon, bartle2, hobbelig, protput ,stillebron, botsko, lebes1, bromance} without changes.  
\bdefi[$\UCT_{\u}$]
\[
(\forall G^{2})(\exists m^{0})(\forall F^{2})\big[\MPC(F, G) \di (\forall  f,g \in C)(\overline{f}m=\overline{g}m\di F(f)=F(g))  ].
\]
\edefi
The difference in quantifier position has big consequences: Heine's theorem follows from weak K\"onig's lemma by \cite{kohlenbach4}*{Prop.\ 4.10}, 
while $\UCT_{\u}$ is not provable in $\SIXK$ for any $k$.   Indeed, we prove in Section~\ref{myheinie} that $\UCT_{\u}$ is equivalent to the Heine-Borel theorem for \emph{uncountable} coverings, and hence to $\PIT_{\u}$.  

\smallskip

The computability-theoretic differences between the uniform and original versions of Heine's theorem are as follows: on one hand, assuming $\MPC(F, G)$, one computes\footnote{If $\MPC(F,G)$, one computes an associate for $F:C\di \N$ from $F$ and $\exists^2$, and one then computes an upper bound for $F$ on $C$, as the \emph{fan functional} has a computable code (\cite{noortje}*{p.\ 102}). } the upper bound from (original) Heine's theorem in terms of $F$ and $\exists^{2}$ from Section \ref{HCT}.  
% i.e.\ the third Big Five system suffices.
On the other hand, given $G$, the class of $F$ such that $\MPC(F,G)$ is equicontinuous (and finite if $F$ is restricted to $C$), but computing a \emph{modulus} of equicontinuity from $G$ is as hard as computing a PR from $G$.
In this light, the (original) Heine theorem is simpler than $\PIT_{o}$ in computability theory, while the uniform versions are equivalent both in RM and computability theory.  

\smallskip

Clearly, many theorems from the RM of $\WKL_{0}$ can be studied in the same way as Pincherle's and Heine's theorems; 
we provide one such example, namely Fej\'er's theorem in Section \ref{myheinie}.  % while a systematic study is reserved for a follow-up paper.  
Speaking of the future, the following remark discusses the place of this paper in the context of a broader research project.  
\begin{rem}[Our project]\label{X}\rm
The paper at hand is part of a series of papers \cite{dagsam, dagsamII, dagsamIII, dagsamVI, dagsamVII} that communicate the results of our joint project on the logical and computational properties of the uncountable.   As is expected, this project has some single-author spin-off papers \cite{dagcie18, samcie19, samwollic19,samsplit, sahotop,samph} as well.  Our motivating research question is based on Shore's \cite{shorecomp}*{Problem 5.1}, namely: \emph{how hard is it to compute \(S1-S9\) the objects claimed to exist by classical theorems of analysis?}  In the spirit of RM, we also study the question which set existence axioms can prove such theorems.   

\smallskip

As a first step, we studied the computational properties of \emph{compactness} for uncountable coverings in \cites{dagsam, dagsamII}.  We showed that the sub-coverings in the Heine-Borel and Vitali covering theorems cannot be computed by any type two functional.  
Nonetheless, the computational power of the former is much greater than that of the latter, even though these theorems are very closely related in the case of countable coverings.  

\smallskip

As a next step, we studied the RM of various covering theorems in \cite{dagsamIII}, with special focus on the \emph{Cousin and Lindel\"of lemmas}.  In terms of the usual scale of comprehension axioms, these theorems are not provable in $\SIXK$ for any $k$ but provable in $\Z_{2}^{\Omega}$ plus potentially countable choice.    
We also showed that the Cousin lemma is equivalent to various properties of the \emph{gauge integral} (see \cite{mullingitover}), while also studying computational properties of the aforementioned lemmas.  

\smallskip

Finally, we studied in \cite{dagsamVI} the RM and computability theory of measure theory, starting with \emph{weak compactness}, i.e.\ the combinatorial essence of the \emph{Vitali covering theorem}. 
We show that weak compactness is equally hard to prove as (full) compactness, in terms of the usual scale of comprehension axioms.  
Despite this hardness, the former is shown to have much more computational and logical strength than the latter.  
We also exhibit striking differences between our `higher-order' measure theory and `measure theory via codes' as in second-order RM (\cite{simpson2}*{X.1}). 

\smallskip

In conclusion, this paper can be viewed as a continuation of the study of compactness in \cite{dagsamIII}.  However, as discussed in Section \ref{intro},
there are fundamental differences between compactness on one hand, and local-to-global principles like Pincherle's theorem, both from the computability and RM point of view. 
In turn, \cite{dagsamVII} can be viewed as a `successor' to the paper at hand, where the aim is to identify (many) theorems pertaining to open sets that exhibit logical and computational behaviour similar to Pincherle's theorem as described in items (A1)-(A3) from Section \ref{sum}.  In fact, any theorem that behaves as in the aforementioned items is said to \emph{exhibit the Pincherle phenomenon} in \cite{dagsamVII}.  Note that open sets in the latter are studied via characteristic functions that have 
a $\Sigma_{1}^{0}$-definition (with higher-order parameters) similar to the `usual' definition of open sets in RM (see \cite{simpson2}*{II.4}).
\end{rem}

\section{Preliminaries}\label{prelim}
We sketch the program \emph{Reverse Mathematics}, as well as its generalisation to \emph{higher-order arithmetic} in Section \ref{KOH}.  
As our main results will be proved using techniques from \emph{computability theory}, we discuss the latter in Section~\ref{HCT}.
\subsection{Introducing Reverse Mathematics}\label{RM}
Reverse Mathematics (RM) is a program in the foundations of mathematics initiated around 1975 by Friedman (\cites{fried,fried2}) and developed extensively by Simpson (\cite{simpson2}) and others.  
We refer to \cite{stillebron} for a basic introduction to RM and to \cite{simpson2} for an overview of RM; we now sketch some of the aspects of RM essential to this paper.  

\smallskip
  
The aim of RM is to find the axioms necessary to prove a statement of \emph{ordinary}, i.e.\ \emph{non-set theoretical} mathematics.   
The classical base theory $\RCA_{0}$ of `computable mathematics' is always assumed.  
Thus, the aim of RM is:  
\begin{quote}
\emph{The aim of \emph{RM} is to find the minimal axioms $A$ such that $\RCA_{0}$ proves $ [A\di T]$ for statements $T$ of ordinary mathematics.}
\end{quote}
Surprisingly, once the minimal axioms $A$ have been found, we almost always also have $\RCA_{0}\vdash [A\asa T]$, i.e.\ not only can we derive the theorem $T$ from the axioms $A$ (the `usual' way of doing mathematics), we can also derive the axiom $A$ from the theorem $T$ (the `reverse' way of doing mathematics).  In light of these `reversals', the field was baptised `Reverse Mathematics'.  

\smallskip

Perhaps even more surprisingly, in the majority of cases, for a statement $T$ of ordinary mathematics, either $T$ is provable in $\RCA_{0}$, or the latter proves $T\asa A_{i}$, where $A_{i}$ is one of the logical systems $\WKL_{0}, \ACA_{0},$ $ \ATR_{0}$ or $\FIVE$ from \cite{simpson2}*{I}.  The latter four systems together with $\RCA_{0}$ form the `Big Five' and the aforementioned observation that most mathematical theorems fall into one of the Big Five categories, is called the \emph{Big Five phenomenon} (\cite{montahue}*{p.~432}).  

\smallskip

Furthermore, each of the Big Five has a natural formulation in terms of (Turing) computability (see \cite{simpson2}*{I}), and each of the Big Five also corresponds (sometimes loosely) to a foundational program in mathematics (\cite{simpson2}*{I.12}).  
The Big Five systems of RM also satisfy a linear order, as follows:
\be\label{linord}
\FIVE\di \ATR_{0}\di \ACA_{0}\di\WKL_{0}\di \RCA_{0}.
\ee
By contrast, there are many incomparable \emph{logical} statements in second-order arithmetic.  For instance, a regular plethora of such statements may be found in the \emph{Reverse Mathematics zoo} in \cite{damirzoo}.  The latter is intended as a collection of (somewhat natural) theorems outside of the Big Five classification of RM.  It is also worth noting that the Big Five only constitute a \emph{very tiny fragment} of $\Z_{2}$; on a related note, the RM of topology does give rise to theorems equivalent to $\SIX$ (\cite{mummy}), but that is the current upper bound of RM to the best of our knowledge.  % In particular, if $\SIXk$ is $
Moreover, the coding of topologies is not without problems, as discussed in \cite{hunterphd}. 
\subsection{Higher-order Reverse Mathematics}\label{KOH}
We sketch Kohlenbach's \emph{higher-order Reverse Mathematics} as introduced in \cite{kohlenbach2}.  In contrast to `classical' RM, higher-order RM makes use of the much richer language of \emph{higher-order arithmetic}.  

\smallskip

As suggested by its name, {higher-order arithmetic} extends second-order arithmetic.  Indeed, while the latter is restricted to numbers and sets of numbers, higher-order arithmetic also has sets of sets of numbers, sets of sets of sets of numbers, et cetera.  
To formalise this idea, we introduce the collection of \emph{all finite types} $\mathbf{T}$, defined by the two clauses:
\begin{center}
(i) $0\in \mathbf{T}$   and   (ii)  If $\sigma, \tau\in \mathbf{T}$ then $( \sigma \di \tau) \in \mathbf{T}$,
\end{center}
where $0$ is the type of natural numbers, and $\sigma\di \tau$ is the type of mappings from objects of type $\sigma$ to objects of type $\tau$.
In this way, $1\equiv 0\di 0$ is the type of functions from numbers to numbers, and where  $n+1\equiv n\di 0$.  Viewing sets as given by characteristic functions, we note that $\Z_{2}$ only includes objects of type $0$ and $1$.    

\smallskip

The language of $\L_{\omega}$ consists of variables $x^{\rho}, y^{\rho}, z^{\rho},\dots$ of any finite type $\rho\in \mathbf{T}$.  Types may be omitted when they can be inferred from context.  
The constants of $\L_{\omega}$ include the type $0$ objects $0, 1$ and $ <_{0}, +_{0}, \times_{0},=_{0}$  which are intended to have their usual meaning as operations on $\N$.
Equality at higher types is defined in terms of `$=_{0}$' as follows: for any objects $x^{\tau}, y^{\tau}$, we have
\be\label{aparth}
[x=_{\tau}y] \equiv (\forall z_{1}^{\tau_{1}}\dots z_{k}^{\tau_{k}})[xz_{1}\dots z_{k}=_{0}yz_{1}\dots z_{k}],
\ee
if the type $\tau$ is composed as $\tau\equiv(\tau_{1}\di \dots\di \tau_{k}\di 0)$.  
Furthermore, $\L_{\omega}$ also includes the \emph{recursor constant} $\mathbf{R}_{\sigma}$ for any $\sigma\in \mathbf{T}$, which allows for iteration on type $\sigma$-objects as in the special case \eqref{special}.  
Formulas and terms are defined as usual.  
\bdefi The base theory $\RCAo$ consists of the following axioms:
\begin{enumerate}
\item  Basic axioms expressing that $0, 1, <_{0}, +_{0}, \times_{0}$ form an ordered semi-ring with equality $=_{0}$.
\item Basic axioms defining the well-known $\Pi$ and $\Sigma$ combinators (aka $K$ and $S$ in \cite{avi2}), which allow for the definition of \emph{$\lambda$-abstraction}. 
\item The defining axiom of the recursor constant $\mathbf{R}_{0}$: for $m^{0}$ and $f^{1}$: 
\be\label{special}
\mathbf{R}_{0}(f, m, 0):= m \textup{ and } \mathbf{R}_{0}(f, m, n+1):= f( n,\mathbf{R}_{0}(f, m, n)).
\ee
\item The \emph{axiom of extensionality}: for all $\rho, \tau\in \mathbf{T}$, we have:
\be\label{EXT}\tag{$\textsf{\textup{E}}_{\rho, \tau}$}  
(\forall  x^{\rho},y^{\rho}, \varphi^{\rho\di \tau}) \big[x=_{\rho} y \di \varphi(x)=_{\tau}\varphi(y)   \big].
\ee 
\item The induction axiom for quantifier-free\footnote{To be absolutely clear, variables (of any finite type) are allowed in quantifier-free formulas of the language $\L_{\omega}$: only quantifiers are banned.} formulas of $\L_{\omega}$.
\item $\QFAC^{1,0}$: The quantifier-free axiom of choice as in Definition \ref{QFAC}.
\end{enumerate}
\edefi
\bdefi\label{QFAC} The axiom $\QFAC$ consists of the following for all $\sigma, \tau \in \textbf{T}$:
\be\tag{$\QFAC^{\sigma,\tau}$}
(\forall x^{\sigma})(\exists y^{\tau})A(x, y)\di (\exists Y^{\sigma\di \tau})(\forall x^{\sigma})A(x, Y(x)),
\ee
for any quantifier-free formula $A$ in the language of $\L_{\omega}$.
\edefi
As discussed in \cite{kohlenbach2}*{\S2}, $\RCAo$ and $\RCA_{0}$ prove the same sentences `up to language' as the latter is set-based and the former function-based.  

\smallskip

Recursion as in \eqref{special} is called \emph{primitive recursion}; the class of functionals obtained from $\mathbf{R}_{\rho}$ for all $\rho \in \mathbf{T}$ is called \emph{G\"odel's system $T$} of all (higher-order) primitive recursive functionals.  

\smallskip

We use the usual notations for natural, rational, and real numbers, and the associated functions, as introduced in \cite{kohlenbach2}*{p.\ 288-289}.  
\begin{defi}[Real numbers and related notions in $\RCAo$]\label{keepintireal}\rm~
\begin{enumerate}
\item Natural numbers correspond to type zero objects, and we use `$n^{0}$' and `$n\in \N$' interchangeably.  Rational numbers are defined as signed quotients of natural numbers, and `$q\in \Q$' and `$<_{\Q}$' have their usual meaning.    
\item Real numbers are coded by fast-converging Cauchy sequences $q_{(\cdot)}:\N\di \Q$, i.e.\  such that $(\forall n^{0}, i^{0})(|q_{n}-q_{n+i})|<_{\Q} \frac{1}{2^{n}})$.  
We use Kohlenbach's `hat function' from \cite{kohlenbach2}*{p.\ 289} to guarantee that every $q^{1}$ defines a real number.  
\item We write `$x\in \R$' to express that $x^{1}:=(q^{1}_{(\cdot)})$ represents a real as in the previous item and write $[x](k):=q_{k}$ for the $k$-th approximation of $x$.    
\item Two reals $x, y$ represented by $q_{(\cdot)}$ and $r_{(\cdot)}$ are \emph{equal}, denoted $x=_{\R}y$, if $(\forall n^{0})(|q_{n}-r_{n}|\leq {2^{-n+1}})$. Inequality `$<_{\R}$' is defined similarly.         
\item Functions $F:\R\di \R$ mapping reals to reals are represented by $\Phi^{1\di 1}$ mapping equal reals to equal reals, i.e. 
\be\tag{$\textsf{\textup{RE}}$}\label{RE}
(\forall x , y\in \R)(x=_{\R}y\di \Phi(x)=_{\R}\Phi(y)).
\ee
\item The relation `$x\leq_{\tau}y$' is defined as in \eqref{aparth} but with `$\leq_{0}$' instead of `$=_{0}$'.  Binary sequences are denoted `$f^{1}, g^{1}\leq_{1}1$', but also `$f,g\in C$' or `$f, g\in 2^{\N}$'.  
\item Sets of type $\rho$ objects $X^{\rho\di 0}, Y^{\rho\di 0}, \dots$ are given by their characteristic functions $f^{\rho\di 0}_{X}$, i.e.\ $(\forall x^{\rho})[x\in X\asa f_{X}(x)=_{0}1]$, where $f_{X}^{\rho\di 0}\leq_{\rho\di 0}1$.  
\end{enumerate}
\end{defi}
We sometimes omit the subscript `$\R$' if it is clear from context.  
Finally, we introduce some notation to handle finite sequences nicely.  
\begin{nota}[Finite sequences]\label{skim}\rm
We assume a dedicated type for `finite sequences of objects of type $\rho$', namely $\rho^{*}$.  Since the usual coding of pairs of numbers goes through in $\RCAo$, we shall not always distinguish between $0$ and $0^{*}$. 
Similarly, we do not always distinguish between `$s^{\rho}$' and `$\langle s^{\rho}\rangle$', where the former is `the object $s$ of type $\rho$', and the latter is `the sequence of type $\rho^{*}$ with only element $s^{\rho}$'.  The empty sequence for the type $\rho^{*}$ is denoted by `$\langle \rangle_{\rho}$', usually with the typing omitted.  

\smallskip

Furthermore, we denote by `$|s|=n$' the length of the finite sequence $s^{\rho^{*}}=\langle s_{0}^{\rho},s_{1}^{\rho},\dots,s_{n-1}^{\rho}\rangle$, where $|\langle\rangle|=0$, i.e.\ the empty sequence has length zero.
For sequences $s^{\rho^{*}}, t^{\rho^{*}}$, we denote by `$s*t$' the concatenation of $s$ and $t$, i.e.\ $(s*t)(i)=s(i)$ for $i<|s|$ and $(s*t)(j)=t(|s|-j)$ for $|s|\leq j< |s|+|t|$. For a sequence $s^{\rho^{*}}$, we define $\overline{s}N:=\langle s(0), s(1), \dots,  s(N-1)\rangle $ for $N^{0}\leq |s|$.  
For a sequence $\alpha^{0\di \rho}$, we also write $\overline{\alpha}N=\langle \alpha(0), \alpha(1),\dots, \alpha(N-1)\rangle$ for \emph{any} $N^{0}$.  By way of shorthand, 
$(\forall q^{\rho}\in Q^{\rho^{*}})A(q)$ abbreviates $(\forall i^{0}<|Q|)A(Q(i))$, which is (equivalent to) quantifier-free if $A$ is.   
\end{nota}

\subsection{Higher-order computability theory}\label{HCT}
As noted above, some of our main results will be proved using techniques from computability theory.
Thus, we first make our notion of `computability' precise as follows.  
\begin{enumerate}
\item[(I)] We adopt $\ZFC$, i.e.\ Zermelo-Fraenkel set theory with the Axiom of Choice, as the official metatheory for all results, unless explicitly stated otherwise.
\item[(II)] We adopt Kleene's notion of \emph{higher-order computation} as given by his nine schemes S1-S9 (see \cites{longmann, Sacks.high}) as our official notion of `computable'.
\end{enumerate}
For the rest of this section, we introduce some existing axioms which will be used below.
These functionals constitute the counterparts of $\Z_{2}$, and some of the Big Five, in higher-order RM.  % by Remark \ref{fookie}.
First of all, $\ACA_{0}$ is readily derived from:
\begin{align}\label{mu}\tag{$\mu^{2}$}
(\exists \mu^{2})(\forall f^{1})\big[ (\exists n)(f(n)=0) \di [f(\mu(f))=0&\wedge (\forall i<\mu(f))(f(i)\ne 0) ]\\
& \wedge [ (\forall n)(f(n)\ne0)\di   \mu(f)=0]    \big], \notag
\end{align}
and $\ACA_{0}^{\omega}\equiv\RCAo+(\mu^{2})$ proves the same sentences as $\ACA_{0}$ by \cite{hunterphd}*{Theorem~2.5}.   The (unique) functional $\mu^{2}$ in $(\mu^{2})$ is also called \emph{Feferman's $\mu$} (\cite{avi2}), 
and is clearly \emph{discontinuous} at $f=_{1}11\dots$; in fact, $(\mu^{2})$ is equivalent to the existence of $F:\R\di\R$ such that $F(x)=1$ if $x>_{\R}0$, and $0$ otherwise (\cite{kohlenbach2}*{\S3}), and to 
\be\label{muk}\tag{$\exists^{2}$}
(\exists \varphi^{2}\leq_{2}1)(\forall f^{1})\big[(\exists n)(f(n)=0) \asa \varphi(f)=0    \big]. 
\ee
Secondly, $\FIVE$ is readily derived from the following sentence:
\be\tag{$S^{2}$}
(\exists S^{2}\leq_{2}1)(\forall f^{1})\big[  (\exists g^{1})(\forall n^{0})(f(\overline{g}n)=0)\asa S(f)=0  \big], 
\ee
and $\FIVE^{\omega}\equiv \RCAo+(S^{2})$ proves the same $\Pi_{3}^{1}$-sentences as $\FIVE$ by \cite{yamayamaharehare}*{Theorem 2.2}.   The (unique) functional $S^{2}$ in $(S^{2})$ is also called \emph{the Suslin functional} (\cite{kohlenbach2}).
By definition, the Suslin functional $S^{2}$ can decide whether a $\Sigma_{1}^{1}$-formula (as in the left-hand side of $(S^{2})$) is true or false.  

\smallskip

We similarly define the functional $S_{k}^{2}$ which decides the truth or falsity of $\Sigma_{k}^{1}$-formulas; we also define 
the system $\SIXK$ as $\RCAo+(S_{k}^{2})$, where  $(S_{k}^{2})$ expresses that $S_{k}^{2}$ exists.  Note that we allow formulas with \emph{function} parameters, but \textbf{not} \emph{functionals} here.
In fact, Gandy's \emph{Superjump} (\cite{supergandy}) constitutes a way of extending $\FIVE^{\omega}$ to parameters of type two.

\smallskip

\noindent
Thirdly, full second-order arithmetic $\Z_{2}$ is readily derived from $\cup_{k}\SIXK$, or from:
\be\tag{$\exists^{3}$}
(\exists E^{3}\leq_{3}1)(\forall Y^{2})\big[  (\exists f^{1})Y(f)=0\asa E(Y)=0  \big], 
\ee
and we therefore define $\Z_{2}^{\Omega}\equiv \RCAo+(\exists^{3})$ and $\Z_{2}^\omega\equiv \cup_{k}\SIXK$, which are conservative over $\Z_{2}$ by \cite{hunterphd}*{Cor.\ 2.6}. 
Despite this close connection, $\Z_{2}^{\omega}$ and $\Z_{2}^{\Omega}$ can behave quite differently\footnote{The combination of the recursor $\textsf{R}_{2}$ from G\"odel's $T$ and $\exists^{3}$ yields a system stronger than $\Z_{2}^{\Omega}$.   By contrast, $\SIXK$ and $\Z_{2}^{\omega}$ do not really change in the presence of this recursor.} in combination with other axioms.   The functional from $(\exists^{3})$ is called `$\exists^{3}$', and we use the same convention for other functionals.  
Note that $(\exists^{3})\asa [(\exists^{2})+(\kappa_{0}^{3})]$ (see \cite{dagsam, samsplit}) where the latter expresses comprehension on $C$: 
\be\tag{$\kappa_{0}^{3}$}
(\exists \kappa_{0}^{3}\leq_{3}1)(\forall Y^{2})\big[\kappa_{0}(Y)=0\asa (\exists f\in C)Y(f)=0  \big].
\ee
Finally, recall that the Heine-Borel theorem (aka \emph{Cousin's lemma}) states the existence of a finite sub-covering for an open covering of a compact space. 
Now, a functional $\Psi:\R\di \R^{+}$ gives rise to the \emph{canonical} covering $\cup_{x\in I} I_{x}^{\Psi}$ for $I\equiv [0,1]$, where $I_{x}^{\Psi}$ is the open interval $(x-\Psi(x), x+\Psi(x))$.  
Hence, the uncountable covering $\cup_{x\in I} I_{x}^{\Psi}$ has a finite sub-covering by the Heine-Borel theorem; in symbols:
\be\tag{$\HBU$}
(\forall \Psi:\R\di \R^{+})(\exists \langle y_{1}, \dots, y_{k}\rangle)\underline{(\forall x\in I)}(\exists i\leq k)(x\in I_{y_{i}}^{\Psi}).
\ee
By Theorem \ref{mooi} below, $\Z_{2}^{\Omega}$ proves $\HBU$, but $\SIXK+\QFAC^{0,1}$ cannot (for any $k$).  
As studied in \cite{dagsamIII}*{\S3}, many basic properties of the \emph{gauge integral} are equivalent to $\HBU$.  
By Remark \ref{kloti}, we may drop the requirement that $\Psi$ in $\HBU$ needs to be extensional on the reals, i.e.\ $\Psi$ does not have to satisfy \eqref{RE} from Definition \ref{keepintireal}.

\smallskip

Furthermore, since Cantor space (denoted $C$ or $2^{\N}$) is homeomorphic to a closed subset of $[0,1]$, the former inherits the same property.  
In particular, for any $G^{2}$, the corresponding `canonical covering' of $2^{\N}$ is $\cup_{f\in 2^{\N}}[\overline{f}G(f)]$ where $[\sigma^{0^{*}}]$ is the set of all binary extensions of $\sigma$.  By compactness, there is a finite sequence $\langle f_0 , \ldots , f_n\rangle$ such that the set of $\cup_{i\leq n}[\bar f_{i} G(f_i)]$ still covers $2^{\N}$.  By \cite{dagsamIII}*{Theorem 3.3}, $\HBU$ is equivalent to the same compactness property for $C$, as follows:
\be\tag{$\HBU_{\c}$}
(\forall G^{2})(\exists \langle f_{1}, \dots, f_{k} \rangle )\underline{(\forall f^{1}\leq_{1}1)}(\exists i\leq k)(f\in [\overline{f_{i}}G(f_{i})]).
\ee
We now introduce the specification $\SCF(\Theta)$ for a (non-unique) functional $\Theta$ which computes a finite sequence as in $\HBU_{\c}$.  
We refer to such a functional $\Theta^{2\di 1^{*}}$ as a \emph{realiser} for the compactness of Cantor space, and simplify its type to `$3$'.  % to improve readability.
\be\tag{$\SCF(\Theta)$}
(\forall G^{2})(\forall f^{1}\leq_{1}1)(\exists g\in \Theta(G))(f\in [\overline{g}G(g)])
\ee
%\edefi
Clearly, there is no unique such $\Theta$: just add more binary sequences to $\Theta(G)$.  
In the past, we have referred to any $\Theta$ satisfying $\SCF(\Theta)$ as a \emph{special fan functional} or a \emph{$\Theta$-functional}, and we will continue to use this language.  
As to its provenance, $\Theta$-functionals were introduced as part of the study of the \emph{Gandy-Hyland functional} in \cite{samGH}*{\S2} via a slightly different definition.  
These definitions are identical up to a term of G\"odel's $T$ of low complexity by \cite{dagsamII}*{Theorem 2.6}.  As shown in \cite{dagsamIII}*{\S3}, one readily obtains a $\Theta$-functional from $\HBU$ if the latter is given; in fact, it is straightforward to establish $\HBU\asa (\exists \Theta)\SCF(\Theta)$ over $\ACA_{0}^{\omega}+\QFAC^{2,1}$.

\section{Pincherle's theorem in computability theory}\label{PRS}
We answer the questions \eqref{kata} and \eqref{greek} from Section \ref{intro}.  
In Section \ref{PRR}, we show that \emph{Pincherle realisers} (PR hereafter) from Section \ref{sum}, cannot be computed by any type two functional.  
We also show that any PR (uniformly) give rise to a functional $F$ that is not Borel measurable\footnote{A functional is \emph{Borel measurable} if the inverse of an open set under that functional is a Borel set, or equivalently, that the graph of the functional is a Borel set.}.  The latter result follows from the \emph{extension theorem} (Theorem \ref{ET}).
On a conceptual note, a special fan functional $\Theta$ provides a finite sub-covering $\Theta(G)=\langle f_{0}, \dots, f_{k}\rangle$ for the uncountable covering $\cup_{f\in C}[\overline{f}G(f)]$, 
while any PR has an equivalent formulation outputting (only) an upper bound for the length $k$ of this finite sub-covering (see Observation \ref{nondezju}).  
Despite this big difference in output information (a finite sub-covering versus only an upper bound for its length), $\Theta$-functionals and PRs will be shown to be highly similar, to the point that we can only conjecture a difference (see Conjecture~\ref{corkes}).  

\smallskip

In Section \ref{PRR2}, we discuss similar questions for $\PIT_{o}$, which is equivalent to the Heine-Borel theorem for \emph{countable} coverings by Corollary \ref{ofmoreinterest} and \cite{simpson2}*{IV.1}, given countable choice as in $\QFAC^{0,1}$.      
However, it is trivial to compute a finite sub-covering as in the \emph{countable} Heine-Borel theorem, but the upper bound in $\PIT_{o}$ cannot be computed by any type two functional by Theorem \ref{kackney}. 
In this way, we observe a fundamental \emph{computational} difference between (countable) compactness and the local-to-global principle $\PIT_{o}$, as promised in item \eqref{ada} in Section \ref{intro}. 
Intuitively, this difference is caused by the use of countable choice \emph{and} the essential use of proof-by-contradiction.  \emph{Nonetheless}, $\PIT_{o}$ is provable \emph{without countable choice} by Theorem \ref{mooi}, i.e.\ the latter axiom is not the cause of the `strange' behaviour. 

\subsection{Realisers for uniform Pincherle's theorem}\label{PRR}
In this section we show that any PR has both considerable computational strength and hardness, as captured by the following theorems.  
\begin{theorem}\label{poilo}
There is no PR that is computable in a functional of type two.
\end{theorem}
\begin{theorem}\label{thm8.3}
There is an arithmetical  functional $F$ of type $(1\times1)\rightarrow  0$ such that for any PR $M$ we have that
$G(f) = M(\lambda g.F(f,g))$ is not Borel measurable.
\end{theorem}
\begin{theorem}[Extension Theorem]\label{ET} 
Let $M$ be a PR and let $e_0$ be a Kleene index for a partially computable functional $\Phi(F) = \{e_0\}(F,\mu)$. 
Then, uniformly in $M$, $\Phi$ has a total extension \(depending on $M$\) that is primitive recursive in $M,\mu$.
\end{theorem}
Note that Theorem \ref{ET} is the `higher-order' version of a known extension theorem.  Indeed, by Corollary \ref{ofmoreinterest}, $\PIT_{o}$ is equivalent to $\WKL$, and the latter implies: 
\begin{center}
\emph{If a partially computable $f:\N\di \N$ is bounded by a total computable function, then $f$ has a total extension.}
\end{center}
By the low basis theorem, the extension may be chosen to be of low degree.
Now, PRs are realisers for \emph{uniform} Pincherle's theorem (and for uniform $\WKL$ by Remark~\ref{flurki}), and Theorem \ref{ET} is the associated `higher-order' extension theorem,  where the concept of computability is relativised to Feferman's $\mu$ using S1-S9. 
By Corollary \ref{claptrap}, PRs also yield a higher-order version of the well-known separation theorem for $\Sigma_{1}^{0}$-sets that follows from $\WKL$ (see e.g.\ \cite{simpson2}*{I.11.7}).  The analogy with the low basis theorem will be that we can separate pairwise disjoint sets of type 2 functionals, semi-computable in $\mu$, with a set relative to  which not all semi-computable sets  are computable, so separation does not imply comprehension for sets semi-computable in $\mu$. It would be interesting to learn if some PRs can provide us with an analogue to sets of low degree. 

\smallskip

We first prove Theorem \ref{poilo}.  The proof is similar to the proof of the fact that no special fan functional $\Theta$ is computable in any type two functional (see \cite{dagsam}*{\S3}).
\begin{proof}
Suppose that $M$ is a PR and that $M$ is computable in the functional $H$ of type two. Thus there is an index $e$ such that 
$M(G) = \{e\}(G,H)$ for all total $G$ of type 2. Since $G$ and $H$ are of type 2, any $f$ such that $G(f)$ or $H(f)$ are called upon in a subcomputation of $\{e\}(G,H)$ via scheme S8, will itself be computable in $G$ and $H$.  
Note that this observation is not necessarily true for oracle calls on $G$ if $H$ were of type 3 or higher.

\smallskip

Without loss of generality, we may assume that $\exists^2$ is computable in $H$, so the machinery of Gandy selection (\cite{longmann}*{p.\ 210}) is at our disposal. We define the (partial) functional $G:C \rightarrow \N \cup \{\bot\}$ by $G(f) = e+1$ where $e$ is the index of $f$ as a function computable in $H$ obtained by application of Gandy selection.  We put $G(f) =\bot$ if $f$ is not computable in $H$.

\smallskip

Now let $\hat G$ be any total extension of $G$. If we evaluate $M(\hat G) = a$ following the assumed algorithm for $M$ from $H$, we see that we actually can replace $\hat G$ with $G$ in the full computation tree. For this we use that $G$ is partially  computable in $H$, and that we will only call upon $\hat G(f)$ for $H$-computable $f$. This step actually requires a proof by induction on the ordinal rank of the subcomputation, an argument that is standard and omitted.

\smallskip

Thus, $M(\hat G)$ is independent of the choice of $\hat G$. On the other hand, we have that for any $N$, the set of $g$ where the bounding condition $\LOC(F,G)$ forces $F(g)$ to be bounded by $N$, is a small and clopen set.  Hence, if we make sure $\hat G(g) > N$ for all $g$ not computable in $H$, we obtain a contradiction. %This ends the proof.
\end{proof}
We now prove a number of theorems, culminating in a proof of Theorem \ref{ET}.  
We assume $M$ to be a PR for the rest of this section. 
\begin{theorem}\label{thm9.1}
For each Kleene-index $e_0$ and numbers $a_0,n$, there are arithmetical, uniformly in $e_0,a_0,n$, functionals $F \mapsto F_{e_0,a_0,n}$ of type $2 \rightarrow 2$ such that if $\{e_0\}(F,\mu)\!\!\downarrow$, 
we can, independently of the choice of $M$, find the value  $a$ of the computation from $\lambda (a_0,n).M(F_{e_0,a_0,n})$ in an arithmetical manner.
\end{theorem}
\begin{proof}
We let $M$, $F$, $e_0$, $a_0$, and $n$ be fixed throughout.
We first need some notation. 

\smallskip

Let $R$ be a preordering of a domain $D\subseteq \N$. For $x \in D$, we denote
\begin{itemize}
\item $[x]^R =  \{y \in D \mid (y,x) \in R\}$
\item $[x]_R = \{y \in D \mid (y,x) \in R \wedge \neg((x,y) \in R)\}$
\item $R^x$ is $R$ restricted to $[x]^R$
\item $R_x$ is $R$ restricted to  $[x]_R$
\end{itemize}

\smallskip

  Let $f \in C$ and define $D_f := \{x  \mid f(\langle x,x \rangle) = 1\}$ and $R_f := \{(x,y) \mid f(\langle x,y \rangle) = 1\}$, where $x,y \in \N$.
Let $\PRE$ be the set of $f \in C$ such that $R_f$ is a preordering of $D_f$. Then $\PRE$ is a $\Pi^0_1$-set, and for each $f \not \in {\rm \PRE}$, we can find an integer $k$ such that $[\bar fk] \cap {\rm \PRE} = \emptyset.$

\smallskip

Let $\Gamma_F$ be the monotone inductive definition of $D_F = \{\langle e,\vec a , b\rangle \mid \{e\}(F, \mu , \vec a) = b\}.$
Since each valid computation $\{e\}(F,\mu , \vec a) = b$ has an ordinal rank $\| \langle e,\vec a , b\rangle \|_F < \aleph_1$, let $R_F$ be the pre-well-ordering on $D_F$ induced by $\| \cdot \|_F$. 
Then $R_F$ is the least fixed point of an, uniformly in $F$, arithmetical and monotone inductive definition $\Delta_F$ such that for all ordinals $\alpha$, $\Delta_F^{\alpha + 1}$ is an end extension of $\Delta_F^\alpha$, where we write $\Delta_F^\alpha$ for $\Delta_F^\alpha(\emptyset)$.

\smallskip

Now, let $R$ be any preordering of the domain $D \subseteq \N$. We call $x \in D$ an \emph{$F$-point} if $R^x = \Delta_F(R_x)$.  
We let $D[F]$ be the maximal $R$-initial segment consisting of $F$-points, and we let $R[F]$ be $R$ restricted to $D[F]$.

\smallskip

{\em Claim 1}:  If $R[F]$ does not contain $R_F$ as an initial segment, then $R[F]$ is an initial segment of $R_F$.
\vspace{2mm}
\newline
{\em Proof of Claim 1}.  Let $\alpha$ be the least ordinal such that $\Delta^\alpha_F$ is not an initial segment of $R[F]$. Then $\bigcup_{\beta < \alpha}\Delta_F^\beta$ is an initial segment of $R[F]$. If this is all of $R[F]$, we are through, since then $R[F]$ is an initial segment of $R_F$. If not, there is some $x \in D[F]$ such that $\bigcup_{\beta < \alpha}\Delta^\beta_F$ is an initial segment of $R_x[F]$.
Since $x$ is an $F$-point and $\Delta_F$ is monotone, we have that $R^x[F] = \Delta_F(R_x[F])$ and that $\Delta^\alpha_F$ is an initial segment of $R^x[F]$, contradicting the choice of $\alpha$. Claim 1 now follows.
\smallskip

\noindent For now, we assume that $f \in {\rm \PRE}$.

\smallskip

{\em Claim 2}:  If $R_f[F]$ is not a fixed point of $\Delta_F$, there is $k\in \N$, $\mu$-computable from $F, f$,
%identifiable from $R_f[F]$ and $f$ in an arithmetical manner, 
such that whenever $g \in {\rm \PRE}$ such that $\Delta_F(R_f[F])$ is an initial segment of $R_g[F]$ we have that $g(k) \neq f(k)$.
\vspace{2mm}
\newline
{\em Proof of Claim 2}.
If there is a pair $(y,x) \in \Delta_F(R_f[F])$ such that $f(\langle y,x \rangle) = 0$, we can just let $k = \langle y,x\rangle$ for one such pair, chosen by numerical search. 
Now assume $f(\langle y,x \rangle) = 1$ whenever $(y,x) \in \Delta_F(R_f[F])$. Now, for all $x \in D_f[F]$ we have that $(R_f)_x \subseteq (R_f)^x = \Delta_F((R_f)_x)$.  Since $\Delta_F$ is monotone, we must have that 
$R_f[F] \subseteq \Delta_F(R_f[F])$. Further, since $R_f[F]$ is not a fixed point of $\Delta_F$, there is some $x$ such that $(x,x) \in \Delta_F(R_f[F]) \setminus R_f[F]$.
Since this $x$ is not an $F$-point, and since $$f(\langle x,y \rangle) = f(\langle y , x \rangle) = 1$$ whenever $(x,y) \in \Delta_F(R_f[F])$ and $(y,x) \in \Delta_F(R_f[F])$, there must be a $y$ such that $f(\langle x,y \rangle) = f(\langle y,x \rangle) = 1$, but $(x,y) \not \in \Delta_F(R_f[F])$ or $(y,x) \not \in \Delta_F(R_f[F])$. We can find such a pair $k = \langle x,y\rangle$ or $k = \langle y,x \rangle$ by effective search. Claim 2 now follows.

\smallskip

We now define $F_{e_0,a_0,n}(f)$, where $f \in C$ is not necessarily in $\PRE$ anymore.
\bdefi\label{godallemagtig}
We define $F_{e_0,a_0,n}(f)$ by cases, assuming for each case that the previous cases fail:
\begin{enumerate}
\item For $f \not \in {\rm \PRE}$, let $F_{e_0,a_0,n}(f) = k$ for the least $k$ such that $[\bar fk] \cap {\rm \PRE}  = \emptyset$
\item There is an $a \in \N$ such that 
\begin{itemize}
\item[(2.i)] $\langle e_0,a\rangle$ is in the domain of $R_f[F]$
\item[(2.ii)] For no $b \in \N$ with $b \neq a$ do we have that $(\langle e_0,b \rangle , \langle e_0,a\rangle) \in R_f[F]$.
\end{itemize}
We then let $F_{e_0,a_0,n}(f) = 0$ if $a \neq a_0$ and $n$ if $a = a_0$.
\item $R_f[F]$ is a fixed point of $\Delta_F$. Then let $F_{e_0,a_0,n}(f) = 0$.
\item $R_f[F]$ is not a fixed point of $\Delta_F$. Then let $F_{e_0,a_0,n}(f) = k+1$, where $k$ is the number identified in Claim 2.
\end{enumerate}
\edefi
\noindent
We now prove the theorem via establishing the following final claim. 

\smallskip

{\em Claim 3}: If $\{e_0\}(F,\mu)\!\!\downarrow$, the following algorithm provides the result:
\begin{center} 
$\{e_0\}(F,\mu)$ is the unique $a_0$ for which $\{M(F_{e_0,a_0,n}) \mid n \in \N\}$ is infinite.
\end{center}
This algorithm is uniformly arithmetical in $M$, by definition. 
\vspace{2mm}
\newline
{\em Proof of Claim 3}.
Assume that $\{e_0\}(F,\mu) = a$, implying that $\langle e_0,a\rangle$ is in the well-founded part of $R_F$.

\smallskip

First assume that $a \neq a_0$. We claim that the value of $F_{e_0,a_0,n}(f)$  is independent of $n$ for all inputs $f$, implying that $M(F_{e_0,a_0,n})$ has a fixed value independent of $n$. Considering Definition \ref{godallemagtig}, we see that (1) only depends on $f$ and if (2) decides the value, we must have that the $a$ required in (2) is the value of $\{e_0\}(F,\mu)$,  we will use the case $a \neq a_0$ and the value is still independent of $n$. Clearly, if we need (3) or (4), the value is also independent of $n$.

\smallskip

Secondly, assume that  $a = a_0$. As a sub-claim, we will show that $M(F_{0,a_0,n}) \geq n$, from which the conclusion follows.
Let $g \in C\cap {\rm \PRE}$ be such that $R_F = R_g[F]$, and let $f$ be arbitrary such that $g \in [\bar fF_{e_0,a_0,n}(f)]$. We will see that $F_{e_0,a_0,n}(f) \geq n$, and our sub-claim follows.
If $f \not \in {\rm \PRE}$, we see from (1) in Definition \ref{godallemagtig} that $g\not \in [\bar fF_{e_0,a_0,n}(f)]$ since $g \in {\rm PRE}$, so (1)   does not apply in the definition of $F_{0,a_0,n}(f)$ when $g$ and $f$ are as given.
If $f \in {\rm \PRE}$, but $\langle e_0 , a_0 \rangle$ is not in the domain of $R_{f}[F]$, then by Claim 1, $R_{f}[F]$ is a proper initial segment of $R_F$, and using Claim 2 we have chosen $F_{e_0,a_0,n}(f) = k+1$ in such a way that $g(k) \neq f(k)$. This again violates our assumption on $f$.

\smallskip

Thus, by our assumption on how $f$ is related to $g$, we must have that $\langle e_0,a_0\rangle$ is in the domain of $R_{f}[F]$, and since this appearance will be in the well-founded part, there will be no competing values $b$ at the same or lower level.  Then we will use (2) from Definition \ref{godallemagtig}, and set the value of $F_{e_0,a_0,n}(f)$ to $n$.
\end{proof}
As an immediate consequence, we obtain a proof of Theorem \ref{ET}.
\begin{proof}
Let $F_{e_0,a_0,n}$ be as in the proof of Theorem \ref{thm9.1}. Define $\Psi(F) = a_0$ if $a_0$ is unique such that $M(F_{e_0,a_0,n}) \geq n$ for all $n$, and define $\Psi(F) = 0$ if there is no such unique $a_0$.
\end{proof}
Finally, we list some corollaries to the theorem. 
\begin{corollary}\label{cor9.3}
Let $\Phi^{3}$ be partial and Kleene-computable in $\mu$.
Then for any PR $M$ there is a total extension of $\Phi$ that is primitive recursive in $M$ and $\mu$.
\end{corollary}
This is almost a rephrasing of Theorem \ref{ET}, modulo some coding of mixed types.
\begin{corollary}\label{claptrap}
Let $X$ and $Y$ be disjoint sets of functionals of type 2, both semicomputable in $\mu$.
Then, for each PR $M$, there is a set $Z$ primitive recursive in $M$ and $\mu$, that separates $X$ and $Y$.
\end{corollary}
\begin{proof}
Since we use $\mu$ as a parameter, we have Gandy selection in a uniform way, so there will be a partial function computable relative to $\mu$ that takes the value 0 on $X$ and 1 on $Y$. Then apply Corollary \ref{cor9.3}.
\end{proof}
As a special case, we obtain the proof of Theorem \ref{thm8.3}, as follows.
\begin{proof}
Let $X = \{(e,f) \mid \{e\}(e,f,\mu) = 0\}$ and $Y = \{(e,f) \mid \{e\}(e,f,\mu) = 1\}$.
The sets $X$ and $Y$ are Borel-inseparable disjoint $\Pi^1_1$-sets, but can be separated using one parameterised application of $M$.
\end{proof}
As another application of Corollary \ref{cor9.3} we see that the partial enumeration of all hyperarithmetical functions, which is partially computable in $\mu$, can be extended to a total enumeration primitive recursive in $M$ and $\mu$ for all Pincherle realisers $M$.  We leave further applications to the imagination of the reader.

\smallskip

It is a natural question whether there are reasonable lower bounds on PRs.  
Now, Hunter introduces a functional in \cite{hunterphd}*{p.\ 23} that constitutes a `uniform' version of $\ATR_{0}$.  
The following corollary shows that this functional is a lower bound on PRs.  
\begin{cor}
Uniformly primitive recursive in $\mu^{2}$ and any \textup{PR} $M^{3}$ there is a functional
$T:\N^\N \times (2^\N \rightarrow 2^\N) \rightarrow 2^\N$ such that when $f^{1}$ codes a well-ordering $<_f$, then $T(f,F)$ satisfies the following recursion equation for $a$ in the domain of $<_f$\textup{:}
\[
\{b : \langle b,a \rangle \in T(f,F)\} = F(\{\langle c,d \rangle \in T(f,F) : d <_f a\}).
\]
\end{cor}
\begin{proof}
Let $f^{1}$ code the well-ordering $<_f$. By the recursion theorem for S1-S9. there is an index $e$, independent of $f$, such that
$T^*(f,F) := \lambda x^0.\big(\{e\}(F,f,x)\big)$ terminates and satisfies the recursion equation for $x = \langle b,a\rangle$ whenever $a$ is in the domain of $<_{f}$. By the Extension Theorem \ref{ET}, the functional $T^{*}$ has a total extension $T$ that is primitive recursive in $\mu$ and any given PR $M$.
\end{proof}
The previous results and the equivalence in Corollary \ref{eessje} suggest a strong similarity between special fan functionals and PRs.  
In fact, Theorem \ref{poilo} can be seen as a consequence of the following theorem and the properties of $\Theta$-functionals established in \cite{dagsam, dagsamII}.  
We establish (and make essential use of) the equivalences in Observation \ref{nondezju} when discussing Heine's theorem below.  
\begin{obs}\label{nondezju}
Let $G:C \rightarrow \N$. The following are equivalent for any $n \in \N$:
\begin{enumerate}
\item There is a {\rm PR} $M$ with $M(G) = n$
\item There is a special fan functional $\Theta$ such that $G(f) \leq n$ for each $f \in \Theta(G)$
\item There are $f_1 , \ldots , f_k \in C$ with $C\subset \cup_{i\leq k}[\bar f_iG(f_i)]$ and $n \geq G(f_i)$ for $i\leq k$.
\end{enumerate}
\end{obs}  
Despite these similarities, there are certain fundamental differences between special fan functionals and Pincherle realisers, leading to the following conjecture.  
Even if it turns out to be incorrect, we still expect that there is no \emph{uniform} way to compute a $\Theta$-functional from an instance of $M$, even modulo $\exists^2$.
\begin{conj}\label{corkes}
There is an $M_{0}^{3}$ satisfying $\PR(M_{0})$ such that no $\Theta^{3}$ as in $\SCF(\Theta)$ is computable \textup{(S1-S9)} in $M_{0}^{3}$.   
\end{conj}
We finish this section with a remark on the exact formulation of (realisers for) local boundedness; recall that we used \emph{one functional} $G$ in $\LOC(F, G)$. 
\begin{rem}\label{nodiff}\rm
In order to be faithful to the original formulation of Pincherle, the bounding condition has to be given by two functionals $G_1$ and $G_2$, as follows: 
\[
\LOC^*(F, G_1,G_2)\equiv (\forall f , g\in C)\big[ g\in [\overline{f}G_1(f)] \di F(g)\leq G_2(f)    \big].
\]
Let $M^*$ be a functional which on input $(G_{1}, G_{2})$ provides an upper bound on $C$ for $F$ satisfying $\LOC^{*}(F, G_{1}, G_{2})$.  
A PR $M$ can be reduced to such $M^{*}$, and vice versa, as follows: $M(G) = M^*(G,G)$ and  $M^*(G_1,G_2) = M(\max\{G_1,G_2\})$.  % they are essentially equivalent. 
\end{rem}

\subsection{Realisers for original Pincherle's theorem}\label{PRR2}
In this section, we study the computational properties of realisers for $\PIT_{o}$.  
As discussed in Section \ref{introwpr}, there are two possible candidates for such realisers, in contrast to $\PIT_{\u}$, where there was only one natural choice.  
We shall study the most natural (and weakest) candidate and argue why the other (stronger) candidate does not really qualify as a realiser.  
Perhaps surprisingly, our natural candidate exhibits `extreme' computational behaviour similar to PRs, as shown in Section \ref{stream2}, despite the lack of uniformity.
\subsubsection{A realiser by any other name}\label{introwpr}
In this section, we sketch two possible candidates for the notion `realiser for $\PIT_o$', and argue for the study of the most natural (and weakest) one.  
These two kinds of realisers arise from the two possible kinds of realisers for $\ATR_{0}$: based on \eqref{WPR1} and \eqref{WPR2} respectively.  The latter formulas are classically equivalent, but yield very different realisers, as discussed now.  
 \begin{rem}\rm
 In \cite{dagsam,dagsamII} we proved that a special fan functional $\Theta$ (with Feferman's $\mu$) computes a realiser for $\ATR_0$ as follows: given a total ordering `$\prec$' and an arithmetical operator `$\Gamma$', we can compute a pair $(x,y)$ such that either $x$ codes a $\Gamma$-chain over $\prec$, or $y$ codes a $\prec$-descending sequence. This is a realiser for:
\be\label{WPR1}
 \neg \WO(\prec) \vee (\exists X\subset \N)(\textup{$X$ is a $\Gamma$-chain over $\prec$}).
 \ee
The situation is different for PRs: \emph{if} $\prec$ is a well-ordering, \emph{then} we can compute the unique $\Gamma$-chain $X$, and by the Extension Theorem \ref{ET}, there is for any $\PR$ $M$, 
a total functional $\Delta(\prec,\Gamma)$ that is primitive recursive in $M, \mu$, and that gives us a $\Gamma$-chain (over $\prec$) \emph{assuming} $\prec$ is a well-ordering.  
The difference is that for PRs, one does not obtain an infinite  $\prec$-descending sequence when $\Delta(\prec,\Gamma)$ is not a $\Gamma$-chain.
Thus, any PR yields a realiser for: 
\be\label{WPR2}
\WO(\prec)\di (\exists X\subset \N)(\textup{$X$ is a $\Gamma$-chain over $\prec$}).
\ee
In conclusion, a difference between special fan functionals and PRs is that the former yield a `stronger' realiser, namely for \eqref{WPR1}, while the latter provide a `weaker' realiser, namely (only) for \eqref{WPR2}.
\end{rem}
In light of the previous remark, we could define two kinds of realisers for $\PIT_{o}$, namely one similar to \eqref{WPR1} and one similar to \eqref{WPR2}.  
However, the first option is not really natural in light of the fact that any PR only provides a realiser for \eqref{WPR2}.  
Thus, we shall study only realisers for $\PIT_{o}$ based on \eqref{WPR2}, namely as follows. 
\bdefi A \emph{weak} Pincherle realiser is any functional $M_{o}^{3}$ satisfying
\be\tag{$\WPR(M_{o})$}
(\forall F^2,G^2 )(\LOC(F,G) \rightarrow (\forall f^1)(F(f) \leq M_o(F,G))).
\ee
\edefi
Note that $\WPR(M_{o})$ is just $\PR(M)$ where $F$ is an additional input.  
Finally, it is an interesting exercise to show that a realiser for $\PIT_{o}$ based on \eqref{WPR1} has the same computational and RM properties as the functional $\kappa_{0}^{3}$ from Section~\ref{HCT}.  
This observation also suggests that \eqref{WPR1} does not yield a natural notion of realiser.  

\subsubsection{A weak realiser for Pincherle's theorem}\label{stream2}
We study the computational properties of weak Pinchere realisers (WPRs hereafter) from the previous section.  %  their computational properties. 

\smallskip

Now, by Theorem \ref{mooi} and Corollary \ref{ofmoreinterest}, $\PIT_o$ is provable without countable choice, and equivalent to $\WKL$ given countable choice.  
These results suggest that $\PIT_{o}$ enjoys a certain constructive nature (see in particular Footnote \ref{PCbrigade}).  Now, $\WKL$ is equivalent to the Heine-Borel theorem for \emph{countable coverings}, and a finite sub-covering of the latter (say for $C$ or $[0,1]$) is trivially computable.  
By contrast, it is impossible to compute any WPR from any type two functional, as follows. 
\begin{theorem}\label{kackney}
There is no functional $M_o$ at type level 3, computable in any type 2 functional, such that $\WPR(M_{o})$.
\end{theorem}
\begin{proof} The proof follows the pattern of our proofs  of similar results.  Let $H$ with $\mu\leq_{\textup{Kleene}} H$ be any type 2 functional, and assume that $M_o$ is computable in $H$.  Let $F^*$ be partially $H$-computable and injective on the set of $H$-computable functions, taking only values $>1$ and let $G^*$ be the constant 0. Then $\LOC(G^*,F)$ for any total extension $F$ of $F^*$. 

\smallskip

The computation of $M_o(G^*,F) = N$ from $H$ will then only make oracle calls $F(f) = F^*(f)$ or $G(f) = 0$ for a countable set of $f$'s enumerable by an $H$-comptutable function. Define $G(f) := N+1$ if $f$ is neither in this enumerated set nor in any neighbourhood induced by $F^*(f)$ where $F^*(f) \leq N$, and 0 elsewhere.  Also define $F(f) = F^*(f)$ when the latter is defined, and $N+1$ elsewhere.  We now still have that $M_0^*(G,F) = N$ and $\LOC(G,F)$, but not that $N$ is an upper bound for $G$. This is the desired contradiction, and the theorem follows.
\end{proof}
Recall that `$\WPR$' stands for `weak PR' in the theorem.  
Despite this suggestive name, the combination of Theorem \ref{horniii} and Theorem \ref{mooier} yields a model that satisfies $\SIXK$, but falsifies $\PIT_{\u}$ and is lacking any and all realisers for Pincherle's theorem. 
A functional of type two is \emph{normal} if it computes the functional $\exists^{2}$.  
\bdefi\label{nikeh}
For normal $H^{2}$, the type structure ${\mathcal M}^H = \{{\mathcal M}^H_k\}_{k \in \N}$ is defined as ${\mathcal M}^H_0 = \N$ and ${\mathcal M}^H_{k+1}$ consists of all $\phi:{\mathcal M}^H_k \rightarrow \N$ computable in $H$ via Kleene's S1-S9.
The set ${\mathcal M}^H_1$ is the \emph{1-section} of $H$; the restriction of $H$ to $\mathcal{M}_{1}^{H}$ is in ${\mathcal M}^H_2$.
\edefi
\begin{theorem}\label{horniii}
For any normal $H^{2}$, the type structure ${\mathcal M}^H$ is a model for $\QFAC^{0,1}$, $ \neg \PIT_{\u}$, $\PIT_{o}$, $(\forall M_{o})\neg\WPR(M_{o})$, and $(\forall M_u) \neg \PR(M_u)$.  % and $(\forall M_{o}^{*})\neg\MPR(M_{o}^{*})$. 
\end{theorem}
\begin{proof}
Fix $N\in \N$ and let $H^*(f) = e+1$ where $e$ is some $H$-index for $f$ found using Gandy selection. The first claim follows readily from Gandy selection. The second claim is proved as for \cite{dagsamIII}*{Theorem 3.4} by noting that $H^*$ restricted to the finite set $\{f_1 , \ldots , f_k\}$ of functions $f$ with $H^*(f) \leq N$ does not induce a sufficiently large sub-covering of $C$ to guarantee that all $F$ satisfying the bounding condition induced by $H^*$ is bounded by $N$. 

\smallskip

In order to prove the third claim, let $G\in {\mathcal M}^H_2$ be arbitrary, and let $F \in {\mathcal M}^H_2$ satisfy the bounding condition induced by $G$. Assume that $F$ is unbounded. Then, employing Gandy selection we can, computably in $H$, find a sequence $\{f_i\}_{i \in \N}$ such that $F(f_i) > i$ for all $i$. Using $\exists^2$ we can then find a convergent subsequence and compute its limit $f$. Then $F$ will be bounded by $G(f)$ on the set $[\bar fG(f)]$, contradicting the choice of the sequence $f_i$. 

\smallskip

In order to prove the fourth claim, assume that $M_o \in {\mathcal M}^H_3$ is a WPR in ${\mathcal M}^H$. Let $F$ be the constant zero, and let $M_o( F,H^*) = N$. We now use that $M_o$ is computable in $H$,  that thus the computation tree of $M_o(F,H^*)$ in itself is computable in $H$. In particular, the set of functions $f$ such that $F(f)$ is called upon in the computation of   $M_o(F,H^*)$ will have an enumeration computable in $H$. By a standard diagonal argument, any basic open neighbourhood will contain a function that is computable in $H$, but not in this set.

\smallskip

Let $f_1 , \ldots , f_k$ be as in the argument for the second claim. There is a basic neighbourhood in $C$ disjoint from $\bigcup_{i = 1}^k[\bar f_iH^*(f_i)]$, and thus there is some $f$ computable in $H$ that is not in any of the neighbourhoods $[\bar f_iH^*(f_i)]$,  and such that $F(f)$ is not called upon in the computation of $M_o(F,H^*)$. We may now define $F_N$ so that $F_N(f) = 0$ if $F(f) = 0$ is used in the computation of $M_o(G, H^*)$ or if $f \in [\bar f_i(H^*(f_i)]$ for $i = 1 , \ldots , k$, and we let $F_N(f) = N+1$ otherwise. Then the computation of $M_o(F_N,H^*)$ yields the same value as the computation of $M_o(F,H^*)$ and $F_N$ still satisfies the bounding condition induced by $H^*$, but the output does not give an upper bound for $F_N$.
The fifth  claim follows from the fourth claim.
\end{proof}
This theorem can be sharpened: $\Z_{2}^{\omega}$ cannot prove $\PIT_{o}$ by Corollary \ref{hantonio}.

\smallskip

Finally, we study a number of concepts in this paper (compactness, the Lindel\"of property, local-to-global principles) and it is a central question how hard it is to compute the objects in the associated theorems. 
In particular, one wonders whether there are natural upper and lower bound in terms of computational hardness.  We provide a positive answer to this question in the following remark. 
\begin{rem}[Bounds to computational hardness]\label{draak}\rm
First of all, the most complex step in reducing a covering of $2^{\N}$ or $[0,1]$ to a \emph{finite} one lies in reducing a general covering to a \emph{countable} one, as illustrated by the following three results.    
\begin{enumerate}
\item The existence of arithmetical coverings of $[0,1]$ or $2^{\N}$ with no hyper-arithmetically enumerable sub-coverings; this result is implicit in \cite{dagsamIII}*{\S4}.
\item The \emph{Lindel\"of lemma} states that general coverings can be reduced to countable coverings.  We show in Section \ref{not2bad} that this lemma for Baire space, called $\LIND(\N^{\N})$, combined with $(\exists^{2})$, implies $\Pi_{1}^{1}$-comprehension.  
\item Given $(\exists^{2})$, realisers for $\LIND(\N^\N)$ are computationally equivalent to the closure operator for \emph{non-monotone inductive definitions} on $\N$, and to $S^{2}+\Theta^{3}$ (see \cite{dagcie18}), all computationally powerful (and hard) objects.
\end{enumerate}
Secondly, in light of these results, general non-monotone induction is a natural upper bound for most of the classical theorems under investigation in this paper, from the point of view of computability. 
Moreover, it seems the associated results in \cite{dagcie18}*{\S3.1} can be converted to equivalences over $\RCAo$, i.e.\ general non-monotone induction is also a natural upper bound from the point of view of RM.
As to lower bounds, any PR computes Hunter's functional from \cite{hunterphd}*{p.\ 23} and the latter yields a `uniform' version of $\ATR_{0}$.  Thus, Hunter's functional provides a lower bound for uniform Pincherle's theorem, and all other theorems.  

\smallskip

Thirdly, the following figure provides a concise and elegant (but not entirely precise) overview of the above.
\be\tag{$\textsf{H}$}
\begin{array}{ccccc}
\left\{\begin{array}{c}
\textup{unbounded} \\
\textup{transfinite}\\
\textup{recursion} \\
\end{array}\right\}
&\di &
\left\{\begin{array}{c}
\textup{compactness} \\
\textup{and related}\\
\textup{theorems} \\
\end{array}\right\}
&\di &
\left\{\begin{array}{c}
\textup{bounded} \\
\textup{transfinite}\\
\textup{recursion} \\
\end{array}\right\}
\end{array}
\label{H}
\ee
Here, `unbounded' transfinite recursion takes place over the countable ordinals, while `bounded' transfinite recursion involves a fixed countable ordinal given by a code.
Note that $\Pi^1_2$-comprehension does not follow from the principles in \eqref{H}.  
%Thus, the latter also fall outside of the usual scale of comprehension axioms.  
\end{rem}
\section{Pincherle's and Heine's theorem in Reverse Mathematics}\label{prm}
Motivated by question \eqref{turk}, we classify the original and uniform versions of Pincherle's and Heine's theorems within the framework of higher-order RM.
In particular, we show that $\PIT_{o}$ is provable from \emph{weak K\"onig's lemma} \emph{given countable choice} as in $\QFAC^{0,1}$, but without the latter even $\SIXK$ cannot prove $\PIT_{o}$ (for any $k$); the system $\Z_{2}^{\Omega}$ of course proves $\PIT_{o}$ and does not involve countable choice.
To be absolutely clear, $\PIT_{o}$ is provable \emph{without $\QFAC^{0,1}$} by Theorem \ref{mooi}, but its logical hardness drops tremendously given $\QFAC^{0,1}$ in light of Corollary~\ref{ofmoreinterest} and Theorem~\ref{dagsamsyn}.  
The main question of RM, namely which set existence axioms are necessary to prove $\PIT_{o}$, therefore does not have an unique or unambiguous answer.
In this way, we observe a fundamental logical difference between countable compactness and the local-to-global principle $\PIT_{o}$, as discussed in Section~\ref{intro}.
\subsection{Pincherle's theorem and second-order arithmetic}\label{prm1}
We formulate a supremum principle which allows us to easily obtain $\HBU$ and $\PIT_{\u}$ from $(\exists^{3})$; this constitutes a significant improvement over the results in \cite{dagsamIII}*{\S3}.   
We show that $\PIT_{\u}$ is not provable in any $\SIXK$, but that $\ACAo+\QFAC^{0,1}$ proves $\PIT_{o}$.

\smallskip

First of all, a formula $\varphi(x^{1})$ is called \emph{extensional on $\R$} if we have
\[
(\forall x, y\in \R)(x=_{\R}y\di \varphi(x)\asa \varphi(y)).  
\]
Note that the same condition is used in RM for defining open sets as in \cite{simpson2}*{II.5.7}.  
\begin{princ}[$\LUB$]
For second-order $\varphi$ \(with any parameters\), if $\varphi(x^{1})$ is extensional on $\R$ and $\varphi(0)\wedge \neg\varphi(1)$, there is a least $y\in [0,1]$ such that $(\forall z \in (y, 1])\neg\varphi(z)$.  
\end{princ}
Secondly, we have the following theorem, which should be compared\footnote{Keremedis proves in \cite{kermend} that the statement \emph{a countably compact metric space is compact} is not provable in $\ZF$ minus the axiom of foundation.  This theorem does follow when the axiom of countable choice is added.} to \cite{kermend}*{\S2}.    
The reversal of the final implication is proved in Corollary \ref{eessje}, using $\QFAC^{0,1}$.
\begin{thm}\label{mooi}
The system $\RCAo$ proves $(\exists^{3})\di \LUB\di \HBU\di \HBU_{\c}\di  \PIT_{\u}$.
\end{thm}
\begin{proof}
For the first implication, note that $(\exists^{3})$ can decide the truth of any formula $\varphi(x)$ as in $\LUB$, as the former implies $S_{k}^{2}$ for any $k^{0}$ and therewith $\SIXk$.  
Hence, the usual interval-halving technique yields the least upper bound as required by $\LUB$.  
For the second implication, fix $\Psi:\R\di \R^{+}$ and consider 
\[
\varphi(x)\equiv x\in [0,1] \wedge (\exists w^{1^{*}})(\forall y^{1}\in [0,x])(\exists z\in w)(y\in I_{z}^{\Psi}), 
\]
which is clearly extensional on $\R$.  Note that $\varphi(0)$ holds with $w=\langle 0\rangle$, and if $\varphi(1)$, then $\HBU$ for $\Psi$ follows.  In case $\neg\varphi(1)$, we use $\LUB$ to find 
the least $y_{0}\in [0,1)$ such that $(\forall z>_{\R} y_{0})\neg\varphi(z)$.  However, by definition $[0, y_{0}-\Psi(y_{0})/2]$ has a finite sub-covering (of the canonical covering provided by $\Psi$), and hence clearly so does $[0, y_{0}+\Psi(y_{0})/2]$, a contradiction. 
For the final implication, to obtain $\PIT_{\u}$, let $F_{0}, G_{0}$ be such that $\LOC(F_{0}, G_{0})$ and let $w_{0}^{1^{*}}$ be the finite sequence from $\HBU_{\c}$ for $G=G_{0}$.  
Then $F_{0}$ is clearly bounded by $\max_{i<|w_{0}|}G_{0}(w(i))$ on Cantor space, and the same holds for \emph{any} $F$ such that $\LOC(F, G_{0})$, as is readily apparent. 

\smallskip

Finally, $\HBU\di \HBU_{\c}$ is readily proved \emph{given} $(\exists^{2})$, since the latter provides a functional which converts real numbers into their binary representation(s).  
Moreover, in case $\neg(\exists^{2})$ all functions on Baire space are continuous by \cite{kohlenbach2}*{Prop.\ 3.7}.  
Hence, $\HBU_{\c}$ just follows from $\WKL_{0}$ (which is immediate from $\HBU$): the latter lemma suffices to prove that a continuous function is uniformly continuous on Cantor space by \cite{kohlenbach4}*{Prop.~4.10}, and hence bounded.  
The law of excluded middle $(\exists^{2})\vee \neg (\exists^{2})$ finishes this part, as we proved $\HBU\di \HBU_{\c}$ for each disjunct.   
\end{proof}
The first part of the proof is similar to Lebesgue's proof of the Heine-Borel theorem from \cite{lebes1}.  We also note that Bolzano used a theorem similar to $\LUB$ (see \cite{nogrusser}*{p.\ 269}).
We now identify which comprehension axioms can prove $\PIT_{\u}$. 
\begin{thm}\label{mooier}
The system ${\Z}_{2}^{\Omega}$ proves $\PIT_{\u}$, while no system $\FIVEK^{\omega}+\QFAC^{0,1}$ \($k\geq1$\) proves it.  
\end{thm}
\begin{proof}
The first part follows from Theorem \ref{mooi}.  
For the second part, we construct a countable model $\mathcal{M}$ for $\FIVEK^{\omega} + \QFAC^{0,1}$ assuming that $\textsf{V = L}$, where $\textsf{L}$ is G\"odel's universe of constructible sets. This is not a problem, since the model $\mathcal M$ we construct also is a model in the full set-theoretical universe $\textsf{V}$. However, this means that when we write $S^2_k$ in this proof, we really mean the relativised version  $(S^2_k)^{\textsf{L}}$. The advantage is that due to the $\Delta^1_2$-well-ordering of $\N^\N$ in \textsf{L}, if a set $A \subseteq \N^\N$ is closed under computability relative to all $S^2_k$, all $\Pi^1_k$-sets are absolute for $(A,{\textsf{L}})$ and hence  $(S^2_k)^A$ is a sub-functional of  $(S^2_k)^{\textsf{L}}$ for each $k$. We now drop the superscript `\textsf{L}' for the rest of the proof.
Put $S^2_\omega(k,f) := S^2_k(f)$ and note that $S^2_\omega$ is a normal functional in which all $S^2_k$ are computable.
Let ${\mathcal M} = {\mathcal M}^{S^2_\omega}$ be as in Definition \ref{nikeh}. 
This model is as requested by Theorem \ref{horniii}, i.e.\ $\neg\PIT_{\u}$ holds.  
\end{proof}
The model $\mathcal{M}$ can be used to show that many classical theorems based on uncountable data cannot be proved in any system $\SIXK+ \QFAC^{0,1}$,
e.g.\ the Vitali covering lemma and the uniform Heine theorem from Section~\ref{heineke}. 

\smallskip

Finally, we show that $\PIT_{o}$ is much easier to prove than $\PIT_{\u}$.
By contrast, weak Pincherle realisers, i.e.\ \emph{realisers} for $\PIT_{o}$, are not computably by any type two functional as established in Section \ref{PRR2}.  
\begin{thm}\label{ofinterest}
The system $\ACAo+\QFAC^{0,1}$ proves $\PIT_{o}$.  
\end{thm}
\begin{proof}
Recall that $\ACA_{0}$ is equivalent to various convergence theorems by \cite{simpson2}*{III.2}, i.e.\ $\ACAo$ proves that a sequence in Cantor space has a convergent subsequence.  
Now let $F, G$ be such that $\LOC(F, G)$ and suppose $F$ is unbounded, i.e.\ $(\forall n^{0})(\exists \alpha\leq 1)(F(\alpha)>n)$.  Applying $\QFAC^{0,1}$, we get a sequence $\alpha_{n}$ in Cantor space such that $(\forall n^{0})(F(\alpha_{n})>n)$.
By the previous, the sequence $\alpha_{n}$ has a convergent subsequence, say with limit $\beta\leq_{1}1$.  By assumption, $F$ is bounded by $G(\beta)$ in $[\overline{\beta}G(\beta)]$, which contradicts the fact that $F(\alpha_{n})$ becomes 
arbitrarily large close enough to $\beta$.   
\end{proof}
We show that $\PIT_{o}\asa \WKL$ in Corollary \ref{ofmoreinterest}.  On one hand, for conceptual reasons\footnote{The $\ECF$-translation is discussed in the context of $\RCAo$ in \cite{kohlenbach2}*{\S2}.  
Intuitively speaking, $\ECF$ replaces any object of type two or higher by an RM-code. 
%BABA
Applying $\ECF$ to $\PIT_{o}$, we obtain a sentence equivalent to $\WKL_{0}$, and hence $\PIT_{o}$ has the first-order strength of $\WKL$.\label{ecf}}, $\PIT_{o}$ cannot be stronger than $\WKL$ in terms of first-order strength.  On the other hand, reflection upon the previous proof suggests that any proof of $\PIT_{o}$ has to involve $\ACAo$.  Thus, the aforementioned equivalence is surprising.    

\subsection{Pincherle's theorem and uncountable Heine-Borel}
We establish that (versions of) Pincherle's theorem and the Heine-Borel theorem are equivalent.  % starting with the following. 
\begin{thm}\label{krooi} 
The system $\ACA_{0}^{\omega}+\QFAC^{2,1}$ proves 
\be\label{worski}
\HBU_{\c}\asa \HBU\asa (\exists \Theta)\SCF(\Theta)\asa \PIT_{\u}\asa (\exists M)\PR(M).
\ee
\end{thm}
\begin{proof}
The first two equivalences in \eqref{worski} are in \cite{dagsamIII}*{Theorem 3.3}, while $\HBU_{\c}\di \PIT_{\u}$ may be found in Theorem \ref{mooi}. 
By \cite{dagsamIII}*{\S2.3}, $\Theta$ as in $\SCF(\Theta)$ computes a finite sub-covering on input an open covering of Cantor space (given by a type two functional); hence $(\exists \Theta)\SCF(\Theta)\di (\exists M)\PR(M)$ follows in the same was as for $\HBU_{\c}\di \PIT_{\u}$ in the proof of Theorem \ref{mooi}.  
Finally, $(\exists M)\PR(M)\di \PIT_{\u}$ is trivial, and we now prove the remaining implication $\PIT_{\u}\di \HBU_{\c}$ in $\ACAo+\QFAC^{1,1}$.  
To this end, fix $G^{2}$ and let $N_{0}$ be the bound from $\PIT_{\u}$.  
We claim:
\be\label{contrje}
(\forall f^{1}\leq 1)(\exists g^{1}\leq 1)(G(g)\leq N_{0}\wedge f\in[ \overline{g}G(g)]).
\ee
Indeed, suppose $\neg\eqref{contrje}$ and let $f_{0}$ be such that $(\forall g\leq 1)( f_{0}\in [\overline{g}G(g)]\di G(g)> N_{0})$.
Now use $(\exists^{2})$ to define $F_{0}^{2}$ as follows: $F_{0}(h):=N_{0}+1$ if $h=_{1}f_{0}$, and zero otherwise.   
By assumption, we have $\LOC(F_{0}, G)$, but clearly $F(f_{0})>N_{0}$ and $\PIT_{\u}$ yields a contradiction.  
Hence, $\PIT_{\u}$ implies \eqref{contrje}, and the latter provides a finite sub-covering for the canonical covering $\cup_{f\leq 1}[\overline{f}G(f)]$.  
Indeed, apply $\QFAC^{1,1}$ to \eqref{contrje} to obtain a functional $\Phi^{1\di 1}$ providing $g$ in terms of $f$.  
The finite sub-covering (of length $2^{N_{0}}$) then consists of all $\Phi(\sigma*00\dots)$ for all binary $\sigma$ of length $N_{0}$.  
\end{proof}
By the previous proof, a Pincherle realiser $M$ provides an upper bound, namely $2^{M(G)}$, for the size of the finite sub-covering of the canonical covering of $G$, but the contents 
of that covering is not provided (explicitly) in terms of $M$. 
This observed difference between special fan functionals and Pincherle realisers also supports the conjecture that $\Theta$-functionals are not computable in any PR as in Conjecture \ref{corkes}.

\smallskip

The previous theorem is of historical interest: Hildebrandt discusses the history of the Heine-Borel theorem in \cite{wildehilde} and qualifies Pincherle's theorem as follows.
\begin{quote}
Another result carrying within it the germs of the Borel Theorem is due to S. Pincherle [\dots] (\cite{wildehilde}*{p.\ 424})
\end{quote}
The previous theorem provides evidence for Hildebrandt's claim, while the following two corollaries provide a better result, for $\PIT_{\u}$ and $\PIT_{o}$ respectively.  
Note that the base theory is conservative over $\RCA_{0}$ by the proof of \cite{kohlenbach2}*{Prop.\ 3.1}.
\begin{cor}\label{eessje}
The system $\RCAo+\QFAC^{0,1}$ proves $\PIT_{\u}\asa \HBU_{\c}\asa \HBU$.
\end{cor}
\begin{proof}
The reverse implications are immediate (over $\RCAo$) from Theorem~\ref{mooi}.  For the first forward implication, $\PIT_{o}$ readily implies $\WKL$ as follows: If a tree $T\leq_{1}1$ has no path, i.e.\ $(\forall f\leq 1)(\exists n)(\overline{f}n\not \in T)$, then using quantifier-free induction and $\QFAC^{1,0}$, there 
is $H^{2}$ such that $(\forall f\leq 1)(\overline{f}H(f)\not \in T)$ and $H(f)$ is the least such number.  Clearly $H^{2}$ is continuous on Cantor space and has itself as a modulus of continuity.   
Hence, $H^{2}$ is also locally bounded, with itself as a realiser for this fact.  By $\PIT_{o}$, $H$ is bounded on Cantor space, which yields that $T\leq_{1} 1$ is finite.   

\smallskip

Secondly, if we have $(\exists^{2})$, then the (final part of the) proof of Theorem \ref{krooi} goes through by applying $\QFAC^{0,1}$ to 
\be\label{contrje2}
(\forall \sigma^{0^{*}}\leq 1)(\exists g^{1}\leq 1)\big(|\sigma|=N_{0} \di (G(g)\leq N_{0}\wedge \sigma\in[ \overline{g}G(g)])\big).
\ee
rather than using \eqref{contrje}.  
On the other hand, if we have $\neg(\exists^{2})$, then \cite{kohlenbach2}*{Prop.~3.7} yields that all $G^{2}$ are continuous on Baire space.  
Since $\WKL$ is given, \cite{kohlenbach4}*{4.10} implies that all $G^{2}$ are uniformly continuous on Cantor space, and hence have an upper bound there.  The latter immediately 
provides a finite sub-covering for the canonical covering of $G^{2}$, and $\HBU_{\c}$ follows.  Since we are working with classical logic, we may conclude $\HBU_{\c}$ by invoking the law of excluded middle $(\exists^{2})\vee \neg(\exists^{2})$.

\smallskip

Finally, $\HBU_{c}\di \HBU$ was proved over $\RCAo$ in \cite{dagsamIII}*{Theorem 3.3}. 
\end{proof}
The previous proof suggests that the RM of $\HBU$ is rather robust: given a theorem $\T$ such that $[\T+(\exists^{2})]\di \HBU\di \T\di \WKL$, we `automatically' obtain $\HBU\asa \T$ over the same base theory, using the previous `excluded middle trick'. 

\smallskip

Note that $\QFAC^{0,1}$ is interesting in its own right as it is exactly what is needed to prove the \emph{pointwise} equivalence
between epsilon-delta and sequential continuity for Polish spaces, i.e.\ $\textsf{ZF}$ alone does not suffice (see \cite{kohlenbach2}*{Rem.\ 3.13}). 
Nonetheless, we can do without $\QFAC^{0,1}$ in Corollary \ref{eessje}, which is perhaps surprising in light of the differences between special fan functionals and PRs.  On the other hand, the proof proceeds by contradiction and lacks computational content.  
We also need \textsf{IND}, the induction axiom for all formulas in the language of higher-order arithmetic; the following base theory has the first-order strength of Peano arithmetic.    
\begin{cor}\label{eessjelk}
The system $\RCAo+\textsf{\textup{IND}}$ proves $\PIT_{\u}\asa \HBU_{\c}$.
\end{cor}
\begin{proof}
We only need to prove the forward direction.  In case $\neg(\exists^{2})$, proceed as in the proof of Corollary \ref{eessje}.  In case $(\exists^{2})$, fix $G^{2}$ and prove \eqref{contrje} as in the proof of the theorem. 
Let $\sigma_1 , \dots ,\sigma_{2^{N_{0}+1 }}$ enumerate all binary sequences of length $N_0 + 1$ and define  $f_i := \sigma_i*00\dots$ for $i\leq 2^{N_{0}+1}$.
Intuitively speaking, we now apply \eqref{contrje} for $f_i$ and obtain $g_i$ for each $i\leq 2^{N_{0}+1}$.  Then $\langle g_1 , \ldots , g_{2^{N_0 + 1}}\rangle$ provides the finite sub-covering for $G$.
Formally, it is well-known that $\ZF$ proves the `finite' axiom of choice via mathematical induction (see e.g.\ \cite{tournedous}*{Ch.\ IV}).  Similarly, one uses $\textsf{IND}$ to prove the existence of the
aforementioned finite sequence based on \eqref{contrje}.
\end{proof}
The following corollary also follows via the above `excluded middle trick'.  
%BABA 
Note that the $\ECF$-translation mentioned in Footnote \ref{ecf} converts both the previous and the following result to the equivalence between $\WKL$ and the boundedness of continuous functions on $C$ (see \cite{simpson2}*{IV}).
\begin{cor}\label{ofmoreinterest}
The system $\RCAo+\QFAC^{0,1}$ proves $\WKL\asa\PIT_{o}$.  
\end{cor}
\begin{proof}
The reverse direction is immediate by the first part of the proof of Corollary~\ref{eessje}.  
For the forward direction, working in $\RCAo+\QFAC^{0,1}+\WKL$, first assume $(\exists^{2})$ and note that Theorem \ref{ofinterest} yields $\PIT_{o}$ in this case. 
Secondly, again working in $\RCAo+\QFAC^{0,1}+\WKL$, assume $\neg(\exists^{2})$ and note that all functions on Baire space are continuous by \cite{kohlenbach2}*{Prop.\ 3.7}.  
Hence, $\HBU_{\c}$ just follows from $\WKL$ as the latter suffices to prove that a continuous function is uniformly continuous (and hence bounded) on Cantor space (\cite{kohlenbach4}*{Prop.~4.10}).  
By Theorem \ref{mooi}, we obtain $\PIT_{\u}$, and hence $\PIT_{o}$.
The law of excluded middle $(\exists^{2})\vee \neg (\exists^{2})$ now yields the forward direction, and we are done.    
\end{proof}
Let $\PIT_{o}'$ be $\PIT_{o}$ without a realiser for local boundedness, i.e.\ as follows:
\[
(\forall F^{2}) \big[ \underline{(\forall f \in C)(\exists n^{0})(\forall g\in C)\big[ g\in [\overline{f}n] \di F(g)\leq n    \big]  } \di (\exists m^{0})(\forall h \in C)(F(h)\leq m)  \big].
\]
By the previous proof, we have $\PIT_{o}\asa \PIT_{o}'\asa \WKL$ over $\RCAo+\QFAC^{0,1}$.  

\smallskip

It is a natural question (see e.g.\ \cite{montahue}*{\S6.1.1}) whether countable choice is essential in the above.  
The following theorem provides the strongest possible answer.  Motivated by this result, a detailed RM-study of $\QFAC^{0,1}$ may be found in Section~\ref{not2bad}.  
\begin{thm}\label{dagsamsyn}
There is a model of $\Z_{2}^{\omega}$ in which $\neg\PIT_{o}$ holds.  
\end{thm}
\begin{proof}
Our starting point is the proof of \cite{dagsamIII}*{Theorem 3.4}; the latter states that there is a type structure $\mathcal{M}$ satisfying $\Z_{2}^{\omega}$ in which $\HBU_{\c}$ is false. 
To this end, a specific functional $F^{2}$ in $\mathcal{M}$ is identified (see Definition \ref{effkes} below), and the latter is shown to have the following property: for any total extension $G$ of $F$, the canonical covering associated to $G$ does not have a finite sub-covering, i.e.\ $\HBU_{\c}$ is clearly false in $\mathcal{M}$.  
To obtain the theorem, we shall define another functional $H^{2}$ such that in the associated model we have $\LOC(H, F)$ but $H$ is unbounded on Cantor space.  % i.e.\ the theorem follows. 

\smallskip

First of all, let $A \subseteq \N^\N$ be a countable set such that all $\Pi^1_k$-statements with parameters from $A$ are absolute for $A$.
Also, let $S^2_k$ be the characteristic function of a complete $\Pi^1_k$-set for each $k$; we also write $S^2_k$ for the restriction of this functional to $A$. 
Clearly, for $f \in \N^\N$ computable in any $S^2_k$ and some $f_1 , \ldots , f_n$ from $A$, $f$ is also in $A$.   We now introduce some essential notation.   
\begin{convention}\label{connie}\rm
Since $A$ is countable, we write it as an increasing union $ \bigcup_{k \in \N}A_n$, where $A_0$ consists of the hyperarithmetical functions, and for $k > 0$ we have:
\begin{itemize}
\item there is an element in $A_k$ enumerating $A_{k-1}$,
\item $A_k$ is the closure of a finite set $g_1 , \ldots , g_{n_k}$ under computability in $S^2_k$.
\end{itemize}
For the sake of uniform terminology, we rename $\exists^2$ to $S^2_0$ and let the associated finite sequence $g_1 , \ldots , g_{n_0}$ be the empty list. 
\end{convention}
Secondly, the following functional $F$ was first defined in \cite{dagsamIII}*{Def.\ 3.6}.
\bdefi[The functional $F$]\label{effkes}
Define $F(f)$ for $f \in A$ as follows.
\begin{itemize} 
\item If $f \not \in 2^\N$, put $F(f) := 0$.
\item If $f \in 2^{\N}$, let $k$ be minimal such that $f \in A_k$. We put $F(f): = 2^{k+2+e}.$ where $e$ is a `minimal' index for computing $f$ from $S^2_k$ and $\{g_1 , \ldots , g_{n_k}\}$ as follows:
the ordinal rank of this computation of $f$ is minimal and $e$ is minimal among the indices for $f$ of the same ordinal rank.
\end{itemize}
\edefi
Note that the only property we need of $F$ in the proof of \cite{dagsamIII}*{Lemma 3.7.(ii)} is that the restriction of $G$ to $A_k$ is computable in $g_{1}, \dots, g_{n_{k}}$ and $S_{k}^{2}$.
We shall make use of the fact that $\sum_{k,e \in \N}2^{-(k + 2 + e)} = 1$, which is also used in the proof in \cite{dagsamIII} that the covering induced by $F$ cannot have a finite sub-covering .

\smallskip

Thirdly, we consider the type two functional $H'(f) := (\mu k)(f \in A_k)$ defined on $A \cap C$. We will not necessarily have that $\LOC(H',F)$, but will see that a minor modification $H$ will  both be unbounded and satisfy $\LOC(H,F)$.
Let  $f \in A_k \setminus A_{k-1}$, and assume that there is some $f'$ such that $H'(f) > F(f')$ while $f(m) = f'(m)$ for all $m < F(f')$, that is, $H'$ does not satisfy the bounding condition $\LOC(H', F)$ at $f$ as induced by  $F(f')$. With this choice of $f$, $H'(f) = k$, and there is a pair $i,e$ with $2^{i + 2 + e} < k$ such that $e$ is an index for computing  $f'$ from $S^2_i$ and $g_1 , \ldots , g_{n_i}$ and such that $f(m) = f'(m)$ for $m < 2^{i + 2 + e}$. For each $k$, the number of such pairs $(i,e)$ is finite and we have that $i < k$ in this case. We now define $H(f) := H'(f)$ if $\LOC(H', F)$ is satisfied at $f$ in the sense described above, and otherwise $H(f): = 0$.

\smallskip

Since $i < k$ for the relevant pairs $(i,e)$, we can \emph{decide} from $S^2_k$ and $g_1 , \ldots , g_{n_k}$ if $(i,e)$ is relevant, i.e. if $e$ is an index  for some $f'$ in $A_i \setminus A_{i-1}$ from $S^2_i$ and $g_1 , \ldots , g_{n_i}$.  
In this way, $H$ restricted to $A_k$ is computable  in $S^2_k$ and $g_1 , \ldots , g_{n_k}$. Moreover, the union of the finite sets of neighbourhoods where $F$ bounds a function to a value at most $ k$, is a clopen set of measure strictly below $ 1$.  
Hence, there are elements $f$ in $A_k$ for which $H(f) = k$, namely every base element contains an element in $A_k \setminus A_{k-1}$. Thus $H$ is unbounded but satisfies $\LOC(H, F)$ by construction.

\smallskip

Finally, all functions computable from $F$, $H$ and some $S^2_k$ and $f_1 , \ldots f_n$ from $A$ will be in $A$. So, we can construct one model for $\neg \PIT_{o}$ satisfying $\SIXk$ for all $k$. This model will obviously not satisfy $\QFAC^{0,1}$.
\end{proof}
\begin{cor}\label{hantonio}
A proof of $\PIT_{o}$ either requires $\WKL+\QFAC^{0,1}$ or $\Z_{2}^{\Omega}$, i.e.\ $\Z_{2}^{\omega}$ does not suffice. 
\end{cor}
The previous corollary does not change if we generalise $\PIT_{o}$ to $\PIT_{o}'$.
Indeed, the underlined formula in $\PIT_{o}'$ expresses that $F$ is locally bounded (without a realiser) and $\Z_{2}^{\Omega}$ readily proves that there is $G^{2}$ such that $\LOC(F, G)$.  

\smallskip

The theorem also gives rise to the following `disjunction' in which the two disjuncts are independent.  
Many similar results may be found in \cite{samsplit}, despite being rare in second-order arithmetic. 
The following is a corollary to Corollary \ref{eessje}.
\begin{cor}
The system $\RCAo$ proves $\WKL_{0}\asa [\ACA_{0}\vee \PIT_{o}]$.
\end{cor}
\begin{proof}
The reverse direction is immediate, while the forward direction follows from considering $(\exists^{2})\vee \neg(\exists^{2})$ as in the proof of Corollary \ref{eessje}.
Note that in case $\neg(\exists^{2})$, no choice is used in this proof of the latter corollary.  
\end{proof}
Based on the known classifications of compactness from \cite{simpson2}*{III and IV}, the previous disjunction reads as: countable compactness is equivalent to `sequential compactness or local-to-global'. 

\smallskip

One further improvement of Theorem \ref{krooi} is possible, using the \emph{fan functional} as in \eqref{FF}, where `$Y^{2}\in \textsf{cont}$' means that $Y$ is continuous on $\N^{\N}$.
\be\tag{$\textsf{\textup{FF}}$}\label{FF}
(\exists \Phi^{3})(\forall Y^{2}\in \textsf{\textup{cont}})(\forall f, g\in C)(\overline{f}\Phi(Y)=\overline{g}\Phi(Y)\di Y(f)=Y(g)).
\ee
Note that the previous two corollaries only dealt with third-order objects, while the following corollary \emph{connects} third and fourth-order objects.
\begin{cor}\label{foepfoep}
The system $\RCAo+\textsf{\textup{FF}}+\QFAC^{2,1}$ proves $\HBU_{\c}\asa (\exists \Theta)\SCF(\Theta)$.
\end{cor}
\begin{proof}
We only need to prove the forward implication. 
Working in $\RCAo+\textsf{\textup{FF}}$, assume $(\exists^{2})$ and note that the forward implication follows from Theorem \ref{krooi}.    
In case of $\neg(\exists^{2})$, all functions on Baire space are continuous by \cite{kohlenbach2}*{Prop.\ 3.7}.  
Hence, $\Phi(Y)$ from $\textsf{FF}$ provides a modulus of uniform continuity for \emph{any} $Y^{2}$.
A special fan functional $\Theta$ is then defined as outputting the finite sequence of length $2^{\Phi(Y)}$ consisting of all sequences $\sigma*00$ for binary $\sigma$ of length $\Phi(Y)$.
\end{proof}
The base theory in the previous corollary is a (classical) conservative extension of $\WKL_{0}$ by \cite{kohlenbach2}*{Prop.\ 3.15}, which is a substantial improvement over the base theory $\ACAo$ from \cite{dagsamIII}*{Theorem 3.3}.  
One similarly proves $ \PIT_{\u}\asa (\exists M)\PR(M)$.  % over the same base theory.

\smallskip

Next, we use the above `excluded middle trick' in the context of the axiom of extensionality on $\R$.  
 \begin{rem}[Real extensionality]\label{kloti}\rm
 The trick from the previous proofs involving $(\exists^{2})\vee\neg(\exists^{2})$ has another interesting application, namely that $\HBU$ does not really change if we drop the extensionality condition \eqref{RE} from Definition \ref{keepintireal} for $\Psi^{1\di 1}$.  In particular, $\RCAo$ proves $\HBU\asa \HBU^{+}$, where the latter is $\HBU$ generalised to \emph{any} functional $\Psi^{1\di 1}$ such that $\Psi(f)$ is a positive real, i.e.\ $\Psi^{1\di 1}$ need not satisfy \eqref{RE}.  To prove $\HBU\di \HBU^{+}$, note that $(\exists^{2})$ yields a functional $\xi$ which converts $x\in [0,1]$ to a unique binary representation $\xi(x)$, choosing $\sigma*00\dots$ if $x$ has two binary representations; then $\lambda x.\Psi(\r(\xi(x)))$ with $\r(\alpha):=\sum_{n=0}^{\infty}\frac{\alpha(n)}{2^{n}}$ satisfies \eqref{RE} restricted to $[0,1]$, even if $\Psi^{1\di 1}$ does not, and we have $\HBU\di \HBU^{+}$ assuming $(\exists^{2})$.  In case of $\neg(\exists^{2})$, all functionals on Baire space are continuous by \cite{kohlenbach2}*{Prop~3.7}, and $\HBU\di \WKL$ yields that all functions on Cantor space are uniformly continuous (and hence bounded).  
Now, consider $\Psi$ as in $\HBU^{+}$ and note that for $\lambda\alpha.\Psi(\r(\alpha))$ there is $n_{0}\in \N$ such that $(\forall \alpha \in C)(\Psi(\r(\alpha))> \frac{1}{2^{n_{0}}})$.  Hence, the canonical covering of $\Psi$ has a finite sub-covering consisting of $\r(\sigma_{i}*00\dots)$ where $\sigma_{i}$ is the $i$-th binary sequence of length $n_{0}+1$.  
 i.e.\ $\HBU\di \HBU^{+}$ follows in this case. % and the law of excluded middle finishes this proof.  
 \end{rem}
 Next, we discuss a number of variations and generalisations of Pincherle's theorem, and a local-global principle due to Weierstrass, with similar properties.
 %BABA   
\begin{rem}[Other local-to-global principles]\label{kowlk}\rm
First of all, Pincherle describes the following theorem in a footnote on \cite{tepelpinch}*{p.\ 67}: 
\begin{quote}
\emph{Let $E$ be a closed, bounded subset of $\mathbb{R}^{n}$ and let $f : E \di \R$ be locally bounded away from $0$.  Then $f$  has a positive infimum on $E$}.
\end{quote}
He states that this theorem is proved in the same way as Theorem \ref{gem} and provides a generalisation of Heine's theorem as proved by Dini in \cite{dinipi}.  
We could formulate versions of the centred theorem, and they would be equivalent to the associated versions of Pincherle's theorem.    
Restricted to \emph{uniformly continuous} functions, the centred theorem is studied in \emph{constructive} RM (\cite{brich}*{Ch.\ 6}).  
Recall from Remark~\ref{kiekenkkk} that Pincherle works with \emph{possibly discontinuous} functions.  % not just continuous ones.    

\smallskip
\noindent
Secondly, the following theorem generalises $\PIT_{o}$ to Baire space: 
\be\tag{$\PIT_{b}$}
(\forall F^{2}, G^{2})\big[  \LOC_{b}(F, G)\di (\exists H^{2} \in \textsf{cont})(\forall g^{1})(F(g)\leq H(g))\big],
\ee
where $\LOC_{b}(F, G)$ is $\LOC(F, G)$ with Cantor space replaced by Baire space.  Clearly, $\PIT_{b}+\WKL$ implies $\PIT_{o}$, while $\PIT_{b}$ has no first-order strength in light of the $\ECF$-translation.  
Hence, $\PIT_{b}$ is at least `as hard' to prove as $\PIT_{o}$, while the Lindel\"of lemma for Baire space, combined with $(\exists^{2})$, proves $\PIT_{b}$.  We could also let $H^{2}$ in $\PIT_{b}$ be a step function, similar to the majorants in the Lebesgue integral.

\smallskip

Thirdly, the first proof\footnote{Weierstrass does not enunciate his result in full detail in \cite{amaimennewekker}*{p.\ 202}.  The proof however does apply to \emph{any} locally uniformly convergent series of functions, and his definition of (local) uniform convergence is also fully general in that no continuity is mentioned.} in Weierstrass' 1880 paper `Zur Funktionenlehre' (see \cite{amaimennewekker}*{p.\ 203}) establishes the following local-to-global principle: \emph{any locally uniformly convergent function on an interval is uniformly convergent}.
Local uniform convergence is uniform convergence in a neighbourhood of every point of the interval.  One readily shows that this local-to-global principle \emph{for Cantor space} is equivalent to $\PIT_{o}$ over $\RCAo$.    
In general, Weierstrass tended to lecture extensively (in Berlin)  on his research before (eventually) publishing it.  Pincherle actually attended Weierstrass' lectures around 1878 and published an overview in \cite{pinkersgebruiken}.   

\smallskip

Fourth, the following principle \emph{looks} similar to $\QFAC^{0,1}$ restricted to $C$, but does not provide a choice functional for zeros of $Y$: only a shrinking neighbourhood in which zeros of $Y$ \emph{can} be found is given.  For $Y(f, n)$ increasing in $n$, we have
\[
(\forall n^{0})(\exists f\in C)(Y(f, n)=0) \di (\exists g\in C)(\forall n^{0})(\exists f\in [\overline{g}n])(Y(f, n)=0).
\]
One readily proves this sentence in either $\RCAo+\ACA_{0}+\QFAC^{0,1}$ or $\Z_{2}^{\Omega}$.  However, taking $Y(f, n)=0$ to be the characteristic function of $F(f)\geq n$, one also immediately obtains $\PIT_{0}'$.  Thus, this principle behaves similarly to $\PIT_{o}$ and its kin. 
We note that the above principle is inspired by the Bolzano-Weierstrass theorem in the introduction of Section \ref{pproof}.
\end{rem}
 Finally, the fourth item in the Remark \ref{kowlk} gives rise to the `local-to-global principle' $\LGP$, which is a generalisation of $\PIT_{o}$ and implies $\HBU$ by Theorem~\ref{koknk}.
\bdefi[$\LGP$]
For any $\L_{2}$-formula $A(f, n)$ \emph{with type two parameters}:
\[
 (\forall f\in C)(\exists n^{0})(\forall g\in [\overline{f}n])A(g, n)  \di (\exists m^{0})(\forall h\in C)A(h,m) \big], 
\]
where $A(f,n)$ is `increasing in $n^{0}$', i.e.\ $(\forall f^{1}, n^{0}, m^{0})(n\leq m \di [A(f, m)\di A(f, n)])$.
\edefi
Note that $\LGP$ trivially implies $\PIT_{o}'$.  
We also have the following theorem.
\begin{thm}\label{koknk}
 The system $\RCAo+\LGP+\textsf{\textup{IND}}$ proves $\HBU$.
\end{thm} 
\begin{proof}
Fix $G^{2}$ and let $A(g, n)$ be the following `increasing' formula
\[
(\exists h\in C)( G(g)\leq n\wedge (g\in [\overline{h}G(h)])  ).
\]
Note that the antecedent of $\LGP$ holds and obtain $\HBU$ as for Corollary \ref{eessjelk}.
\end{proof}
Let $\LGP^{-}$ be the principle $\LGP$ with $A$ restricted to $\L_{2}$-formulas, i.e.\ type two parameters are not allowed in $A$.
\begin{cor}\label{koknk2}
 The system $\Z_{2}^{\omega}+\QFAC^{0,1}$ proves $\LGP^{-}$, and hence the latter does not imply $\HBU$ over $\RCAo+\textup{\textsf{IND}}$.
\end{cor} 
\begin{proof}
The second part follows from the first part and Theorems \ref{mooi} and \ref{mooier}.  For the first part, replace $A(f, n)$ by $Y(f, n)=0$, where the latter is defined in terms of $S_{k}^{2}$. 
Apply $\QFAC^{0,1}$ to $(\forall n^{0})(\exists h\in C)(Y(h,n)>0)$ to find a sequence $h_{n}$ in $C$ such that $(\forall n^{0})(Y(h_{n}, n)>0)$.  By $(\exists^{2})$, the sequence $h_{n}$ has a convergent sub-sequence, say to $f_{0}\in C$, and 
$ (\forall  n^{0})(\exists g\in [\overline{f_{0}}n])Y(g, n) $ now readily follows. 
\end{proof}
 
\subsection{Subcontinuity and Pincherle's theorem}\label{pitche}
We study an equivalent version of Pincherle's theorem based on an existing notion of continuity, called \emph{sub-continuity}.  
As it happens, sub-continuity is actually used in (applied) mathematics in various contexts: see e.g.\ \cite{gordon3}*{\S4.7}, \cite{migda}*{\S14.2}, \cite{mizera}*{p.\ 318}, and \cite{lola}*{\S4}.  

\smallskip

First of all, in a rather general setting, local boundedness is equivalent to the notion of \emph{sub-continuity}, introduced by Fuller in \cite{voller}.  
The equivalence between sub-continuity and local boundedness (for first-countable Haussdorf spaces $X$ and functions $f:X\di \R$) may be found in \cite{roykes}*{p.\ 252}.  For the purposes of this paper, we restrict ourselves to $I\equiv [0,1]$, which simplifies the definition as follows.  
\bdefi[Subcontinuity]
A function $f:\R\di \R$ is \emph{sub-continuous on $I$} if for any sequence $x_{n}$ in $I$ convergent to $ x\in I$, $f(x_{n})$ has a convergent subsequence.  
\edefi
Secondly, the equivalence between sub-continuity and local boundedness (without realisers) can then be proved as in Theorem \ref{forgu}.  The weak base theory in the latter constitutes a surprise: sub-continuity 
has a typical `sequential compactness' flavour, while local boundedness has a typical `open-cover compactness' flavour.  The former and the latter are classified in the RM of resp.\ $\ACA_{0}$ and $\WKL$ ($\HBU$).      
\begin{thm}\label{forgu}
The system $\RCAo+\QFAC^{0,1}$ proves that a function $f:\R\di \R$ is locally bounded on $I$ if and only if it is sub-continuous on $I$.   Only the reverse direction uses countable choice. 
\end{thm}
\begin{proof}
We establish the equivalence in $\RCAo+\QFAC^{0,1}$ in two steps: first we prove it assuming $(\exists^{2})$ and then prove it again assuming $\neg(\exists^{2})$.  
The law of excluded middle as in $(\exists^{2})\vee \neg(\exists^{2})$ then yields the theorem.  

\smallskip

Hence, assume $(\exists^{2})$ and suppose $f:\R\di \R$ is sub-continuous on $I$ but not locally bounded.  The latter assumption implies that there is $x_{0}\in I$ such that 
\be\label{contrake}\textstyle
{(\forall n^{0})(\exists x\in I)(|x-x_{0}|<_{\R}\frac{1}{n+1}\wedge |f(x)|>_{\R}n)}.  
\ee
Both conjuncts in \eqref{contrake} are $\Sigma_{1}^{0}$-formula, i.e.\ we may apply $\QFAC^{0,1}$ to \eqref{contrake} to obtain $\Phi^{0\di 1}$ such that for $y_{n}:=\Phi(n)$ and $x_{0}$ as in \eqref{contrake}, we have 
\be\label{missyoumuch}\textstyle
(\forall n\in \N)(|y_{n}-x_{0}|<_{\R}\frac{1}{n+1}\wedge |f(y_{n})|>_{\R}n), 
\ee
Clearly $y_{n}$ converges to $x_{0}$, and hence for some function $g:\N\di \N$, the subsequence $f(y_{g(n)})$ converges to some $y\in \R$ by the sub-continuity of $f$.  
However, $f(y_{g(n)})$ also grows arbitrarily large by \eqref{missyoumuch}, a contradiction, and the reverse implication follows.  

\smallskip

Next, again assume $(\exists^{2})$; for the forward implication, suppose $f$ is locally bounded and let $y_{n}$ be a sequence in $I$ convergent to $x_{0}\in I$.  
Then there is $k\in \N$ such that for all $y\in B(x_{0}, \frac{1}{2^k})$, $|f(y)|\leq k$.  However, for $n$ large enough, $y_{n}$ lies in $B(x, \frac{1}{2^{k}})$, implying that $|f(y_{n})|\leq k$ for $n$ large enough.  
In other words, the sequence $f(y_{n})$ eventually lies in the interval $[-k, k]$, and hence has a convergent subsequence by $(\exists^{2})$ and \cite{simpson2}*{I.9.3}.
Thus, $f$ is sub-continuous and we are done with the case $(\exists^{2})$.

\smallskip

Finally, in case that $\neg(\exists^{2})$, any function $f:\R\di \R$ is everywhere sequentially continuous and everywhere $\eps$-$\delta$-continuous by \cite{kohlenbach2}*{Prop.\ 3.12}.  
Hence, any $f:\R\di \R$ is also sub-continuous on $I$ and locally bounded on $I$, and the equivalence from the theorem is then trivially true.  
\end{proof}
Next, the statement \emph{a continuous function on $C$ is bounded} is equivalent to $\WKL_{0}$ by \cite{simpson2}*{IV.2.3}.   
The statement \emph{a sub-continuous function on $C$ is bounded} is a generalisation of the first statement, and also a special case of e.g.\ \cite{voller}*{Theorem~2.1} or \cite{systemenouveau}*{Theorem~2.4}.
Following Theorem \ref{forgu}, the second statement is a variation of Pincherle's theorem and denoted $\PIT_{o}''$; sub-continuity on $C$ is obvious.  
\begin{thm}
The system $\Z_{2}^{\omega}$ cannot prove $\PIT_{o}''$, while $\RCAo+\QFAC^{0,1}$ proves $\WKL\asa \PIT_{o}''$.
\end{thm}
\begin{proof}
Similar to the proof of Theorem \ref{forgu}, $\ACAo$ proves that a locally bounded function on $C$ is sub-continuous.  
Hence, $\PIT_{o}''\di \PIT_{o}$ without the use of countable choice, and the first part of the proof follows from Theorem \ref{dagsamsyn}.
For the second part, similar to the proof of Theorem \ref{forgu}, a sub-continuous function on $C$ is locally bounded, using $\QFAC^{0,1}$.  
Hence we have $\PIT_{o}\di \PIT_{o}''$, and $\PIT_{o}'' \di \WKL$ is immediate by \cite{simpson2}*{IV.2.3}.  Corollary \ref{ofmoreinterest} now finishes the proof.  
\end{proof}
As to similar statements, discontinuous functions can be unbounded, but a sub-continuous function on $\N^{\N}$ is bounded above by a continuous function, following Remark \ref{kowlk}.
By contrast, restricting $\PIT_{o}$ to quasi-continuous functions (see e.g.\ \cite{berserk}) yields a theorem provable from $\WKL$ alone.  

\smallskip

Finally, the definition of sub-continuity in \cite{migda,noiri} is based on sequences, while the (general topological) definition of sub-continuity in \cite{voller} is based on \emph{nets}, the generalisation of sequences to possibly uncountable index sets (see e.g.\ \cite{ooskelly}*{p.\ 65}).  
As studied in \cite{samnet}, the definition of sub-continuity involving nets implies local boundedness \emph{without the use of countable choice} as in $\QFAC^{0,1}$.
Moreover, $\PIT_{o}'$ formulated with nets is equivalent to $\WKL$ over $\RCAo+\QFAC^{0,1}$ as well.

\subsection{Heine's theorem, Fej\'er's theorem, and compactness}\label{myheinie}
We prove that the uniform versions of Heine's theorem from Section \ref{heineke} are equivalent to $\HBU_{\c}$.
We prove similar results for Heine's theorem \emph{for the unit interval} and the related \emph{Fej\'er's theorem}.
The latter states that for continuous $f:\R\di \R$, the Ces\`aro mean of the partial sums of the Fourier series uniformly converges to $f$. 
Throughout this section, the use of $\QFAC^{0,1}$ can be removed in favour of $\textsf{IND}$ as in the proof of Corollary \ref{eessjelk}.
The following proof shows that the same type three functionals are realisers for $\PIT_{\u}$ and $\UCT_{\u}$.
\begin{thm}\label{roofer}
The system $\RCAo+\QFAC^{0,1}$ proves $\UCT_{\u}\asa \HBU_{\c}$
\end{thm}
\begin{proof}
The reverse implication is immediate.
The rest of the proof is based on that of $\PIT_{\u} \di \HBU_{\c}$ in Corollary \ref{eessje}: for fixed $G$, let $N$ be as in $\UCT_{\u}$. Then for $g^{1} \leq 1$ there is $f^{1} \leq 1$ such that $G(f) \leq N$ and $\bar g(G(f)) = \bar f(G(f))$.
Because, if this is not the case, there is a binary sequence $s$ of length $N$ such that for all $f$ extending $s$ we have that $G(f) > N$. Then we can define $F(f) = 0$ if $f$ does not extend $s$ and $F(f) = f(N)$ if $f$ extends $s$. Then $F$ has a modulus of continuity given by $G$, but not a modulus of uniform continuity given by $N$. 
\end{proof}
The previous results establish the equivalence between the uniform version of Heine's theorem and the Heine-Borel theorem for uncountable coverings, \emph{in the case of Cantor space}.  
One similarly proves the equivalence between the Heine-Borel theorem $\HBU$ and uniform Heine's theorem \emph{for the unit interval}, as follows.  
\begin{princ}[$\UCT_{\u}^{\R}$]
For any $\eps>_{\R}0$ and $g:(I\times \R)\di \R^{+}$, there is $\delta>_{\R}0$ such that for any $f:I\di \R$ with modulus of continuity $g$, we have
\[
(\forall x, y\in I)(|x-y|<_{\R}\delta)\di |f(x)-f(y)|<_{\R}\eps ).
\]
\end{princ}
We shall prove that $\UCT_{\u}^{\R}$ is equivalent to the uniform version of Fej\'er's theorem.  
We follow the approach in \cite{kohabil}*{p.\ 65} and we define $I_{\pi}\equiv[-\pi, \pi]$.
\bdefi\label{popolop}
Define $\sigma_{n}(f, x) :=\frac{1}{n} \sum_{k=0}^{n-1}S(k, f, x)$, where
$S(n,f,x):= \frac{a_{0}}{2} +  \sum_{k=1}^{n} ( a_{k} \cdot \cos(kx)+b_{k} \cdot \sin(kx) )$ and
$ a_{k} := \frac{1}{\pi}\int_{-\pi}^{\pi}   f (t)\cos(kt)dt, b_{k} :=\frac1\pi \int_{-\pi}^{\pi}   f (t) \sin(kt)dt$.
\edefi
Note that Fej\'er's theorem already deals with \emph{uniform} convergence, i.e.\ the notion of 
convergence in $\FEJ_{\u}$ below is `super-uniform' in that it only depends on the modulus of continuity for the function.  
\begin{princ}[$\FEJ_{\u}$]
For any $k\in\N$ and $g:(I_{\pi}\times \R)\di \R^{+}$, there is $N\in \N$ such that for any $f:I_{\pi}\di \R$ with modulus of continuity $g$ and $f(0)=0$, we have
\be\label{trop}\textstyle
(\forall n\geq N, x\in I_{\pi})(|\sigma_{n}(f, x)-f(x)|<\frac1k) \wedge (\forall y\in I_{\pi}, n\in \N)(|\sigma_{n}(f, y)|\leq nN).
\ee
\end{princ}
Note that functions like $\sin x$ and $e^{x}$ can be defined in $\RCAo$ by \cite{simpson2}*{II.6.5}, while $\WKL$ is needed to make sure $\sigma_{n}$ makes sense by \cite{simpson2}*{IV.2.7}.
\begin{thm}\label{fejerpejerpotloodgat}
The system $\RCAo+\WKL$ proves $\UCT_{\u}^{\R}\asa \FEJ_{\u}$.
\end{thm}
\begin{proof} 
For the forward implication, the modulus of uniform convergence $\Psi$ for Fej\'er's theorem from \cite{kohabil}*{p.\ 65} is
\[
\Psi(f, k):=48(k+1)\cdot \|f\|_{\infty} \cdot(\omega_{f}(2(k+1))+1)^{2},
\]
for a modulus of uniform continuity $\omega_{f}:\N\di \N$ for $f$.  Note that we can replace $\|f\|_{\infty}$ by $16\omega_{f}(1)$ if $f(0)=0$.  
Due to the high level of uniformity of $\Psi$, the first conjunct of \eqref{trop} immediately follows from $\UCT_{\u}^{\R}$.  For the the second conjunct of \eqref{trop}, fix $g$ and apply $\UCT_{\u}^{\R}$ for $\eps=1$ to obtain $\delta_{1}$ as in the latter.  
Now note that any $f$ such that $f(0)=0$ and $g$ is a modulus of continuity for $f$, we have $(\forall x\in I_{\pi})(|f(x)|\leq N)$ where $N=\lceil \frac{2\pi}{\delta_{1}}\rceil$.  Intuitively, this $N$ is a `uniform' bound for $f$ that only depends on a modulus 
of continuity for the latter.  By definition, this also yields a uniform bound for $\sigma_{n}(f, x)$ (in terms of $n$ and $N$ only). 

\smallskip

For the reverse implication, note that $\UCT_{\u}^{\R}$ does not change if we additionally require $f(0)=0$, since we can consider $f_{0}(x):=f(x)-f(0)$, which has the same modulus of continuity as $f$.  
Now fix $g$ as in $\UCT_{\u}^{\R}$, fix $x, y\in I_{\pi},\eps>0$ and consider 
\be\label{argo}
|f(x)-f(y)|\leq |f(x)-\sigma_{n}(f,x)|+|\sigma_{n}(f, x)-\sigma_{n}(f, y)|+|f(y)-\sigma_{n}(f, y)|
\ee
for $f$ with $g$ as modulus of continuity and $f(0)=0$.  
The first and third part of the sum in \eqref{argo} are both below $\eps/3$ for $n$ large enough.  Such number, with the required independence properties, is provided by $\FEJ_{\u}$.  
Moreover, $\sigma_{n}(f, x)$ is uniformly continuous on $I_{\pi}$ with a modulus which depends on $n$ but not $f$ due to the second conjunct of \eqref{trop}.  
Hence, \eqref{argo} implies that $f$ is uniformly continuous in the sense required by $\UCT_{\u}^{\R}$.
\end{proof}
Using a proof similar to that of Theorem \ref{roofer}, we obtain.  
\begin{cor}
 $\RCAo+\WKL+\QFAC^{0,1}$ proves $\UCT_{\u}^{\R}\asa \FEJ_{\u}\asa \HBU$.
\end{cor}
By the previous, realisers for $\FEJ_{\u}$ and  $\UCT_{\u}$ are equi-computable modulo $\exists^2$. When coded as functionals of pure type 3, these realisers are all essentially PRs, modulo a computable scaling.
Similarly, many theorems from analysis yield analogous uniform versions, and there are at least two sources: on one hand, as noted above, 
the redevelopment of analysis based on techniques from the gauge integral (as in e.g.\ \cite{bartle3}) yields uniform theorems.  
On the other hand, as hinted at in the proof of Theorem \ref{fejerpejerpotloodgat}, Kohlenbach's \emph{proof mining} program is known to produce highly uniform results (see e.g.\ \cite{kohlenbach3}*{Theorem~15.1}), which yield uniform versions, like $\FEJ_{\u}$ for Fej\'er's theorem.
We finish this section with some conceptual remarks. 
\begin{rem}[Atsuji spaces]\rm
A metric space $X$ is called \emph{Atsuji} if for any metric space $Y$, any continuous function $f:X\di Y$ is uniformly continuous.
The RM study of {Atsuji spaces} may be found in \cite{withgusto}*{\S4}, and one of the results is that the Heine-Borel theorem for countable coverings of $[0,1]$ is equivalent to the latter being Atsuji.
Theorem \ref{roofer} may be viewed as a generalisation (or refinement) establishing that $\HBU$ is equivalent to $[0,1]$ being `uniformly' Atsuji, i.e.\ as in $\UCT_{\u}$.
\end{rem}
\begin{rem}[Other uniform theorems]\label{flurki}\rm
It is possible to formulate uniform versions (akin to $\PIT_{\u}, \UCT_{\u}$, and $\FEJ_{\u}$) of many theorems.  
For reasons of space, we delegate the study of such theorems to a future publication.  
We point the reader to \cite{gauwdief}*{Example 2} and \cite{taokejes} for `real-world' examples using $\HBU$ by two Fields medallists. 
We also provide the example of \emph{uniform} weak K\"onig's lemma $\WKL_{\u}$:
\[
(\forall G^{2})(\exists m^{0})(\forall T\leq_{1}1)\big[(\forall \alpha \in C)(\overline{\alpha}G(\alpha)\not\in T)\di (\forall \beta \in C)( \overline{\beta}m\not\in T )     \big],
\]
Note that $\WKL_{\u}$ expresses that a binary tree $T$ is finite if it has no paths, \emph{and} the upper bound $m$ only depends on a realiser $G$ of `$T$ has no paths'.  
It is fairly easy to show that $\WKL_{\u}$ is equivalent to $\HBU$ by adapting the proof of Theorem \ref{eessje}.
\end{rem}
\begin{rem}[Finer than comprehension]\label{NFP}\rm
The results in \cite{dagsamIII} and this paper imply that e.g.\ $\HBU$ cannot be proved in $\Z_{2}^{\omega}+\QFAC^{0,1}$, while $\Z_{2}^{\Omega}$ suffices. 
Given the weak first-order strength of the former compared to the latter, it is an understatement to say that (higher-order) comprehension (embodied by $S_{k}^{2}$ and $(\exists^{3})$) does not capture open-cover compactness well, i.e.\ there is no natural fragment equivalent to $\HBU$.   
Similar claims can be made for the Lindel\"of property and related concepts.  
As it turns out, the \emph{neighbourhood function principle} $\NFP$ from \cite{troeleke1} give rise to a hierarchy that captures e.g.\ $\HBU$ and the Lindel\"of lemma well.  
Indeed, as shown in \cite{samph}*{\S4}, there are natural fragments of $\NFP$ that are equivalent to the latter. 
\end{rem}
\section{Some subtleties of higher-order arithmetic}\label{not2bad}
This section is devoted to the detailed study of certain, in our opinion, subtle aspects of the results obtained above and in \cite{dagsamIII}.  
Firstly, in light of our use of the axiom of (countable) choice in our RM-results, Section~\ref{finne} is devoted to the study of $\QFAC^{0,1}$, its tight connection to the  \emph{Lindel\"of lemma} in particular.
In turn, we show in Section \ref{kurzweil} that the Lindel\"of lemma \emph{for Baire space}, called $\LIND_{4}$ in \cite{dagsamIII}, together with $(\exists^{2})$, proves $\FIVE$, improving the results in \cite{dagsamIII}*{\S4}.  

\subsection{A finer analysis: the role of the axiom of choice}\label{finne}
Our above proofs often make use of the axiom of countable choice, and its status in RM is studied in this section.
We first discuss some required preliminaries in Section~\ref{introp}, and then study the tight connection between $\QFAC^{0,1}$ and the \emph{Lindel\"of lemma} in Section~\ref{rose}.  
We also show that the logical status of the latter is highly dependent on its formulation (provable in a weak fragment of $\Z_{2}^{\Omega}$ versus unprovable in $\ZF$).    

\subsubsection{Historical and mathematical context}\label{introp}
To appreciate the study of countable choice and the Lindel\"of lemma, some mathematical/historical facts are needed.  

\smallskip

First of all, many of the results proved above or in \cite{dagsamIII} make use of the axiom of choice, esp.\ $\QFAC^{0,1}$ in the base theory.  
Whether the axiom of choice is really necessary is then a natural RM-question (posed first by Hirshfeldt; see \cite{montahue}*{\S6.1}).  
Moreover, $\QFAC^{0,1}$ also figures in the grander scheme of things: e.g.\ the local equivalence of `epsilon-delta' and sequential continuity is not provable in \textsf{ZF} set theory, 
while $\QFAC^{0,1}$ yields the equivalence in a general context (\cite{kohlenbach2}*{Rem.\ 3.13}).   
Finally, countable choice for subsets of $\R$ is equivalent to the fact that $\R$ is a Lindel\"of space over $\ZF$ (\cite{heerlijk}).  Thus, the role of $\QFAC^{0,1}$ is connected to the status of the \emph{Lindel\"of property}, i.e.\ that every open covering has a countable sub-covering.

\smallskip

Secondly, the previous points give rise to a clear challenge: find a version of the Lindel\"of lemma equivalent to $\QFAC^{0,1}$, over $\RCAo$.  
An immediate difficulty is that the aforementioned results from \cite{heerlijk} are part of set theory, while the framework of RM is much more minimalist by design; for instance, what is a (general) open covering in $\RCAo$?
Fortunately, the pre-1900 work by Borel and Schoenflies on open-cover compactness provides us with a suitable starting point.  %as follows.   

\smallskip

Thirdly, we consider Lindel\"of's \emph{original} lemma from \cite{blindeloef}*{p.\ 698}.
\begin{quote}
Let $P$ be any set in $\R^{n}$ and construct for every point of $P$ a sphere $S_{P}$ with $x$ as center and radius $\rho_{P}$, where the latter can vary from point to point; there exists a countable infinity $P'$ of such spheres such that every point in $P$ is interior to at least one sphere in $P'$.  
\end{quote}
A similar formulation was used by Cousin in \cite{cousin1}.  
However, these coverings are `special' in that for $x\in \R^{n}$, one \emph{knows} the open set covering $x$, namely $B(x, \rho(x))$, similar to our notion of canonical covering.  % covering $x$. 
By contrast, a (general) open covering of $\R$ is such that for every $x\in \R$, there \emph{exists} a set in the covering containing $x$.  
This is the modern definition, and one finds its roots with Borel (\cite{opborrelen}) as early as 1895 (and in 1899 by Schoenflies), the same year Cousin published \emph{Cousin's lemma} (aka $\HBU$) in \cite{cousin1}.  

\smallskip

Motivated by the above, we shall study the Borel-Schoenflies formulation of the Lindel\"of lemma (and $\HBU$) in Section~\ref{rose}.  This version turns out to be equivalent to $\QFAC^{0,1}$ on the reals, and also provides further nice results.  

\subsubsection{A rose by many other names}\label{rose}
We formulate versions of the Heine-Borel theorem and Lindel\"of lemma based on the 1895 and 1899 work of Borel and Schoenflies on open-cover compactness (\cite{opborrelen, schoen2}).
These versions provide a nice classification involving $\QFAC^{0,1}$ and show that the logical status of the Lindel\"of lemma is highly dependent on its formulation (provable in weak fragments of $\Z_{2}^{\Omega}$ versus unprovable in $\ZF$).    
We note that Schoenfield in \cite{schoen2}*{Theorem V, p.~51} first reduces an uncountable covering to a countable sub-covering, and then to a finite sub-covering.  

\smallskip

For our purposes it suffices that open coverings are `enumerated' by $2^{\N}$ and have rational endpoints.  
As discussed in Remark \ref{subm}, this restriction is insignificant in our context.
As to notation, $J_{g}^{\Psi}$ is the open set $ (\Psi(g)(1), \Psi(g)(2))$ for $\Psi:C\di \Q^{2}$, 
while we say that \emph{$\Psi:C\di \R^{2}$ provides an open covering of $\R$} if $(\forall x\in \R)(\exists g\in C)(x\in J_{g}^{\Psi})$. 
We first study the following version of the Lindel\"of lemma for the real line.  
\bdefi[$\LIN^{\bs}$] For every open covering of $\R$ provided by $\Psi:C\di \Q^{2}$, there exists $\Phi:\N\di C$ such that $(\forall x\in \R)(\exists n\in \N)(x\in J_{\Phi(n)}^{\Psi})$.
\edefi
To gauge the strength of $\LIND^{\bs}$, we first prove that $\QFAC^{0,1}$ in Corollary \ref{ofmoreinterest} may be replaced by the latter.  
While this theorem also follows from Theorem \ref{corekl}, the following proof is highly illustrative.  
%$\LIN^{\bs}$. 
\begin{thm}
The system $\RCAo+\LIN^{\bs}$ proves $\WKL\asa\PIT_{o}$.  
\end{thm}
\begin{proof}
The reverse implication is immediate from (the proof of) Corollary \ref{eessje}.  
The proof of the forward implication in Corollary \ref{ofmoreinterest} makes use of $\QFAC^{0,1}$ \emph{once}, namely to conclude from $(\forall n^{0})(\exists \alpha\leq 1)(F(\alpha)>n)$ the existence of a 
sequence $\alpha_{n}$ in Cantor space such that $(\forall n^{0})(F(\alpha_{n})>n)$ in the proof of Theorem \ref{ofinterest}.  This application of $\QFAC^{0,1}$ can be replaced by $\LIN^{\bs}$ as follows: since $F$ is unbounded on Cantor space, 
$\Psi(x):=(-F(x), F(x))$ yields an open covering of $\R$, and the countable sub-covering $\Phi$ provided by $\LIN^{\bs}$ is such that $(\forall m\in \N)(\exists n\in \N)( F(\Phi(n))>m )$.  Applying $\QFAC^{0,0}$ now yields the sequence $\alpha_{n}$. 
 % and we are done.    
\end{proof}  
The previous proof goes through, but becomes a lot messier, if we assume $\Psi$ from $\LIND^{\bs}$ has $[0,1]$ or $\R$ as a domain, rather than Cantor space.   
This is the reason we have chosen the latter domain. 
As expected, we also have the following theorem.   
\begin{thm}\label{corekl}
$\RCAo+\LIN^{\bs}$ proves $\QFAC^{0,1}_{\R}$, i.e.\ for all $F:\R\di \N$, we have
\be\label{rivool}
(\forall n\in \N)(\exists x\in \R)(F(x, n)=0)\di (\exists Y^{0\di 1})(\forall n\in \N)(F(Y(n), n)=0).
\ee
\end{thm}
\begin{proof}
In case of $\neg(\exists^{2})$, all functions on the reals are continuous by \cite{kohlenbach2}*{Prop.~3.12}, and the antecedent of \eqref{rivool} then implies $(\forall n\in \N)(\exists q\in \Q)F(q, n)=0$; 
by definition, $\QFAC^{0,0}$ is included in $\RCAo$ and finishes this case.  
In case of $(\exists^{2})$, we fix $F:\R\di \N$ such that $(\forall n\in \N)(\exists x\in \R)(F(x, n)=0)$.  Now use $(\exists^{2})$ to define $\textsf{inv}(x)$ as $0$ if $x=_{\R}0$ and $1/x$ otherwise; note that:
\be\label{foiklo}
(\forall n\in \N)(\exists x\in [0,1])\big(F(x, n)\times F(\textsf{inv}(x), n)\times F(-x, n)\times F(-\textsf{inv}{(x)}, n)=0 \big).  
\ee
Thus, we may assume that  $(\forall n\in \N)(\exists x\in [0,1])(F(x, n)=0)$. Using $\exists^{2}$, define $G: C\di \N$ as follows for $f\in C$ and $w_{n}=\langle 1\dots 1\rangle$ with length $n$:
\be\label{Defg}
G( w_{n}*f):=
\begin{cases}
n+2 & \textup{ if } (\forall i\leq n)F(\mathbb{r}(\pi(f,n)(i)),i )=0 \wedge f(0)=0 \\
1 & \textup{ if there is no such $n$ }\\
\end{cases}, 
\ee
where $\mathbb{r}(x)=\sum_{i=0}^{\infty}\frac{x(i)}{2^{i}}$ and $\pi^{(1\times 0)\di 1^{*}}$ is a `decoding' function, i.e.\ $\pi(f, n)$ produces a finite sequence $w^{1^{*}}$ in Baire space on input  $f^{1}$ coding $w$ and its length $|w|=n$. 

\smallskip

Since $\exists^{2}$ can compute a binary representation of any real in the unit interval, we have $(\forall n\in \N)(\exists x\in C)F(\mathbb{r}(x), n)=0$, and $\Psi(x):=(-G(x), G(x))$ yields an open covering of $\R$.  Then $\LIN^{\bs}$ provides $\Phi^{0\di 1}$ such that the countable sub-covering $\cup_{n\in \N}(-G(\Phi(n)), G(\Phi(n)))$ still covers $\R$.  
Hence, $(\forall m^{0})(\exists n^{0})(G(\Phi(n))>m+1)$, and applying $\QFAC^{0,0}$, there is $g^{1}$ such that  $(\forall m^{0})(G(\Phi(g(m)))>m+1)$.  In the latter, the first case of $G$ from \eqref{Defg} must always hold, 
and we have  that $(\forall m^{0})(F(\mathbb{r}(\pi(\Phi(g(m)))(m)),m )=0$, as required. 
\end{proof}
\begin{cor}
The system $\ZF$ cannot prove $\LIND^{\bs}$.
\end{cor}
\begin{proof}
By the proof of \cite{kohlenbach4}*{Prop.\ 4.1}, $\QFAC^{0,1}_{\R}$ suffices to prove that for any $f:\R\di \R$ and $x\in \R$, $f$ is `epsilon-delta' continuous at $x$ if and only if $f$ is sequentially continuous at $x$.  
However, this equivalence is independent of $\ZF$ (\cite{heerlijk}). 
\end{proof}
In hindsight, the previous theorem is not \emph{that} surprising: applying $\QFAC^{1, 0}$ to the conclusion of $\LIND^{\bs}$, we obtain a functional which provides for each $x\in \R$ 
an interval $J_{g}^{\Psi}$ covering $x$, while we only assume $(\forall x\in \R)(\exists g\in C)(x\in J_{g}^{\Psi})$, i.e.\ a typical application of the axiom of choice.    
Indeed, the functional $\Phi$ from $\LIND^{\bs}$ is essential to the proof of the theorem, and it is a natural question what the status is of 
the following \emph{weaker} version which only states the \emph{existence} of a countable sub-covering, but does not provide a sequence of reals which \emph{generates} the sub-covering.  
\bdefi[$\LIN^{\bs}_{\w}$]
 For every open covering of $\R$ provided by $\Psi:C\di \Q^{2}$, there is a sequence $\cup_{n\in \N}(a_{n}, b_{n})$ covering $\R$ such that $(\forall n \in\N)(\exists x \in \R)[(a_{n}, b_{n}) = J_{x}^{\Psi} ]$
\edefi
We also study the associated version of the Heine-Borel theorem. 
\bdefi[$\HBU^{\bs}$] 
For every open covering of $[0,1]$ provided by $\Psi:C\di \Q^{2}$, there exists a finite sub-covering, i.e.\
$(\exists  y_{1}, \dots, y_{k}\in C)(\forall x\in \R)(\exists i\leq k)(x\in J_{y_{i}}^{\Psi})$.
\edefi
In contrast to its sibling, $\LIND_{\w}^{\bs}$ is provable in $\ZF$, as follows.  
\begin{thm}
The system $\Z_{2}^{\Omega}$ proves $\HBU^{\bs}$ and $\LIND_{\w}^{\bs}$, while $\Z_{2}^{\Omega}+\QFAC^{0,1}_{\R}$ proves $\LIND^{\bs}$.
\end{thm}
\begin{proof}
To prove $\HBU^{\bs}$ from $(\exists^{3})$, use the same proof as for $\HBU$ in Theorem~\ref{mooi}.  
Note that the point $y_{0}$ in the proof of the latter is such that we only need to know that is \emph{has} a covering interval, namely $y_{0}\in J_{g_{0}}^{\Psi}$ for some $g_{0}\in C$; note that this interval need not be centred at $y_{0}$.
To obtain $\LIND^{\bs}_{\w}$ from $\HBU^{\bs}$, note that the latter readily generalises to $[-N,  N]$, implying
\begin{align}
(\forall N \in \N)(\exists  a_{0},b_{0},  \dots, a_{k}, b_{k}\in \Q )\big[(\forall y\in& [-N, N])(\exists i\leq k)\big(y\in  (a_{i}, b_{i})\big)\label{murki}\\ 
&\wedge (\forall i\leq k)(\exists f\in C) ( (a_{i}, b_{i})= J_{f}^{\Psi}  )\big].\notag
\end{align}
where $\Psi:C\di \Q^{2}$ provides an open covering of $\R$; the formula in square brackets in \eqref{murki} is treated as quantifier-free by $(\exists^{3})$. 
Applying $\QFAC^{0,0}$, \eqref{murki} yields $\LIND^{\bs}_{\w}$.  Apply $\QFAC^{0,1}$ and $(\exists^{2})$ to the final formula in $\LIND^{\bs}_{\w}$ to obtain $\LIND^{\bs}$.  
\end{proof}
As it turns out, $\LIND^{\bs}_{\w}$ and $\LIND^{\bs}$ are even \emph{finitistically reducible}\footnote{Note that $\RCAo+(\kappa_{0}^{3})+\QFAC^{0,1}$ is conservative over $\WKL_{0}$ by \cite{kohlenbach2}*{Prop.\ 3.15}.  
According to Simpson in \cite{simpson2}*{IX.3.18}, the versions of the Lindel\"of lemma as in $\LIND^{\bs}$ and $\LIND_{\w}^{\bs}$ are thus \emph{reducible to finitistic mathematics in the sense of Hilbert}.} as follows.  % similar to Corollary \ref{kapaf}.  
\begin{cor}\label{pallap}
 $\RCAo+(\kappa_{0}^{3})$ proves $\LIND_{\w}^{\bs}$.  Adding $\QFAC^{0,1}$ yields $\LIND^{\bs}$.
\end{cor}
\begin{proof}
In case of $(\exists^{2})$, the theorem applies, using $[(\exists^{2})+(\kappa^{3}_{0})]\asa (\exists^{3})$.  % as in the proof of Corollary \ref{kapaf}.  
In case of $\neg(\exists^{2})$, all functions on Baire space are continuous, and the countable sub-covering is provided by listing $J_{\sigma*00\dots}^{\Psi}$ for all finite binary $\sigma$.  
\end{proof}
Before we continue, we discuss why our restriction to $\Psi:C\di \Q^{2}$ is insignificant.  
\begin{remark}\label{subm}\rm
By way of a practical argument, while we \emph{could} have formulated $\LIN^{\bs}$ using $\Psi:\R\di \R^{2}$, we already obtain $\QFAC^{0,1}_{\R}$ with the above version, i.e.\ $\Psi:C\di \Q^{2}$ `is enough', and this choice makes the above proofs easier.
On a more conceptual level, $\exists^{2}$ computes a functional converting reals in the unit interval into a binary representation, which combines nicely with our `excluded middle trick' in the proof of Theorem \ref{corekl}.  Moreover, $\RCAo+(\kappa_{0}^{3})$ seems to be the weakest system that still proves $\LIND^{\bs}_{\w}$, and this system also readily generalises $\LIND_{\w}^{\bs}$ from $\Psi:C\di \Q^{2}$ to $\Psi:C\di \R^{2}$. 
\end{remark}
As noted above, $\ZF$ proves the equivalence between the fact that $\R$ is a Lindel\"of space and the axiom of countable choice for subsets of $\R$ (\cite{heerlijk}). 
The base theory in the following theorem is significantly weaker than $\ZF$. 
\begin{cor}
The system $\RCAo$ proves $\LIND^{\bs}\asa [\QFAC^{0,1}_{\R}+ \LIND^{\bs}_{\w}]$, while $\Z_{2}^{\Omega}$ proves $\LIND^{\bs}\asa \QFAC^{0,1}_{\R}$.  
\end{cor}
The following theorem provides a nice classification of the above theorems.
\begin{cor}
The system $\RCAo$ proves $[\HBU^{\bs}+\QFAC^{0,1}_{\R}]\asa [\LIN^{\bs} +\WKL]$. 
\end{cor}
\begin{proof}
The reverse implication follows from the theorem and the equivalence between $\WKL$ and the Heine-Borel theorem for countable coverings (see \cite{simpson2}*{IV.1}). 
For the forward implication, $\neg(\exists^{2})$ implies the continuity of all functionals on Baire space, and a countable sub-covering as in $\LIN^{\bs}$ is in this case provided by 
the sequence of all finite binary sequences.  In the case of $(\exists^{2})$, note that $\HBU^{\bs}$ implies:
\be\label{whynotsaidthegruffalo}
(\forall N \in \N)(\exists  x_{0}, \dots, x_{k}\in [-N, N])\big[(\forall y\in [-N, N]\cap \Q)(\exists i\leq k)(y\in I_{x_{i}}^{\Psi}  )\big].
\ee 
Since the formula in big square brackets in \eqref{whynotsaidthegruffalo} only involves numerical quantifiers, it is decidable modulo $\exists^{2}$.  
Hence, we may apply $\QFAC^{0,1}_{\R}$ to \eqref{whynotsaidthegruffalo} to obtain $\Phi^{0\di 1^{*}}$ such that for all $N$, the finite collection of intervals $\cup_{i<|\Phi(N)|}I^{\Psi}_{\Phi(N)(i)}$ covers all rationals in $[-N, N]$. 
For any $N^{0}$, adding the intervals that cover the end-points of the intervals in this collection, one obtains a covering of $[-N, N]$, finishing this case. 
The law of excluded middle $(\exists^{2})\vee \neg(\exists^{2})$ finishes the proof.  
\end{proof}
By \cite{simpson2}*{p.\ 54, Note 1}, $\WKL_{0}\asa \Pi_{1}^{0}\textsf{-AC}_{0}$ over $\RCA_{0}$, yielding the elegant equation:
\[
[\HBU^{\bs}+\QFAC^{0,1}_{\R}]\asa [\LIN^{\bs} + \Pi_{1}^{0}\textsf{\textup{-AC}}_{0}].
\]
In conclusion, we have formulated two versions of the Lindel\"of lemma based on the Borel-Schoenflies framework; one version is provable in (a weak fragment of) $\Z_{2}^{\Omega}$, while the other one is not provable in $\ZF$.  
The latter is due to the `hidden presence of the axiom of choice' in $\LIND^{\bs}$: an open covering in the sense of the latter only tells us that $x\in \R$ is in some interval, but not which one. 
The sequence $\Phi$ however provides such an interval for $x\in \R$ by applying $\QFAC^{1,0}$ to $(\forall x\in \R)(\exists n\in \N)(x\in J_{\Phi(n)}^{\Psi})$.
In a nutshell, the Lindel\"of lemma only becomes unprovable in $\ZF$ \emph{if} we build some choice into it, something of course set theory is wont to do.  

\subsection{More on the Lindel\"of lemma}\label{kurzweil}
We show that the Lindel\"of lemma for Baire space, called $\LIND(\N^{\N})$ hereafter, when combined with $(\exists^{2})$, yields the strongest `Big Five' system $\FIVE$.
A higher-order version of this result was proved in \cite{dagsamIII}*{\S4.2.2}, namely that the existence of the Suslin functional $(S^{2})$, i.e.\ higher-order $\FIVE$, can be proved from $(\exists^{2})$ and the existence of a \emph{realiser} $\Xi^{3}$ for $\LIND(\N^{\N})$.
Note that $\LIND(\N^{\N})$ and $(\exists^{2})$ are part of the language of \emph{third-order} arithmetic. 

\smallskip

Firstly, we point out that Lindel\"of already proved that Euclidean space is \emph{hereditarily Lindel\"of} in \cite{blindeloef} around 1903.  
Now, the latter hereditary property implies that $\N^{\N}$ has the Lindel\"of property, since $\N^\N$ is homeomorphic to the irrationals in $[0,1]$ using continued fractions expansion.
Thus, for any $\Psi^{2}$, the corresponding `canonical covering' of $\N^{\N}$ is $\cup_{f\in \N^{\N}}\big[\overline{f}\Psi(f)\big]$ where $[\sigma^{0^{*}}]$ is the set of all extensions in $\N^{\N}$ of $\sigma$.  By the Lindel\"of lemma for $\N^{\N}$, there is $f_{(\cdot)}^{0\di 1}$ such that the set of $\cup_{i\in \N}[\bar f_{i} \Psi(f_i)]$ still covers $\N^{\N}$, i.e.\
\be\tag{$\LIND(\N^{\N})$}
(\forall \Psi^{2})(\exists f_{(\cdot)}^{0\di 1})(\forall g^{1})(\exists n^{0})( g \in  \big[\overline{f_{n}}\Psi(f_{n})\big] ).
\ee
We also require the following version of the Lindel\"of lemma which expresses that for a \emph{sequence} of open coverings of $\N^{\N}$, there is a \emph{sequence} of countable sub-coverings. 
\be\tag{$\LIND_{\seq}$}
(\forall \Psi_{(\cdot)}^{0\di 2})(\exists f_{(\cdot, \cdot)}^{(0\times 0)\di 1})(\forall m^{0})\big[(\forall g^{1})(\exists n^{0})\big( g \in  [\overline{f_{n,m}}\Psi_{m}(f_{n,m})] \big)\big].
\ee
Note that such `sequential' theorems are well-studied in RM, starting with \cite{simpson2}*{IV.2.12}, and can also be found in e.g.\ \cites{fuji1,fuji2,hirstseq,dork2,dork3}.  
\begin{thm}\label{theultimate}
The system $\ACAo+\LIND_{\seq}$ proves $\FIVE$.
\end{thm}
\begin{proof}
First of all, by \cite{simpson2}*{V.1.4}, any $\Sigma_{1}^{1}$-formula can be brought into the `normal form' $(\exists g^{1})(\forall n^{0})(f(\overline{g}n) =0)$, given arithmetical comprehension. 
Thus, suppose $\varphi(m)\in \Sigma_{1}^{1}$ has normal form $(\exists g^{1})(\forall n^{0})(f(\overline{g}n,m) =0)$ and define $F^2_m$ as follows: $F_{m}(g)$ is $n+1$ if $n$ is minimal such that $f(\bar gn,m) > 0$, and $0$ if there is no such $n$.
Apply $\LIND_{\seq}$ for $\Psi^{2}_{(\cdot)}=F_{(\cdot)}$ and let $f_{(\cdot, \cdot)}$ be the sequence thus obtained.  We define $X\subset\N$ as follows:
\be\label{jawell}
X:=\{ m^{0} : (\exists n^{0})(F_{m}(f_{n,m})=0)   \}, 
\ee
using $(\mu^{2})$.  We now prove $(\forall m^{0})(m\in X\asa \varphi(m))$, establishing the corollary.  If $m\in X$, then there is $g^{1}$ such that $F_{m}(g)=0$, i.e.\ $(\forall n^{0})(f(\overline{g}n,m)=0)$ by definition, and hence $\varphi(m)$.  
Now assume $\varphi(m_{0})$ for fixed $m_{0}$, i.e.\ let $g_{0}$ be such that $(\forall n^{0})(f(\overline{g_{0}}n,m_{0}) =0)$, and note that for any $m^{0}, g^{1}, h^{1}$, if $F_m(h) > 0$ and $\bar gF_m(h) = \bar hF_m(h)$, then $F_m(g) = F_m(h)$.
In particular, if $F_{m_{0}}(h) > 0$, we have  $g_{0} \not\in [\bar h F_{m_{0}}(h)]$.  Hence, if $F_{m_{0}}(f_{n,m_{0}}) > 0$ \emph{for all $n^{0}$}, $g_{0}$ is not in the covering consisting of the union of $[\overline{f_{n,m_{0}}}F_{m_{0}}(f_{n,m_{0}})]$ for all $n^{0}$, contradicting $\LIND_{\seq}$.  
Thus, we must have $(\exists n^{0})(F_{m_{0}}(f_{n,m_{0}})=0)$, implying that $m_{0}\in X$ by \eqref{jawell}. 
\end{proof}
The previous proof is inspired by the results in \cite{dagsamIII}*{\S4.2.2}.
Due to the fact that $\N \times \N^\N$ is trivially homeomorphic to $\N^\N$, $\LIND_\seq$ is derivable from (and hence equivalent to) $\LIND(\N^{\N})$, and we obtain the following result.
\begin{cor}\label{jaweltoch}
The system $\RCAo+(\exists^{2})+\LIND(\N^{\N})$ proves $\FIVE$. 
\end{cor}
The significance of the previous corollary for \emph{predicativist mathematics} is discussed in \cite{dagsamIII}*{Remark 4.16}.  
We finish this section with a folklore observation regarding $(\exists^{2})$.  By \cite{simpson2}*{I.9.3}, $\ACA_{0}$ is equivalent over $\RCA_{0}$ to the $\L_{2}$-sentence:
\be\tag{$\MCT$}
\text{\emph{An increasing sequence $x_{n}$ in $[0,1]$ has a least upper bound $x:=\sup_n x_{n}$.}}
\ee
Working in the $\L_{\omega}$-language, if we introduce real parameters in $x_{n}$ and $x$, by which they become third-order objects, the resulting `parameter-augmented' version of $\MCT$ is equivalent to $(\exists^{2})$.  
However, (real) parameters of the following kind:
\[\textstyle
g(y):= \sup_{x\in X} f(x, y) \quad\textup{(for suitable $X$ and possibly discontinuous $f$)} 
\]
are actually common in introductory/undergraduate texts (see e.g.\ \cite{bartle2}*{p.\ 330}, \cite{royfitz}*{p.\ 17}, \cite{nudyrudy}*{p.\ 56}, and  \cite{taoana2}*{p.\ 56}). 
In conclusion, allowing `real' real parameters in $\MCT$, we obtain not $\ACA_{0}$ but $(\exists^{2})$, and this `parameter-practice' 
may be found in basic mathematics.  However, $(\exists^{2})$ and $\LIND(\N^{\N})$  are quite `explosive' by Corollary~\ref{jaweltoch}, in that both are weak (in isolation) compared to $\FIVE$.  

\appendix
\section{Uniform proofs in the literature}\label{pproof}
We discuss numerous proofs of Heine's and Pincherle's theorem from the literature and show that 
these proofs actually establish the uniform versions, sometimes after minor modifications (only).  
Our motivation is to convince the reader that mathematicians like Dini, Pincherle, Lebesgue, Young, Riesz, and Bolzano were using strong axioms (like the centred theorem below) in their proofs, and the latter establish (sometimes after minor modification) highly uniform theorems.  
Another reason is that these uniform theorems were the initial motivation for this paper.  
 
\smallskip

Some of the aforementioned proofs are only discussed briefly due to their similarity to the above proofs.      
We first discuss Heine's theorem in Section \ref{heikel}, as Dini's proof of the latter (\cite{dinipi}) predates the proof of Pincherle's theorem from \cite{tepelpinch}; the latter theorem is discussed in Section \ref{forgopppp}.  
A comparison between the proofs by Dini and Pincherle suggests that Pincherle based his proof on Dini's.   
Both proofs make use of the following version of the Bolzano-Weierstrass theorem.  
\begin{quote}
If a function has a definite property infinitely often within a finite domain, then there is a point such that in any neighbourhood of this point there are infinitely many points with the property.
\end{quote}
Note that Weierstrass has indeed formulated this theorem in \cite{weihimself}*{p.\ 77}, while Pincherle mentions it in \cite{pinkersgebruiken}*{p.\ 237} (with an attribution to Weierstrass); Dini states a special case of the centred theorem in \cite{dinipi}*{\S36}.  

\smallskip

%%CHANGE
Finally, we stress the speculative nature of historical claims (say compared to mathematical ones).  
We have taken great care to accurately interpret all the mentioned proofs, but more certainty than the level of interpretation we cannot claim.

\subsection{Proofs of Heine's theorem}\label{heikel}
First of all, the proofs of Heine's  theorem in \cite{messias}*{\S4.20}, \cite{bartle2}*{p.~148}, \cite{botsko}*{Theorem 3}, \cite{gormon}*{Theorem 7}, \cite{hardy}*{V},  \cite{hobbelig}*{p.\ 239}, \cite{knapgedaan}*{p.\ 111}, \cite{langebaard}*{p.\ 35}, \cite{leaderofthepack}*{p.\ 14}, \cite{lebes1}*{p.\ 105}, \cite{mensiesson}*{p.~185}, \cite{munkies}*{p.\ 178}, \cite{protput}*{p.\ 82}, \cite{rudin}*{p.\ 91}, \cite{stillebron}*{p.\ 62}, \cite{thom2}*{Example~3, p.~474}, and \cite{younger}*{p.\ 218} are basic compactness arguments, i.e.\ they amount to little more than $\HBU_{\c}\di \UCT_{\u}$ from Theorem~\ref{roofer}. 

\smallskip

Secondly, the proof of Heine's theorem by Dini in \cite{dinipi}*{\S41} (Italian) and \cite{dinipi2}*{\S41} (German) is essentially as in Theorem \ref{heineforeal}, with one difference: Dini does not use the function from \eqref{hopla}, but introduces a modulus of continuity as follows:
\begin{quote}
the number $\eps$ should be interpreted as the supremum of all values of $\eps$ that, in reference to the point $x$, are compatible with those properties any $\eps$ should have. (see \S41 in \cite{dinipi, dinipi2})
%
% soll unter s die obere Grenze der Werthe von s ver-
%standen werden, welche in Bezug auf den Punkt x mit den
%Eigenschaften, die alle £ haben m\UTF{00FC}ssen, vereinbar sind (eine
%solche Grenze ist offenbar vorhanden. § 15). s(6, x) be-
\end{quote}
Thus, Dini's modulus modulus of continuity $\eps(x, \sigma)$ is the supremum of all $\eps'>0$ such that $(\forall x, y\in I)(|x-y|<\eps' \di |f(x)-f(y)|<\sigma)$. 
Our modulus $\eps_{0}(x, \sigma)$ from \eqref{hopla} is always below $\eps(x, \sigma)$, but does not depend on the function $f$ and hence yields \emph{uniform} Heine's theorem.  

\begin{thm}\label{heineforeal}
Any continuous $f:[a,b]\di \R$ is uniformly continuous on $[a,b]$.  
\end{thm}
\begin{proof}
For simplicity, we work over $I\equiv [0,1]$.  
Using Dini's notations, let $\eps:(I\times \R)\di \R^{+}$ be a modulus of (pointwise) continuity for $f:I\di \R$, i.e.\
\[
(\forall \sigma >_{\R}0)(\forall x, y\in I)(|x-y|<_{\R}\eps(x, \sigma)\di |f(x)-f(y)|<_{\R}\sigma ). 
\]
Without loss of generality, we may assume that $\eps(x, \sigma)<2$ for all $x\in I$.  
There are many moduli of continuity, and we need a `nice' modulus, or similar object.    
To this end, define $I_{x}^{\eps(x, \sigma)}$ as the interval $(x-\eps(x, \sigma), x+\eps(x, \sigma))$ and define
\be\label{hopla}
\eps_{0}(x, \sigma):= \sup \big\{  {\eps(y, \sigma)} : y\in I \wedge x\in  I_{y}^{\frac{1}{2}\eps(y, \sigma)}  \big\}.
%\eps_{0}(x, \sigma):= \sup \{ | I_{y}^{\eps(y, \sigma)}| : y\in I \wedge I_{x}^{\eps(x, \sigma)}\subseteq I_{y}^{\eps(y, \sigma)}  \}.\\
%\eps_{0}(x, \sigma):= \sup \{ I_{y}^{\eps(y, \sigma)} : y\in I \wedge I_{x}^{\eps(x, \sigma)}\subseteq I_{y}^{\eps(y, \sigma)}  \}.\\
%\eps_{0}(x, \sigma):= \sup \{ (a,b) : (\forall y\in I)( I_{x}^{\eps(x, \sigma)}\subseteq I_{y}^{\eps(y, \sigma)}\di (a, b)\subseteq I_{y}^{\eps(y, \sigma)})  \}.
\ee
Note that if $|x-z| < \eps_{0}(x,\sigma)/2$, then $|f(x) - f(z)| < 2\sigma$, i.e.\ $\eps_{0}$ is essentially a modulus of continuity for $f$ too.  
Now fix $\sigma>_{\R}0$ and let $\lambda_{0}$ be $\inf_{z\in I}\eps_{0}(z, \sigma/2)$.  Then there is a point $x'\in I$ such that for any neighbourhood $U$ of $x'$, no matter how small, we have $\inf_{z\in U} \eps_{0}(z, \sigma/2) = \lambda_{0}$.  
Now consider $U_{0}=I_{x'}^{\frac{1}{2}\eps(x', \sigma/2)}$ and note that $\inf_{z\in U_{0}} \eps_{0}(z, \sigma/2) = \lambda_{0}$ by definition.
However, for $z\in U_{0}$, \eqref{hopla} (for $\sigma/2$) implies that $\eps_{0}(z, \sigma/2)$ is at least ${\eps(x', \sigma/2)}$, i.e.\ $\eps_{0}(z, \sigma/2)\geq \eps(x', \sigma/2)$. 
Taking the infimum, $\lambda_{0}=\inf_{z\in U_{0}} |\eps_{0}(z, \sigma/2)| \geq \eps(x', \sigma/2)$.  Define $\eps_{1}:=\frac{1}{2}\eps(x', \sigma/2)$ and note
\[
(\forall x, y\in I)(|x-y|<_{\R}\eps_{1})\di |f(x)-f(y)|<_{\R}\sigma ),  
\]
and the uniform continuity of $f$ follows. 
\end{proof}
L\"uroth's proof of Heine's theorem \cite{grosselul} proceeds in the same way: a \emph{nice} modulus of continuity is defined, for which it is argued that 
the infimum cannot be zero anywhere in the interval, establishing uniform continuity.  With inessential modification, L\"uroth's proof also yields \emph{uniform} Heine's theorem. 

\smallskip

Incidentally, Weierstrass' proof from \cite{amaimennewekker}*{p.\ 203-204} establishes the Heine-Borel theorem (without explicit formulation) and also starts with the introduction of a nice modulus (in casu: of uniform convergence).   
A detailed motivation for this observation is in \cite{medvet}*{p. 96-97}.  
The following corollary is now immediate.  % from the proof. 
\begin{cor}
For any $\eps>_{\R}0$ and $g:(I\times \R)\di \R^{+}$, there is $\delta>_{\R}0$ such that for any $f:I\di \R$ with modulus of continuity $g$, we have
\[
(\forall x, y\in I)(|x-y|<_{\R}\delta)\di |f(x)-f(y)|<_{\R}\eps ), 
\]
%i.e.\ there is a 
\end{cor}
Thirdly, as discussed in Remark \ref{kowlk}, Pincherle mentions a variation of Pincherle's theorem in \cite{tepelpinch}*{Footnote 1} and states it is a generalisation of Heine's theorem as proved by Dini in \cite{dinipi}*{\S41}.  
As discussed in Section \ref{forgopppp}, Pincherle's proof of Pincherle's theorem \emph{with minor modification} also establishes the uniform version, and the uniform version of the variation from Remark \ref{kowlk} immediately yields \emph{uniform} Heine's theorem when applied to a modulus of continuity.  
Hence, Pincherle's proof from \cite{tepelpinch} establishes uniform Heine's theorem  \emph{with minor modification}.  

\smallskip

Fourth, Bolzano provides an incorrect proof of Heine's theorem in \cite{nogrusser}*{p.\ 575, \S6}.  However, Rusnock claims in \cite{russje}*{p.\ 113} that Bolzano's basic strategy is solid;  Rusnock 
also provides a corrected proof, which he calls \emph{a Bolzanian proof of Heine's theorem}, in \cite{russje}*{Appendix}.  The latter proof can establish uniform Heine's theorem as it is is similar in spirit to the proof of Theorem \ref{heineforeal}: one starts from a modulus of continuity, then defines a certain sequence in terms of the latter, and the cluster point of this sequence is used to define a modulus of \emph{uniform} continuity.  

\smallskip

Fifth, Lebesgue provides (what he refers to as) a `pretty proof' of Heine's theorem in \cite{lebes1}*{p.\ 105, Footnote 1} as an application of the Heine-Borel theorem for \emph{uncountable coverings}. 
Lebesgue's proof establishes $\HBU\di \UCT_{\u}^{\R}$ as follows: Lebesgue's notion of (uniform) continuity (see \cite{lebes1}*{p.\ 22}) involves a \emph{modulus} of (uniform) continuity.  
Given a modulus of continuity $g$ for $f$ on $[a, b]$, the ball $(x-g(x, \eps), x+g(x, \eps))$ is such that the oscillation of $f(x)$ is at most $\eps$.  
Hence, applying $\HBU$ to the covering $\cup_{x\in I}I^{g}_{x}$, immediately implies $\UCT_{\u}^{\R}$.  The proofs by Bromwich, Riesz, Hardy, and Young in \cites{manon, younger, hardy,bromance} amount to the same proof.  

\smallskip

Sixth, Thomae's proof (\cite{thomeke}*{p.\ 5}) of Heine's theorem is not correct, but actually suggests using \eqref{hopla}.  Indeed, for the associated canonical covering, build a sequence in which the first interval covers zero, and the next one the right end-point of the previous one, as in Thomae's proof.  The latter now yields \emph{uniform} Heine's theorem.    

\smallskip

Finally, neither Weierstrass' proof in \cite{weihimself}, or Heine's proof in \cite{keine}, or Dirichlet's proof in \cite{didi2} establish the uniform version of Heine's theorem, as far as we can see.

\subsection{Proofs of Pincherle's theorem}\label{forgopppp}
First of all, the proofs of Pincherle's theorem in \cite{bartle2}*{p.\ 149}, \cite{gormon}*{p.\ 111}, and \cite{thom2}*{p.\ 185} are basic compactness arguments, amounting to little more than the proof of $\HBU_{\c}\di \PIT_{\u}$ in Theorem \ref{mooi}.  

\smallskip

Secondly, the proof of Pincherle's theorem by Pincherle himself is essentially as follows (see \cite{tepelpinch}*{p.\ 67 for the Italian original}).    
\begin{thm}[Pincherle]\label{gemtoo}
Let $E$ be a closed, bounded subset of $\mathbb{R}^{n}$ and let $f : E \di \R$ be
locally bounded with realisers $L, r:\R\di \R^{+}$. Then $f$ is bounded on $E$.
\end{thm}
\begin{proof}
We start with a note regarding references: Pincherle motivates the crucial step in the proof in \cite{tepelpinch}*{p.\ 67} as follows: \emph{per le proposizioni generali sulle grandezze variabili}, which translates to \emph{due to general propositions on variable magnitudes}.    
Pincherle does not provide references, but it is clear from his proof that he meant the version of the Bolzano-Weierstrass theorem from the beginning of this section. 

\smallskip

Now suppose $f:E\di \R$ is locally bounded with realisers $L', r:\R\di \R^{+}$, i.e.\ for every $x\in E$ and $y\in E\cap B(x, r(x)) $, we have $|f(y)|\leq L'(x)$. 
Let $L(x)$ be the lim sup of $|f(y)|$ for $y \in E\cap B(x, r(x))$.  By assumption $L:E\di\R^+$ is always finite (and well-defined) for inputs from $E$.
Now let $L\in \R^{+}\cup\{+\infty\}$ be the lim sup of $L(x)$ for $x\in E$; we show that $L$ is a finite number.  

\smallskip

In fact, there is, due to the first paragraph, 
a point $x'\in E$
such that for any neighbourhood $U$ of $x'$, however small, the lim sup of $L(y)$ for $y\in U$ is $L$.  By locally boundeness, the lim sup of $|f(y)|$ for $y\in B(x', r(x'))$ is a finite number, namely less than $L':=L(x')$.    
By the previous, the lim sup of $|L(y)|$ for $y\in B(x', r(x')/2)$ is $L$.  But since $B(x', r(x')/2)\subset B(x', r(x')$, we have $L\leq L'$, and $L$ is indeed finite.  
\end{proof}
A minor modification of the previous proof now yields the uniform version.  
\begin{cor}
Let $E$ be a closed, bounded subset of $\mathbb{R}^{n}$ and let $f : E \di \R$ be
locally bounded with realisers $L, r:\R\di \R^{+}$. Then $|f|$ has an upper bound on $E$ that only depends on the latter.
\end{cor}
\begin{proof}
It suffices to define a suitable $L(x)$ in terms of $L'(x)$ (rather than in terms of $f(x)$).  This can be done in the same way as $\eps_{0}(x, \sigma)$ in \eqref{hopla} is defined in terms of $\eps(x, \sigma)$.  
For instance, define $L:E\di \R^{+}$ as follows:
\be\label{dorkioplk}
L(x):=\inf_{z\in E}\{ L'(z): I_{x}^{r}\subseteq I_{z}^{r}   \},
\ee
where $L', r:E\di \R^{+} $ are realisers for the local boundedness of $f$. 
\end{proof}
In conclusion, Dini \emph{almost} establishes $\UCT_{\u}^{\R}$ in \cites{dinipi, dinipi2}, while Pincherle later probably adapted Dini's proof to obtain Pincherle's theorem in \cite{tepelpinch}.  
Pincherle's proof is uniform \emph{if} we define $L(x)$ as in \eqref{dorkioplk} rather than in terms of $f$ itself, i.e.\ similar to \eqref{hopla}.  Moreover, the proof in \cite{russje}*{Appendix} 
seems to establish $\UCT_{\u}^{\R}$, and is claimed by the historian Rusnock to be a \emph{Bolzanonian proof of Heine's theorem}.  
Finally, Lebesgue, Riesz, and Young prove $\HBU\di\UCT_{\u}^{\R}$ in \cite{lebes1,manon, younger}.  

\smallskip

In a nutshell, we observe that the version of the Bolzano-Weierstrass theorem 
from the beginning of this section, as well as the Heine-Borel theorem for uncountable coverings, was (or could be) used to prove \emph{uniform} versions of Heine's and Pincherle's theorems.  
Weierstrass' more `constructive' approach as in \cites{weihimself, didi2} later became the norm however, until  the redevelopment of analysis as in e.g.\ \cite{bartle2} based on techniques from gauge integration.  
With that, both history and this paper have come full circle, which constitutes a nice ending for this section.  

\begin{ack}\rm
Our research was supported by the John Templeton Foundation via the grant \emph{a new dawn of intuitionism} with ID 60842, the Alexander von Humboldt Foundation, LMU Munich (via the Excellence Initiative and the Center for Advanced Studies of LMU), and the University of Oslo.
We express our gratitude towards these institutions. 
We thank Fernando Ferreira, Paul Rusnock, and Anil Nerode for their valuable advice.  
We also thank the anonymous referee for the helpful suggestions.  
Opinions expressed in this paper do not reflect those of the John Templeton Foundation.    
\end{ack}


\begin{bibdiv}
\begin{biblist}
%\bibselect{allkeida}
\bib{messias}{book}{
  author={Apostol, Tom M.},
  title={Mathematical analysis: a modern approach to advanced calculus},
  publisher={Addison-Wesley Publishing Company, Inc., Reading, Mass.},
  date={1957},
  pages={xii+553},
}

\bib{avi2}{article}{
  author={Avigad, Jeremy},
  author={Feferman, Solomon},
  title={G\"odel's functional \(``Dialectica''\) interpretation},
  conference={ title={Handbook of proof theory}, },
  book={ series={Stud. Logic Found. Math.}, volume={137}, },
  date={1998},
  pages={337--405},
}

\bib{bartle3}{book}{
  author={Robert {Bartle}},
  title={{The elements of real analysis.}},
  year={1976},
  publisher={John Wiley\&Sons. XV, 480 p.},
}

\bib{bartle2}{book}{
  author={Robert {Bartle} and Donald {Sherbert}},
  title={{Introduction to real analysis}},
  pages={xi + 404},
  year={2000},
  publisher={Wiley},
}

\bib{nogrusser}{book}{
  author={Bernard {Bolzano}},
  title={{The mathematical works of Bernard Bolzano. Edited by Steve Russ. Translated from the German.}},
  pages={xxx + 698},
  year={2004},
  publisher={Oxford University Press},
}

\bib{botsko}{article}{
  author={Botsko, Michael},
  title={A Unified Treatment of Various Theorems in Elementary Analysis},
  journal={Amer. Math. Monthly},
  volume={94},
  date={1987},
  number={5},
  pages={450--452},
}

\bib{berserk}{article}{
  author={Bors\'ik, J\'an},
  journal={Real Anal. Exchange},
  number={2},
  pages={339--350},
  title={Points of Continuity, Quasicontinuity, Cliquishness, and Upper and Lower Quasicontinuity},
  volume={33},
  year={2007},
}

\bib{brich}{book}{
  author={Bridges, Douglas},
  author={Richman, Fred},
  title={Varieties of constructive mathematics},
  series={London Mathematical Society Lecture Note Series},
  volume={97},
  publisher={Cambridge University Press},
  place={Cambridge},
  date={1987},
  pages={x+149},
}

\bib{bromance}{book}{
  author={Bromwich, Thomas},
  title={An introduction to the theory of infinite series},
  publisher={Macmillan Publishing Co.\ London},
  date={1908},
  pages={540},
}

\bib{opborrelen}{article}{
  author={Borel, Emile},
  title={Sur quelques points de la th\'eorie des fonctions},
  journal={Ann. Sci. \'Ecole Norm. Sup. (3)},
  volume={12},
  date={1895},
  pages={9--55},
}

\bib{chorda}{article}{
  author={Chorlay, Renaud},
  title={``Local-global'': the first twenty years},
  journal={Arch. Hist. Exact Sci.},
  volume={65},
  date={2011},
  number={1},
  pages={1--66},
}

\bib{cousin1}{article}{
  author={Cousin, Pierre},
  title={Sur les fonctions de $n$ variables complexes},
  journal={Acta Math.},
  volume={19},
  date={1895},
  number={1},
  pages={1--61},
}

\bib{dinipi}{book}{
  author={U. {Dini}},
  title={{Fondamenti per la teorica delle funzioni di variabili reali}},
  year={1878},
  publisher={{Nistri, Pisa}},
}

\bib{dinipi2}{book}{
  author={U. {Dini}},
  title={{Grundlagen f\"ur eine Theorie der Functionen einer ver\"anderlichen reelen Gr\"osse}},
  year={1892},
  publisher={{Leipzig, B.G. Teubner}},
}

\bib{didi3}{book}{
  author={Dirichlet, Lejeune P.~G.},
  title={\"Uber die Darstellung ganz willk\"urlicher Funktionen durch Sinus- und Cosinusreihen},
  year={1837},
  publisher={Repertorium der physik, von H.W. Dove und L. Moser, bd. 1},
}

\bib{didi2}{book}{
  author={Dirichlet, Lejeune P.~G.},
  title={{Vorlesungen \"uber die Lehre von den einfachen und mehrfachen bestimmten Integralen}},
  year={1904},
  publisher={{Vieweg \& Sohn. XXIII u. 476 S. $8^\circ $.}},
}

\bib{didi1}{article}{
  author={Dirichlet, Lejeune P.~G.},
  title={Sur la convergence des s\'eries trigonom\'etriques qui servent \`a repr\'esenter une fonction arbitraire entre des limites donn\'ees},
  journal={arXiv},
  year={2008},
  note={\url {https://arxiv.org/abs/0806.1294}},
}

\bib{dork2}{article}{
  author={Dorais, Fran\d {c}ois G.},
  title={Classical consequences of continuous choice principles from intuitionistic analysis},
  journal={Notre Dame J. Form. Log.},
  volume={55},
  date={2014},
  number={1},
  pages={25--39},
}

\bib{dork3}{article}{
  author={Dorais, Fran\d {c}ois G.},
  author={Dzhafarov, Damir D.},
  author={Hirst, Jeffry L.},
  author={Mileti, Joseph R.},
  author={Shafer, Paul},
  title={On uniform relationships between combinatorial problems},
  journal={Trans. Amer. Math. Soc.},
  volume={368},
  date={2016},
  number={2},
  pages={1321--1359},
}

\bib{damirzoo}{misc}{
  author={Dzhafarov, Damir D.},
  title={Reverse Mathematics Zoo},
  note={\url {http://rmzoo.uconn.edu/}},
}

\bib{fourierlachaud}{book}{
  author={Fourier, Joseph},
  title={Th\'{e}orie analytique de la chaleur},
  language={French},
  note={Reprint of the 1822 original},
  publisher={\'{E}ditions Jacques Gabay, Paris},
  date={1988},
  pages={xxii+644},
}

\bib{fried}{article}{
  author={Friedman, Harvey},
  title={Some systems of second order arithmetic and their use},
  conference={ title={Proceedings of the International Congress of Mathematicians (Vancouver, B.\ C., 1974), Vol.\ 1}, },
  book={ },
  date={1975},
  pages={235--242},
}

\bib{fried2}{article}{
  author={Friedman, Harvey},
  title={ Systems of second order arithmetic with restricted induction, I \& II (Abstracts) },
  journal={Journal of Symbolic Logic},
  volume={41},
  date={1976},
  pages={557--559},
}

\bib{voller}{article}{
  author={Fuller, Richard V.},
  title={{Relations among continuous and various non-continuous functions.}},
  journal={{Pac. J. Math.}},
  volume={25},
  pages={495--509},
  year={1968},
}

\bib{fuji2}{article}{
  author={Fujiwara, Makoto},
  author={Higuchi, Kojiro},
  author={Kihara, Takayuki},
  title={On the strength of marriage theorems and uniformity},
  journal={MLQ Math. Log. Q.},
  volume={60},
  date={2014},
  number={3},
}

\bib{fuji1}{article}{
  author={Fujiwara, Makoto},
  author={Yokoyama, Keita},
  title={A note on the sequential version of $\Pi ^1_2$ statements},
  conference={ title={The nature of computation}, },
  book={ series={Lecture Notes in Comput. Sci.}, volume={7921}, publisher={Springer, Heidelberg}, },
  date={2013},
  pages={171--180},
}

\bib{supergandy}{article}{
  author={Gandy, R. O.},
  title={General recursive functionals of finite type and hierarchies of functions},
  journal={Ann. Fac. Sci. Univ. Clermont-Ferrand No.},
  volume={35},
  date={1967},
  pages={5--24},
}

\bib{gordon3}{book}{
  author={Gordon, E.},
  author={{Kusraev}, A.},
  author={Kutateladze, S.},
  title={Infinitesimal analysis},
  pages={xiv + 422},
  year={2002},
  publisher={Dordrecht: Kluwer Academic Publishers},
}

\bib{gormon}{article}{
  author={Gordon, Russell A.},
  title={The use of tagged partitions in elementary real analysis},
  journal={Amer. Math. Monthly},
  volume={105},
  date={1998},
  number={2},
  pages={107--117},
}

\bib{gauwdief}{article}{
  author={Gowers, Timothy},
  title={When is proof by contradiction necessary?},
  journal={Blog: \url {https://gowers.wordpress.com/2010/03/28/when-is-proof-by-contradiction-necessary/}},
  date={2010},
}

\bib{withgusto}{article}{
  author={Giusto, Mariagnese},
  author={Simpson, Stephen G.},
  title={Located sets and reverse mathematics},
  journal={J. Symbolic Logic},
  volume={65},
  date={2000},
  number={3},
  pages={1451--1480},
}

\bib{hankelijkheid}{book}{
  author={H. {Hankel}},
  title={{Untersuchungen \"uber die unendlich oft oscillirenden und unstetigen Functionen.}},
  journal={{Math. Ann.}},
  volume={20},
  pages={63--112},
  year={1882},
  publisher={Springer},
}

\bib{hardy}{book}{
  author={Hardy, G. H.},
  title={A course of pure mathematics},
  note={2nd ed},
  publisher={Cambridge, at the University Press},
  date={1914},
}

\bib{keine}{article}{
  author={Heine, E.},
  title={Die Elemente der Functionenlehre},
  language={German},
  journal={J. Reine Angew. Math.},
  volume={74},
  date={1872},
  pages={172--188},
}

\bib{heerlijk}{article}{
  author={Horst {Herrlich}},
  title={{Choice principles in elementary topology and analysis.}},
  journal={{Commentat. Math. Univ. Carol.}},
  volume={38},
  number={3},
  pages={545--552},
  year={1997},
}

\bib{wildehilde}{article}{
  author={Hildebrandt, T. H.},
  title={The Borel theorem and its generalizations},
  journal={Bull. Amer. Math. Soc.},
  volume={32},
  date={1926},
  number={5},
  pages={423--474},
}

\bib{hirstseq}{article}{
  author={Hirst, Jeffry L.},
  author={Mummert, Carl},
  title={Reverse mathematics and uniformity in proofs without excluded middle},
  journal={Notre Dame J. Form. Log.},
  volume={52},
  date={2011},
  number={2},
  pages={149--162},
}

\bib{hobbelig}{book}{
  author={Hobson, Ernest William},
  title={The theory of functions of a real variable and the theory of Fourier's series},
  publisher={Cambridge: University Press},
  date={1907},
  pages={pp.\ 772},
}

\bib{hunterphd}{book}{
  author={Hunter, James},
  title={Higher-order reverse topology},
  note={Thesis (Ph.D.)--The University of Wisconsin - Madison},
  publisher={ProQuest LLC, Ann Arbor, MI},
  date={2008},
  pages={97},
}

\bib{ooskelly}{book}{
  author={Kelley, John L.},
  title={General topology},
  note={Reprint of the 1955 edition; Graduate Texts in Mathematics, No. 27},
  publisher={Springer-Verlag},
  date={1975},
  pages={xiv+298},
}

\bib{kermend}{article}{
  author={Keremedis, Kyriakos},
  title={Disasters in topology without the axiom of choice},
  journal={Arch. Math. Logic},
  volume={40},
  date={2001},
  number={8},
}

\bib{kleine}{book}{
  author={Kleiner, Israel},
  title={Excursions in the history of mathematics},
  publisher={Birkh\"auser/Springer, New York},
  date={2012},
}

\bib{knapgedaan}{book}{
  author={Knapp, Anthony W.},
  title={Basic real analysis},
  publisher={Birkh\"{a}user},
  date={2005},
  pages={xxiv+653},
}

\bib{kohlenbach3}{book}{
  author={Kohlenbach, Ulrich},
  title={Applied proof theory: proof interpretations and their use in mathematics},
  series={Springer Monographs in Mathematics},
  publisher={Springer-Verlag},
  place={Berlin},
  date={2008},
  pages={xx+532},
}

\bib{kohlenbach2}{article}{
  author={Kohlenbach, Ulrich},
  title={Higher order reverse mathematics},
  conference={ title={Reverse mathematics 2001}, },
  book={ series={Lect. Notes Log.}, volume={21}, publisher={ASL}, },
  date={2005},
  pages={281--295},
}

\bib{kohlenbach4}{article}{
  author={Kohlenbach, Ulrich},
  title={Foundational and mathematical uses of higher types},
  conference={ title={Reflections on the foundations of mathematics (Stanford, CA, 1998)}, },
  book={ series={Lect. Notes Log.}, volume={15}, publisher={ASL}, },
  date={2002},
  pages={92--116},
}

\bib{kohabil}{book}{
  author={Kohlenbach, Ulrich},
  title={Real Growth in Standard Parts of Analysis},
  publisher={Habilitationsschrift, J.W. Goethe Universit\"at Frankfurt, pp. xv+166},
  date={1995},
}

\bib{langebaard}{book}{
  author={Lang, Serge},
  title={Real analysis},
  edition={2},
  publisher={Addison-Wesley},
  date={1983},
  pages={xv+533},
}

\bib{leaderofthepack}{book}{
  author={Leader, Solomon},
  title={The Kurzweil-Henstock integral and its differentials},
  series={Monographs and Textbooks in Pure and Applied Mathematics},
  volume={242},
  publisher={Marcel Dekker, Inc., New York},
  date={2001},
  pages={viii+355},
}

\bib{lebes1}{book}{
  author={Henri Leon {Lebesgue}},
  title={{Le\c cons sur l'int\'egration et la recherche des fonctions primitives profess\'ees au Coll\`ege de France.}},
  note={Reprint of the 1904 ed.},
  pages={vii + 136},
  year={2009},
  publisher={Cambridge University Press},
}

\bib{lola}{article}{
  author={M.B. {Lignola} and J. {Morgan}},
  title={{Topological existence and stability for min sup problems.}},
  journal={{J. Math. Anal. Appl.}},
  volume={151},
  number={1},
  pages={164--180},
  year={1990},
}

\bib{blindeloef}{article}{
  author={Lindel\"of, Ernst},
  title={Sur Quelques Points De La Th\'eorie Des Ensembles},
  journal={Comptes Rendus},
  date={1903},
  pages={697--700},
}

\bib{loba}{book}{
  author={Lobachevsky, N. I.},
  title={On the vanishing of trigonometric series (1834)},
  year={1951},
  publisher={Works [Russian], T. 5, Gostekhizdat, Moscow-Leningrad},
  pages={31--80},
}

\bib{longmann}{book}{
  author={Longley, John},
  author={Normann, Dag},
  title={Higher-order Computability},
  year={2015},
  publisher={Springer},
  series={Theory and Applications of Computability},
}

\bib{grosselul}{article}{
  author={J. {L\"uroth}},
  title={{Bemerkung \"uber gleichm\"assige Stetigkeit.}},
  journal={{Math. Ann.}},
  volume={5},
  pages={319--320},
  year={1872},
  publisher={Springer, Berlin/Heidelberg},
}

\bib{mensiesson}{book}{
  author={Mendelson, Bert},
  title={Introduction to topology},
  publisher={Allyn and Bacon},
  date={1962},
  pages={ix+217},
}

\bib{medvet}{book}{
  author={Medvedev, Fyodor A.},
  title={Scenes from the history of real functions},
  series={Science Networks. Historical Studies},
  volume={7},
  publisher={Birkh\"auser Verlag, Basel},
  date={1991},
  pages={265},
}

\bib{migda}{book}{
  editor={Athanasios {Migdalas} and Panos M. {Pardalos} and Peter {V\"arbrand}},
  title={{Multilevel optimization: algorithms and applications.}},
  isbn={0-7923-4693-9},
  pages={xxii + 384},
  year={1998},
  publisher={Dordrecht: Kluwer Academic Publishers},
}

\bib{mizera}{article}{
  author={Mizera, I.},
  title={A remark on existence of statistical functionals},
  journal={Kybernetika},
  volume={31},
  date={1995},
  pages={315--319},
}

\bib{montahue}{article}{
  author={Montalb{\'a}n, A.},
  title={Open questions in reverse mathematics},
  journal={Bull. Symb. Logic},
  volume={17},
  date={2011},
  pages={431--454},
}

\bib{mullingitover}{book}{
  author={Muldowney, P.},
  title={A general theory of integration in function spaces, including Wiener and Feynman integration},
  series={Pitman Research Notes in Mathematics Series},
  volume={153},
  publisher={Longman Scientific \& Technical, Harlow; John Wiley \& Sons, Inc., New York},
  date={1987},
  pages={viii+115},
}

\bib{munkies}{book}{
  author={Munkres, James R.},
  title={Topology},
  publisher={Prentice-Hall},
  date={2000, 2nd edition},
  pages={xvi+537},
}

\bib{mummy}{article}{
  author={Mummert, Carl},
  author={Simpson, Stephen G.},
  title={Reverse mathematics and $\Pi _2^1$ comprehension},
  journal={Bull. Symbolic Logic},
  volume={11},
  date={2005},
  number={4},
  pages={526--533},
}

\bib{noiri}{article}{
  author={Noiri, Takashi},
  title={Sequentially subcontinuous functions},
  year={1975},
  journal={Accad. Naz. dei Lincei},
  volume={58},
  pages={370--373},
}

\bib{noortje}{book}{
  author={Normann, Dag},
  title={Recursion on the countable functionals},
  series={LNM 811},
  volume={811},
  publisher={Springer},
  date={1980},
  pages={viii+191},
}

\bib{dagcie18}{article}{
  author={Normann, Dag},
  title={Functionals of Type 3 as Realisers of Classical Theorems in Analysis},
  year={2018},
  journal={Proceedings of CiE18, Lecture Notes in Computer Science 10936},
  pages={318--327},
}

\bib{dagsam}{article}{
   author={Normann, Dag},
   author={Sanders, Sam}
   title={Nonstandard Analysis, Computability Theory, and their connections},
   journal={Journal of Symbolic Logic},
   volume={84}, 
   number={4},
   pages={1422-1465},
   date={2019},
}



\bib{dagsamII}{article}{
   author={Normann, Dag},
   author={Sanders, Sam},
   title={The strength of compactness in computability theory and
   nonstandard analysis},
   journal={Ann. Pure Appl. Logic},
   volume={170},
   date={2019},
   number={11},
%   pages={102710, 42},
%   issn={0168-0072},
%   review={\MR{4012249}},
%   doi={10.1016/j.apal.2019.05.007},
}


\bib{dagsamIII}{article}{
  author={Normann, Dag},
  author={Sanders, Sam},
  title={On the mathematical and foundational significance of the uncountable},
  journal={Journal of Mathematical Logic, doi: \url {10.1142/S0219061319500016}},
  date={2018},
}

\bib{dagsamVI}{article}{
  author={Normann, Dag},
  author={Sanders, Sam},
  title={Representations in measure theory: between a non-computable rock and a hard to prove place},
  journal={Submitted, arXiv: \url {https://arxiv.org/abs/1902.02756}},
  date={2019},
}

\bib{dagsamVII}{article}{
  author={Normann, Dag},
  author={Sanders, Sam},
  title={Open sets in Reverse Mathematics and computability theory},
  journal={Submitted, arXiv: \url {https://arxiv.org/abs/1910.02489}},
  date={2019},
  pages={pp.\ 30},
}

\bib{systemenouveau}{article}{
  author={Novotn\'{y}, Branislav},
  title={On subcontinuity},
  journal={Real Anal. Exchange},
  volume={31},
  date={2005/06},
  number={2},
  pages={535--545},
}

\bib{pinkersgebruiken}{article}{
  author={Pincherle, Salvatore},
  title={{Saggio di una introduzione alla teoria delle funzioni analitiche secondo i principii del Prof. C. Weierstrass.}},
  journal={{Giornale di Matematiche}},
  volume={18},
  pages={178--254},
  year={1880},
}

\bib{werkskes}{article}{
  author={Pincherle, Salvatore},
  title={Notice sur les travaux},
  journal={Acta Math.},
  volume={46},
  date={1925},
  number={3-4},
  pages={341--362},
}

\bib{tepelpinch}{article}{
  author={Pincherle, Salvatore},
  title={Sopra alcuni sviluppi in serie per funzioni analitiche (1882)},
  journal={Opere Scelte, I, Roma},
  date={1954},
  pages={64--91},
}

\bib{protput}{article}{
  author={M.H. {Protter} and C.B. {Morrey}},
  title={{A first course in real analysis.}},
  journal={{Undergraduate Texts Math.}},
  year={1977},
  publisher={Springer},
}

\bib{manon}{article}{
  author={Riesz, F.},
  title={Sur un th\'eor\`eme de M. Borel},
  journal={Comptes rendus de l'Acad\'emie des Sciences, Paris, Gauthier-Villars},
  volume={140},
  date={1905},
  pages={224--226},
}

\bib{rudin}{book}{
  author={Rudin, Walter},
  title={Principles of mathematical analysis},
  edition={3},
  note={International Series in Pure and Applied Mathematics},
  publisher={McGraw-Hill},
  date={1976},
  pages={x+342},
}

\bib{russje}{book}{
  author={Paul {Rusnock}},
  title={Bolzano’s contributions to real analysis},
  year={2003},
  publisher={Academia Verlag, Beitr\"age zur Bolzano-Forschung, Band 16, p.\ 99-116},
}

\bib{ruskesnokken}{article}{
  author={Paul {Rusnock} and Angus {Kerr-Lawson}},
  title={{Bolzano and uniform continuity.}},
  journal={{Hist. Math.}},
  volume={32},
  number={3},
  pages={303--311},
  year={2005},
}

\bib{roykes}{article}{
  author={Roy, M.},
  author={Wingler, E.},
  title={Locally bounded functions},
  year={1998},
  journal={Real Analysis Exchange},
  pages={251--258},
}

\bib{royfitz}{book}{
  author={Royden, H. L.},
  author={Fitzpatrick, H. L.},
  title={Real analysis},
  edition={4},
  publisher={Pearson Education},
  date={2010},
  pages={vi+505},
}

\bib{nudyrudy}{book}{
  author={Rudin, Walter},
  title={Real and complex analysis},
  edition={3},
  publisher={McGraw-Hill},
  date={1987},
  pages={xiv+416},
}

\bib{yamayamaharehare}{article}{
  author={Sakamoto, Nobuyuki},
  author={Yamazaki, Takeshi},
  title={Uniform versions of some axioms of second order arithmetic},
  journal={MLQ Math. Log. Q.},
  volume={50},
  date={2004},
  number={6},
  pages={587--593},
}

\bib{Sacks.high}{book}{
  author={Sacks, Gerald E.},
  title={Higher recursion theory},
  series={Perspectives in Mathematical Logic},
  publisher={Springer},
  date={1990},
  pages={xvi+344},
}

\bib{samGH}{article}{
  author={Sanders, Sam},
  title={The Gandy-Hyland functional and a computational aspect of Nonstandard Analysis},
  year={2018},
  journal={Computability},
  volume={7},
  pages={7-43},
}

\bib{sahotop}{article}{
  author={Sanders, Sam},
  title={Reverse Mathematics of topology: dimension, paracompactness, and splittings},
  year={2018},
  journal={Submitted, arXiv: \url {https://arxiv.org/abs/1808.08785}},
  pages={pp.\ 17},
}

\bib{samsplit}{article}{
  author={Sanders, Sam},
  title={Splittings and disjunctions in Reverse Mathematics},
  year={2018},
  journal={To appear in NDJFL, arXiv: \url {https://arxiv.org/abs/1805.11342}},
  pages={pp.\ 18},
}

\bib{samcie19}{article}{
author={Sanders, Sam},
title={Nets and Reverse Mathematics: initial results},
year={2019},
journal={Proceedings of CiE19, Lecture Notes in Computer Science 11558, Springer},
pages={pp.\ 12},
}

\bib{samwollic19}{article}{
author={Sanders, Sam},
title={Reverse Mathematics and computability theory of domain theory},
year={2019},
journal={Proceedings of WoLLIC19, Lecture Notes in Computer Science 11541, Springer},
pages={pp.\ 20},
}


\bib{samnet}{article}{
  author={Sanders, Sam},
  title={Nets and Reverse Mathematics: a pilot study},
  year={2019},
  journal={To appear in \emph {Computability}, \url {arxiv.org/abs/1905.04058}},
  pages={pp.\ 30},
}

\bib{samph}{article}{
  author={Sanders, Sam},
  title={Plato and the foundations of mathematics},
  year={2019},
  journal={Submitted, arxiv: \url {https://arxiv.org/abs/1908.05676}},
  pages={pp.\ 40},
}

\bib{schoen2}{book}{
  author={Shoenflies, Arthur},
  title={Die Entwickelung der Lehre von den Punktmannigfaltigkeiten},
  year={1900},
  publisher={Jahresbericht der deutschen Mathematiker-Vereinigung, vol 8,b Leipzig: B.G. Teubner},
}

\bib{simpson2}{book}{
  author={Simpson, Stephen G.},
  title={Subsystems of second order arithmetic},
  series={Perspectives in Logic},
  edition={2},
  publisher={CUP},
  date={2009},
  pages={xvi+444},
}

\bib{stillebron}{book}{
  author={Stillwell, John},
  title={Reverse mathematics, proofs from the inside out},
  pages={xiii + 182},
  year={2018},
  publisher={Princeton Univ.\ Press},
}

\bib{shorecomp}{article}{
  author={Shore, Richard A.},
  title={Reverse mathematics, countable and uncountable},
  conference={ title={Effective mathematics of the uncountable}, },
  book={ series={Lect. Notes Log.}, volume={41}, publisher={Assoc. Symbol. Logic, La Jolla, CA}, },
  date={2013},
  pages={150--163},
}

\bib{taokejes}{collection}{
  author={Tao, {Terence}},
  title={{Compactness and Compactification}},
  editor={Gowers, Timothy},
  pages={167--169},
  year={2008},
  publisher={The Princeton Companion to Mathematics, Princeton University Press},
}

\bib{taoana2}{book}{
  author={Tao, Terence},
  title={Analysis. II},
  series={Texts and Readings in Mathematics},
  volume={38},
  edition={3},
  publisher={Springer},
  date={2014},
  pages={xvi+218},
}

\bib{thomeke}{book}{
  author={Thomae, Carl J.T.},
  title={Einleitung in die Theorie der bestimmten Integrale},
  publisher={Halle a.S. : Louis Nebert},
  date={1875},
  pages={pp.\ 48},
}

\bib{thom2}{book}{
  author={Thomson, B.},
  author={Bruckner, J.},
  author={Bruckner, A.},
  title={Elementary real analysis},
  publisher={Prentice Hall},
  date={2001},
  pages={pp.\ 740},
}

\bib{tournedous}{book}{
  author={Tourlakis, George},
  title={Lectures in logic and set theory. Vol. 2},
  series={Cambridge Studies in Advanced Mathematics},
  volume={83},
  note={Set theory},
  publisher={Cambridge University Press},
  date={2003},
  pages={xvi+575},
}

\bib{troeleke1}{book}{
  author={Troelstra, Anne Sjerp},
  author={van Dalen, Dirk},
  title={Constructivism in mathematics. Vol. I},
  series={Studies in Logic and the Foundations of Mathematics},
  volume={121},
  publisher={North-Holland},
  date={1988},
  pages={xx+342+XIV},
}

\bib{amaimennewekker}{book}{
  author={Weierstrass , Karl},
  title={Ausgew\"ahlte Kapitel aus der Funktionenlehre},
  series={Teubner-Archiv zur Mathematik},
  volume={9},
  publisher={BSB B. G. Teubner Verlagsgesellschaft, Leipzig},
  date={1988},
  pages={272},
}

\bib{weihimself}{book}{
  author={Weierstrass, Karl},
  title={{Einleiting in die Theorie der analytischen Funktionen}},
  year={1986},
  publisher={{Schriftenr. Math. Inst. Univ. M\"unster, 2. Ser. 38, 108 S.}},
}

\bib{younger}{book}{
  author={Young, W. H.},
  author={Young, G. C.},
  title={The Theory of Sets of Points},
  publisher={Cambridge University Press},
  date={1906},
  pages={v+311},
}


\end{biblist}
\end{bibdiv}
\bye